\documentclass{article}

\usepackage[utf8]{inputenc}
\usepackage[a4paper, top=2.2cm, bottom=2.5cm, left=2cm, right=2cm]
{geometry}
\usepackage{amsmath,amssymb,amsfonts,bm,amscd,mathtools}
\usepackage{graphicx}
\usepackage{caption, subcaption}
\usepackage{comment}
\usepackage{cite}
\usepackage{float}
\usepackage{appendix}
\usepackage{physics}
\usepackage{tcolorbox}
\usepackage{tikz}
\usetikzlibrary{cd}
\usepackage{array}
\usepackage{xcolor}
\usepackage{pst-node}
\usepackage{tikz-cd} 
\usepackage{amsthm}
\usepackage{amsmath}
\usepackage{authblk}

\usepackage{
	amsmath,
 	amsfonts,
  	amssymb,
 	amsthm,
        datetime,
	}

\usepackage{tikz-cd}
\usepackage{tikz}

\tikzset{Rightarrow/.style={double equal sign distance,>={Implies},->},
triple/.style={-,preaction={draw,Rightarrow}}}

\usepackage{rotating}

\usetikzlibrary{decorations.pathmorphing, decorations.pathreplacing,decorations.markings}
 \usetikzlibrary{backgrounds,}
 \usetikzlibrary{shapes}

\usepackage{color}

\usepackage{stmaryrd}
\usepackage[cal=boondoxo]{mathalfa}
\usepackage{mathtools}
\usepackage{bbm}

\usepackage{hyperref}
\usepackage[capitalise]{cleveref}

\usepackage{a4wide}

\usepackage{mathabx}
\usepackage{mathbbol}
\usepackage{enumerate}
\usepackage{comment}
\usepackage{scalerel}

\newlength\bshft
\def\fakebold#1{\ThisStyle{\ooalign{$\SavedStyle#1$\cr%
  \kern-\bshft$\SavedStyle#1$\cr%
  \kern\bshft$\SavedStyle#1$}}}

\newcommand{\Z}{{\mathbb Z}}
\newcommand{\R}{{\mathbb R}}
\newcommand{\C}{{\mathbb C}}
\newcommand{\Q}{{\mathbb Q}}

\newcommand{\N}{{\mathbb N}}

\def\be{\begin{equation}}
\def\ee{\end{equation}}
\def\bea{\begin{align}}
\def\eea{\end{align}}
\newcommand{\Tb}{{\boldsymbol{T}}}
\newcommand{\Pb}{{\boldsymbol{P}}}
\newcommand{\Sb}{{\boldsymbol{S}}}
\newcommand{\Cb}{{\boldsymbol{C}}}
\newcommand{\Kb}{{\boldsymbol{K}}}
\newcommand{\Rb}{{\boldsymbol{R}}}
\newcommand{\kk}{{\mathbb{k}}}
\newcommand{\Cbb}{{\fakebold{\mathbb{C}}}}

\newcommand{\Kbb}{{\fakebold{\mathbb{K}}}}

\newcommand{\Bb}{{\fakebold{\mathbb{B}}}}
\newcommand{\cupp}{\check{\mathcal{C}}_\square}

\newcommand{\jj}{{\boldsymbol{j}}}

\tikzset{
  mid<-/.style={
    postaction={
      decorate,
      decoration={
        markings,
        mark=at position 0.5 with {\arrow{<}}
      }
    }
  },
  mid->/.style={
    postaction={
      decorate,
      decoration={
        markings,
        mark=at position 0.5 with {\arrow{>}}
      }
    }
  }
}

\tikzset{
  upper<-/.style={
    postaction={
      decorate,
      decoration={
        markings,
        mark=at position 0.75 with {\arrow{<}}
      }
    }
  },
  upper->/.style={
    postaction={
      decorate,
      decoration={
        markings,
        mark=at position 0.75 with {\arrow{>}}
      }
    }
  }
}

\tikzset{
  lower<-/.style={
    postaction={
      decorate,
      decoration={
        markings,
        mark=at position 0.25 with {\arrow{<}}
      }
    }
  },
  lower->/.style={
    postaction={
      decorate,
      decoration={
        markings,
        mark=at position 0.25 with {\arrow{>}}
      }
    }
  }
}

\theoremstyle{plain}
\newtheorem{thm}{Theorem}[section]
\newtheorem{cor}[thm]{Corollary}
\newtheorem{lem}[thm]{Lemma}
\newtheorem{prop}[thm]{Proposition}
\newtheorem{conj}[thm]{Conjecture}

\theoremstyle{definition}
\newtheorem{rem}[thm]{Remark}

\newtheorem{ex}[thm]{Example}

\theoremstyle{definition}
\newtheorem{defn}[thm]{Definition}

\definecolor{myblue}{rgb}{0,.5,1}
\definecolor{mypurple}{rgb}{.75,0,.75}
\definecolor{mygreen}{rgb}{.3,.75,.1}

\definecolor{vcolor}{rgb}{0,.5,1}
\definecolor{pcolor}{rgb}{.75,0,.75}

\newcommand{\tikzdiagh}[2][]{\tikz[#1,anchor= center,very thick,baseline={([yshift=1ex+#2]current bounding box.center)}]}

\newcommand{\tikzdiag}[1][]{\tikzdiagh[#1]{-1.5ex}}
\tikzstyle{tikzdot}=[fill, circle, inner sep=2pt]
\tikzstyle{tikzstar}=[star,star points=6,star point ratio=3,fill, inner sep=1pt]

\newcommand{\tikzbrace}[5][]{\draw[#1,decoration={brace,mirror,raise=-10pt},decorate]  (#2-.1,#4 -.35) -- node {#5} (#3+.1,#4-.35)}
\newcommand{\tikzbraceop}[5][]{\draw[#1,decoration={brace,raise=-10pt},decorate]  (#2-.1,#4 +.35) -- node {#5} (#3+.1,#4+.35)}

\tikzstyle{bund}=[black, double=white, double distance=3pt]
\tikzstyle{stdhl}=[red,double=red!30,double distance=1pt,line width=.8pt]
\tikzstyle{rstdhl}=[white,double=red!80,double distance=2pt]
\tikzstyle{bstdhl}=[white,double=blue!80,double distance=2pt]
\tikzstyle{vstdhl}=[vcolor,double=vcolor!30,double distance=1pt,line width=1.6pt]
\tikzstyle{pstdhl}=[mypurple,double=mypurple!30,double distance=1pt,line width=1.6pt]
\tikzstyle{nail}=[draw,color=black,fill=white!30,circle,inner sep=2pt]

\title{On Diagrammatic Categorification of Verma Modules I: Braiding }
\author{Pedro Guicardi}
\affil{California Institute of Technology, Pasadena, CA 91125, USA}
\date{\vspace{-1cm}}
\begin{document}
\maketitle

\begin{abstract}
In this paper, we study the extensions of KLRW algebras to tensor products of Verma module representations of $\mathfrak{sl}_2$. Our motivation is to construct a theory of Khovanov homology for knot complements in $S^3$ (and which also categorifies the Gukov-Manolescu two-variable series for knot complements), which will be done in the second part of this work. We construct the categorification of R-matrices for Verma modules as functors given by derived tensor products with diagrammatic bimodules and explicitly compute their projective resolutions. We also prove these braiding functors induce an action of the braid group on the relevant categories. Then, we describe how to incorporate strands in finite-dimensional representations of $\mathfrak{sl}_2$, thereby establishing functors that serve as the Khovanov homology on a braid complement. In the case of the unknot, this gives knot homologies in $S^1\times D^2$, which we compare to Annular Khovanov Homology through several examples and show they are very closely related, conjecturing they are of the same dimension. We conclude with a proposal for the categorification of the cups and caps of Verma module colored strands, which we build upon in the next paper.

\end{abstract}

\tableofcontents
\newpage

\section{Introduction}\label{sec:intro}
Since the advent of Khovanov Homology \cite{Kh}, categorification of quantum knot invariants has been an incredibly fruitful object of study in quantum topology. To this end, the venture of higher representation theory and 2-analogues was developed and led to the categorification of quantum groups of simple Lie algebras, $U_q(\mathfrak{g})$, into 2-categories $\mathcal{U}_q(\mathfrak{g})$. Some pioneering work on this includes \cite{CR}, \cite{Rou}, \cite{Lau}, \cite{KL}. Through this program, Webster (\cite{Web}) algebraically realized quantum knot homologies as well as the corresponding categorification of tangle invariants in the language of ``diagrammatic categorification". That is, the quantum knot polynomials associated to quantum groups (\cite{RT}) are defined diagrammatically through the Reshetikhin-Turaev dictionary, 
\begin{enumerate}
    \item A horizontal `slice' of a colored link diagram $\longrightarrow$ tensor product of irreducible, finite-dimensional representations of $U_q(\mathfrak{g})$, $V^{\mu_1}\otimes V^{\mu_2}\otimes \cdots \otimes V^{\mu_r}$
    \item A crossing of colored strands $\longrightarrow$ a homomorphism $R:V^{\mu_i} \otimes V^{\mu_j} \rightarrow V^{\mu_j} \otimes V^{\mu_i}$, referred to as the R-matrix
    \item Cups and caps of the link diagram $\longrightarrow$ (co)evaluation and (co)trace homomorphisms ($V^*\otimes V \rightarrow \kk$ and $\kk\rightarrow V^* \otimes V$). 
\end{enumerate}
The composition of these homomorphisms read, say, from bottom to top of a colored link diagram yield a polynomial link invariant. With the construction of the diagrammatic KLRW algebras, $\Tb^{\underline{\mu}}$, whose derived category of modules (loosely speaking) categorify the tensor products $V^{\mu_1}\otimes V^{\mu_2}\otimes \cdots \otimes V^{\mu_r}$, Webster promoted each homomorphism above into functors. Analogously, the composition of these functors yields some chain complex of graded $\kk$-modules, whose homology categorifies the corresponding knot invariants. Above, $\kk$ is just some background field. \\ \\
The question of whether this philosophy of diagrammatic categorification could apply to Verma module representations of $U_q(\mathfrak{g})$ (with generic weights) and what the resulting theory could mean is a natural one. The categorification of Verma modules has been explored in recent years (see, for instance, \cite{Na6}, \cite{Na5}, \cite{Na4}, \cite{Na2}, \cite{Rou2}) to great success. This led to the work of Dupont and Naisse \cite{Na}, which generalized the aforementioned KLRW algebras to include strands colored by Verma module representations (this was done for $\mathfrak{g}=\mathfrak{sl_2}$, but this is the only case of interest for the current paper). The resulting algebra of red (finite-dimensional) as well as blue (Verma) strands is a differential graded version of the KLRW algebras and was fittingly called the dg-KLRW algebra. These objects naturally led à la Webster to a categorification of bullet point 1) above. The purpose of the present paper is to explore the categorification of point 2), and the sequel will tackle point 3). \\ \\
One of the adjacent epic pursuits in quantum topology has been the categorification of the WRT invariants of $3$-manifolds along the same lines of their quantum link invariant counterparts. In \cite{GPPV}, a $q$-series valued invariant of $3$-manifolds that reduces (in a suitable sense) to the WRT invariants when $q$ is a root of unity was proposed for a certain class of 3-manifolds. The so-called $\hat{Z}$ invariants or BPS $q$-series, have integer coefficients and can be realized in physics as supersymmetric indices in M-theory (or equivalently, as a certain partition function in open topological string theory). This suggests these invariants can be categorified by BPS Hilbert spaces spanned by bound states of $M2$-$M5$ branes, which is the very same physical interpretation of Khovanov homology. \\ \\
The $\hat{Z}$ invariants for knot complements in $S^3$ were also constructed and studied in \cite{GM}. Amazingly, in \cite{Park}, it was shown that $\hat{Z}(S^3\backslash K)$ is the knot invariant resulting from following the Reshetikhin-Turaev dictionary, replacing finite-dimensional representations by Verma modules of generic highest (or lowest) weight. The resulting ``large color R-matrix'' is precisely the object we will categorify in this work. \\ \\
Of course, one can also include strands colored by finite-dimensional representations, in which case, one should find the analogue of the colored Jones polynomial for knot complements. This was done in \cite{Park2}, where given a colored link $L_{\underline{\mu}}$ in a knot complement $S^3\backslash K$, the Reshitikhin-Turaev procedure yielded $\hat{Z}\big(S^3\backslash K,L_{\underline{\mu}} \big)$, which was shown to depend only on the class of $L_{\underline{\mu}}$ in the relevant skein module of $S^3\backslash K$. Of course, when $K=\emptyset$, $\hat{Z}\big(S^3\backslash K,L_{\underline{\mu}} \big)$ reduces to the ordinary colored Jones polynomial (with a certain normalization). The generalization of Khovanov homology to knot complements would befittingly categorify these invariants, 
\[\begin{tikzcd}
\hat{Z}\big(S^3\backslash K,L_{\underline{\mu}} \big) 
\ar[rightsquigarrow]{rrrr}{\text{categorification}} &&&& \mathcal{H}^{S^3\backslash K}(L_{\underline{\mu}},\kk) 
\end{tikzcd}\]
Moreover, this setting would be conducive to a `functorial' theory of Khovanov homology. Among other things, this would include a categorified surgery formula for 3-manifold homological invariants. That is, suppose $M$ is a 3-manifold resulting from doing $\frac{p}{r}$ Dehn surgery on a knot $K$, for the $\hat{Z}$ invariants, there is a well-known surgery formula (see \cite{GJ}, \cite{GM} for details), 
$$\hat{Z}(M;q) = \mathcal{L}_{x,\frac{p}{r}}\big(\hat{Z}(S^3\backslash O;x^{\frac{1}{r}},q)\cdot \hat{Z}(S^3\backslash K;x,q)\big)$$
The resulting theory of homological invariants we aim to build with this work would enjoy some homological lift of this relation, 
$$\mathcal{H}(M,\kk) \cong H_*\bigg(\mathcal{C}(S^3\backslash O,\kk)\boxtimes_{\frac{p}{q}} \mathcal{C}(S^3\backslash K,\kk)\bigg)$$
The goal of this paper and its sequel is precisely to construct the homological invariants $\mathcal{H}^{S^3\backslash K}(L_{\underline{\mu}},\kk)$. \\ \\
The main feat of the present work is the construction of a categorical braid group action using dg-KLRW algebras, $\Tb^{\underline{\mu}}$. That is, for a suitable derived category of modules of $\Tb^{\underline{\mu}}$, which we call $\mathcal{D}_{dg}(\Tb^{\underline{\mu}})$ following \cite{Na}, we show,

\begin{thm}
    Let $\underline{\mu} = (\mu_1,\mu_2,...,\mu_r)$ be a tuple of (possibly Verma module) representations of $U_q(\mathfrak{sl}_2)$. The category $\mathcal{D}_{dg}(\Tb^{\underline{\mu}})$ carries functors, $\Bb_i$ for $i>1$, categorifying the action of the large color R-matrices ($R^{-1}$ in \cite{Park}, to be precise) corresponding to the braid group generators $\sigma_i$ and satisfying, 
    $$\Bb_i \circ \Bb_j \cong \Bb_j \circ \Bb_i \hspace{10mm} |i-j|>1$$
$$\Bb_i \circ \Bb_{i+1} \circ \Bb_i \cong \Bb_{i+1} \circ \Bb_i \circ \Bb_{i+1}$$
There are also functors $\Bb^{-1}_i$ categorifying the inverse R-matrices satisfying, 
 $$\Bb^{-1}_{\sigma_i} \circ \Bb_{\sigma_i} \cong \Bb_{\sigma_i}\circ \Bb^{-1}_{\sigma_i}\cong \mathbb{1}$$
    
\end{thm}

Following the Reshetikhin-Turaev procedure using the large color R-matrices and applying it to a braid, $\beta$, (as opposed to a knot) colored by highest weight Verma modules, one can also find the $\hat{Z}$ invariants for `braid complements.' These are homomorphisms from tensor products of Verma modules of generic highest weight $V^{\underline{\mu}} = V^{\mu_1}\otimes V^{\mu_2}\otimes \cdots \otimes V^{\mu_r}$ to $V^{\sigma_\beta (\underline{\mu})}$, 
$\hat{Z}(\beta): V^{\underline{\mu}}\rightarrow V^{\sigma_\beta (\underline{\mu})}$. One can also easily include links in $D^2\times I \backslash\beta$ colored by finite-dimensional representations of $U_q(\mathfrak{sl}_2)$, $\underline{\nu}$ , to find invariants, 
\be \hat{Z}(\beta, L_{\underline{\nu}}): V^{\underline{\mu}}\rightarrow V^{\sigma_\beta (\underline{\mu})} \label{eq: zhat_braids}\ee
Here, we also categorify these invariants in the case that all finite-dimensional representations are the fundamental representation, 

\begin{thm}
Given a braid $\beta$ and link in its complement colored by the fundamental representation of $U_q(\mathfrak{sl}_2)$, $L$, there is a functor-valued invariant of oriented tangles $\Phi'(\beta,L): \mathcal{D}_{dg}(\Tb^{\underline{\mu}}) \rightarrow \mathcal{D}_{dg}(\Tb^{\sigma_\beta(\underline{\mu})})$ categorifying \eqref{eq: zhat_braids}. 
\end{thm}

\subsection*{Summary}
We structure the paper as follows, 
\begin{itemize}
    \item In Section \ref{sec: prelim}, we give an overview of some of the necessary background for the later sections, including the diagrammatic definition of dgKLRW algebras and the relevant results from \cite{Na}. 
    \item In Section \ref{sec:Standards}, we slightly generalize the definition of the standard modules from \cite{Na} and study the iterated extension of projectives by these modules. This is analogous to the filtration by standards of the projectives in \cite{Web}. 
    \item In Section \ref{sec:braiding}, we define the braiding functors through derived tensor products with certain braiding bimodules. We study the braid group action through the projective resolutions of the braiding bimodules and show they endow the category $\mathcal{D}_{dg}(\Tb^{\underline{\mu}})$ with an honest braid group action, proving the first theorem above.
    \item In Section \ref{sec:fund_strands}, we adapt Webster's constructions for finite-dimensional representations to our case with blue strands. We check several topological moves needed for tangle invariance through explicit computations and define functor-valued invariants of oriented tangles mentioned in the second theorem above, $\Phi'$. 
    \item In Section \ref{sec:example}, we examine the homological invariant $\Phi'$ in the case that the braid $\beta$ is a single blue strand. This defines a homology theory of knots in $S^1\times D^2$, $\mathcal{H}^{S^1\times D^2}(\hspace{2mm}\cdot \hspace{2mm},\kk) $, which we compare to Annular Khovanov Homology in several explicit examples. Short of an honest isomorphism between them, we show the two theories seem to at least have the same dimension. 
    \item In Section \ref{sec:discussion}, we propose a framework for describing the cups and caps for blue strands using the `doubled' version of our KLRW algebra. We show that our proposed caps reduce to the correct $U_q(\mathfrak{sl}_2)$ expressions in the Grothendieck group and supply evidence to suggest the categorified caps for Verma modules appropiately reduce to their finite-dimensional counterparts after taking `finite-color' differentials. 
\end{itemize}
\section{Preliminaries}\label{sec: prelim}
\subsection{\texorpdfstring{$U_q(\mathfrak{sl_2})$}{Uq(sl2)} Conventions}
We follow the conventions of \cite{Na}. Let $U_q(\mathfrak{sl}_2)$ be the $\Q((q))$-algebra generated by elements $K,K^{-1},E,F$ with relations, 
$$
\begin{matrix}
    KE =q^2 EK &&&&&& KF = q^{-2}FK\\ \\ \\
    KK^{-1} = K^{-1}K =1 &&&&&& EF-FE = \frac{K-K^{-1}}{q-q^{-1}}
\end{matrix}$$

With comultiplication $\Delta: U_q(\mathfrak{sl_2}) \rightarrow U_q(\mathfrak{sl_2}) \otimes U_q(\mathfrak{sl_2})$ given by, 
$$\Delta(K^{\pm1}) = K^{\pm1}\otimes K^{\pm1}$$
$$\Delta(F) = F\otimes K +1\otimes F$$
$$\Delta(E) = E\otimes 1 +K^{-1}\otimes E$$

And counit, $\varepsilon(K^{\pm1}) = \pm1$, $\varepsilon(E) = \varepsilon(F) =0$. Additionally, we have the antipode, 
\begin{align*}
S(E)=-KE\\
S(F) = -FK^{-1}\\
S(K^{\pm1}) = K^{\mp1 }
\end{align*}
\subsubsection{Finite Dimensional Representations}
For each $N\in \N$, there is a finite dimensional, irreducible representation of $U_q(\mathfrak{sl}_2)$, $V(N)$. We use the basis $(v_{N,0}, v_{N,1},...,v_{N,N})$ with $U_q(\mathfrak{sl}_2)$ action given by, 
\begin{align*}
    K\cdot v_{N,i} = q^{N-2i}v_{N,i}\\
    F\cdot v_{N,i} = v_{N,i+1}\\
    E\cdot v_{N,i} = [i]_q [N-i+1]_qv_{N,i-1}
\end{align*}
\subsubsection{Verma Module Representations}
For a generic complex weight $\lambda \in \C$, the highest weight verma module $M(\lambda)$ is an infinite dimensional representation over $\Q((x,q))$, where $x=q^{2\lambda}$. It has basis $(v_{\lambda, 0}, v_{\lambda,1},...)$ with $U_q(\mathfrak{sl}_2)$ action given by, 
\begin{align*}
    K\cdot v_{\lambda,i} = xq^{-2i}v_{\lambda,i}\\
    F\cdot v_{\lambda,i} = v_{\lambda,i+1}\\
    E\cdot v_{\lambda,i} = [i]_q [2\lambda-i+1]_qv_{\lambda,i-1}
\end{align*}
Where we use the shorthand, 
$$[m \lambda +n]_q := \frac{x^\frac{m}{2} q^n - x^{-\frac{m}{2}} q^{-n}}{q-q^{-1}}$$

\subsection{Tensor Product Algebras}\label{sec:KLRW_def}
In \cite{Web}, the categorification of the tensor product of finite dimensional representations of quantum groups was introduced via an algebra $\Tb$, whose category of modules categorified the vector space corresponding to the tensor product. In \cite{Na}, this was generalized to include tensor products with Verma module representations of generic highest weight. In this section, we briefly review the necessary background and definitions from \cite{Na}. \\
\begin{defn}
    Fix some field $\kk$. Let $\underline{\mu} = (\mu_1,\mu_2,...,\mu_r)$ be a string of weights taking values in $\mathbb{N}$ or generically ( meaning $\mu_i=2\lambda$ with $\lambda\in \C/\frac{1}{2}\Z$). The \textbf{pre-dgKLRW algebra} $\tilde{\Tb}^{\underline{\mu}}$ algebra is the diagrammatic $\kk$-algebra defined as follows,
    \begin{itemize}
        \item $\tilde{\Tb}^{\underline{\mu}}$ is generated by diagrams made up from a collection of finitely many curves in $\R\times [0,1]$, which can be either colored or black. For $\tilde{\Tb}^{\underline{\mu}}$, we require for there to be $r$ colored curves, which we draw as red if $\mu_i \in \mathbb{N}$ and blue if $\mu_i$ is generic. We allow for arbitrarily many black strands. Colored strands may not intersect each other, but black strands can intersect each other or colored strands. Such diagrams will be called \textbf{string diagrams}.
        \item Black strands may carry dots or may be nailed to the leftmost strand along a white dot.  $\tilde{\Tb}^{\underline{\mu}}$ is $\Z\times \Z^2$ graded, with every element having a homological grading, $x$-grading, and $q$-grading. Here, $x=q^{\mu_i} = q^{2\lambda}$ whenever $\mu_i$ is a generic weight. The elementary diagrams are graded as follows, 
        \[
        deg\hspace{2mm} \tikzdiag[scale=1]{
        \draw (1,0)..controls (1,0.5) and (0,0.5) .. (0,1);
        \draw (0,0) ..controls (0,0.5) and (1,0.5) .. (1,1);
        }\hspace{2mm} =
        \hspace{2mm} q^{-2}
        \]
        \[
       deg\hspace{2mm} \tikzdiag[scale=1]{
        \draw (1,0)..controls (1,0.5) and (0,0.5) .. (0,1);
        \draw[pstdhl] (0,0) node[below]{\small $\mu_i$} ..controls (0,0.5) and (1,0.5) .. (1,1);
        }\hspace{2mm} =
        \hspace{2mm} deg\hspace{2mm} \tikzdiag[scale=1]{
        \draw (0,0)  ..controls (0,0.5) and (1,0.5) .. (1,1);
        \draw[pstdhl] (1,0) node[below]{\small $\mu_i$} ..controls (1,0.5) and (0,0.5) .. (0,1);
        } \hspace{2mm} = q^{\mu_i}
        \]
        \[
        deg\hspace{2mm} \tikzdiag[scale=1]{
        \draw (0,0) -- (0,1) node[tikzdot, pos=0.5]{};
        
        }\hspace{2mm} =
        \hspace{2mm} q^{2}
        \]
        \[
        deg\hspace{2mm} \tikzdiag[scale=1]{
        \draw (1,0)..controls (1,0.5) and (0,0.5) .. (0,0.5);
        \draw (0,0.5) ..controls (0,0.5) and (1,0.5) .. (1,1);
        \draw[pstdhl] (0,0) node[below]{\small $\mu_i$} --(0,1) node[midway,nail]{};
        }\hspace{2mm} =
        \hspace{2mm} h q^{2 \mu_i}
        \]
        \item We endow $\tilde{\Tb}^{\underline{\mu}}$ with algebraic structure by declaring that multiplication is vertical composition of diagrams when the ends match and $0$ otherwise. Our conventions are such that $a\cdot b$ is the composition of the diagrams $a$ and $b$ with $a$ at the top and $b$ at the bottom. 
        \item Any diagram with a black strand as the leftmost strand at some point along the vertical is 0
        
    \end{itemize}
\end{defn}

\begin{rem}
    In fact, we can just as easily allow for multiple strands colored by Verma modules of different generic weights. In this case the algebra would be $\Z_h \times \Z_q \times \Z^v$-graded, where $v$ is the number of distinct generic weights. For most of the paper we focus on $v=1$, so we will not dwell on this point. 
\end{rem}
With this in mind, we can define our main object of study,
\begin{defn}{\cite{Na}}\label{def:dgKLRWalgebra}
    The \textbf{dgKLRW} algebra $\Tb^{\underline{\mu}}$ is the quotient of $\tilde{\Tb}^{\underline{\mu}}$ by the following local relations, 
    \begin{itemize}
        \item KLR relations
        \be \tikzdiag[scale=2]{ 
    \draw (0,0.5)..controls (0,0.75) and (0.5, 0.75).. (0.5, 1);
      \draw (0,0) -- (1,1)  ;
      \draw(1,0) --(0,1);
      \draw (0.5,0)..controls (0.5,0.25) and (0,0.25) .. (0,0.5);
      } \hspace{2mm}
      =\hspace{2mm}
      \tikzdiag[scale=2]{ 
    \draw (0.5,0)..controls (0.5,0.25) and (1,0.25) .. (1,0.5);
      \draw (0,0) -- (1,1)  ;
      \draw (1,0) --(0,1);
      \draw (1,0.5)..controls (1,0.75) and (0.5, 0.75).. (0.5, 1);
      } \hspace{40mm} \tikzdiag[scale =2]{
      \draw (0,0) ..controls (1,0.25) and (1,0.75) .. (0,1);
      \draw (1,0) ..controls (0,0.25) and (0,0.75) .. (1,1);
      } \hspace{2mm}= \hspace{2mm}0 \ee

      \be 
      \tikzdiag[scale=1]{
        \draw (1,0)..controls (1,0.5) and (0,0.5) .. (0,1);
        \draw (0,0) ..controls (0,0.5) and (1,0.5) .. (1,1) node[tikzdot, pos=0.25]{};
        } \hspace{2mm}= \hspace{2mm}\tikzdiag[scale=1]{
        \draw (1,0)..controls (1,0.5) and (0,0.5) .. (0,1);
        \draw (0,0) ..controls (0,0.5) and (1,0.5) .. (1,1) node[tikzdot, pos=0.75]{};
        } \hspace{2mm}+ \hspace{2mm}\tikzdiag[scale=1]{
        \draw (0,0)--(0,1);
        \draw (1,0)--(1,1);
        } \hspace{30mm} \tikzdiag[xscale=-1]{
        \draw (1,0)..controls (1,0.5) and (0,0.5) .. (0,1);
        \draw (0,0) ..controls (0,0.5) and (1,0.5) .. (1,1) node[tikzdot, pos=0.75]{};
        } \hspace{2mm}= \hspace{2mm}\tikzdiag[xscale=-1]{
        \draw (1,0)..controls (1,0.5) and (0,0.5) .. (0,1);
        \draw (0,0) ..controls (0,0.5) and (1,0.5) .. (1,1) node[tikzdot, pos=0.25]{};
        } \hspace{2mm}+ \hspace{2mm}\tikzdiag[scale=1]{
        \draw (0,0)--(0,1);
        \draw (1,0)--(1,1);
        }
      \ee
      \item The relations that hold for all $\mu_i$,
      \be 
      \tikzdiag[scale=1]{
        \draw[pstdhl] (1,0)..controls (1,0.5) and (0,0.5) .. (0,1);
        \draw (0,0) ..controls (0,0.5) and (1,0.5) .. (1,1) node[tikzdot, pos=0.25]{};
        } \hspace{2mm}= \hspace{2mm}\tikzdiag[scale=1]{
        \draw (1,0)[pstdhl]..controls (1,0.5) and (0,0.5) .. (0,1);
        \draw (0,0) ..controls (0,0.5) and (1,0.5) .. (1,1) node[tikzdot, pos=0.75]{};
        } 
        \hspace{40mm}
        \tikzdiag[xscale=-1]{
        \draw (1,0)[pstdhl]..controls (1,0.5) and (0,0.5) .. (0,1);
        \draw (0,0) ..controls (0,0.5) and (1,0.5) .. (1,1) node[tikzdot, pos=0.75]{};
        } \hspace{2mm}= \hspace{2mm}\tikzdiag[xscale=-1]{
        \draw (1,0)[pstdhl]..controls (1,0.5) and (0,0.5) .. (0,1);
        \draw (0,0) ..controls (0,0.5) and (1,0.5) .. (1,1) node[tikzdot, pos=0.25]{};
        }
      \ee
      \be 
      \tikzdiag[scale=2]{ 
    \draw (0,0.5)..controls (0,0.75) and (0.5, 0.75).. (0.5, 1);
      \draw (0,0) -- (1,1)  ;
      \draw[pstdhl](1,0) --(0,1);
      \draw (0.5,0)..controls (0.5,0.25) and (0,0.25) .. (0,0.5);
      } \hspace{2mm}
      =\hspace{2mm}
      \tikzdiag[scale=2]{ 
    \draw (0.5,0)..controls (0.5,0.25) and (1,0.25) .. (1,0.5);
      \draw (0,0) -- (1,1)  ;
      \draw[pstdhl] (1,0) --(0,1);
      \draw (1,0.5)..controls (1,0.75) and (0.5, 0.75).. (0.5, 1);
      } \hspace{30mm} \tikzdiag[scale=2,xscale=-1]{ 
    \draw (0,0.5)..controls (0,0.75) and (0.5, 0.75).. (0.5, 1);
      \draw (0,0) -- (1,1)  ;
      \draw[pstdhl](1,0) --(0,1);
      \draw (0.5,0)..controls (0.5,0.25) and (0,0.25) .. (0,0.5);
      } \hspace{2mm}
      =\hspace{2mm}
      \tikzdiag[scale=2,xscale=-1]{ 
    \draw (0.5,0)..controls (0.5,0.25) and (1,0.25) .. (1,0.5);
      \draw (0,0) -- (1,1)  ;
      \draw[pstdhl] (1,0) --(0,1);
      \draw (1,0.5)..controls (1,0.75) and (0.5, 0.75).. (0.5, 1);
      }
      \ee
      \item The relations that hold only for red strands $\mu_i \in \mathbb{N}$, 
      \be 
      \tikzdiag[scale=2,scale =1]{
      \draw[stdhl] (0,0) node[below]{\small $\mu_i$} ..controls (1,0.25) and (1,0.75) .. (0,1);
      \draw (1,0) ..controls (0,0.25) and (0,0.75) .. (1,1);
      } \hspace{2mm}= \hspace{2mm} \tikzdiag[scale=2,scale =1]{
      \draw[stdhl] (0,0) node[below]{\small $\mu_i$} -- (0,1);
      \draw (1,0) -- (1,1) node[midway,tikzdot]{} node[midway,right]{\small $\mu_i$};
      } \hspace{30mm}
      \tikzdiag[scale=2,xscale =-1]{
      \draw[stdhl] (0,0) node[below]{\small $\mu_i$} ..controls (1,0.25) and (1,0.75) .. (0,1);
      \draw (1,0) ..controls (0,0.25) and (0,0.75) .. (1,1);
      } \hspace{2mm}= \hspace{2mm} \tikzdiag[scale=2,xscale =-1]{
      \draw[stdhl] (0,0) node[below]{\small $\mu_i$} -- (0,1);
      \draw (1,0) -- (1,1) node[midway,tikzdot]{} node[midway,right]{\small $\mu_i$};
      }
      \ee
      \be 
      \tikzdiag[scale=2]{ 
      \draw (0,0) -- (1,1)  ;
      \draw(1,0) --(0,1);
      \draw[stdhl] (0,0.5) ..controls (0,0.75) and (0.5, 0.75).. (0.5, 1);
      \draw[stdhl] (0.5,0) node[below]{\small $\mu_i$} ..controls (0.5,0.25) and (0,0.25) .. (0,0.5);
      } \hspace{2mm}
      =\hspace{2mm}
      \tikzdiag[scale=2]{ 
      \draw (0,0) -- (1,1)  ;
      \draw (1,0) --(0,1);
      \draw[stdhl]  (0.5,0) node[below]{\small $\mu_i$} ..controls (0.5,0.25) and (1,0.25) .. (1,0.5);
      \draw[stdhl] (1,0.5)..controls (1,0.75) and (0.5, 0.75).. (0.5, 1);
      }\hspace{2mm} + \hspace{2mm} \sum_{a+b +1 = \mu_i}\tikzdiag[scale=2]{ 
      \draw (0,0) -- (0,1) node[midway,tikzdot]{} node[midway,left]{\small $b$};
      \draw (1,0) --(1,1) node[midway,tikzdot]{} node[midway,right]{\small $a$};
      \draw[stdhl] (0.5,0) node[below]{\small $\mu_i$}--(0.5,1);
      }
      \ee
      \item The relations that hold only for blue strands ($\mu_i$ generic)
      \be 
      \tikzdiag[scale=2,scale =1]{
      \draw[vstdhl] (0,0) node[below]{\small $\mu_i$} ..controls (1,0.25) and (1,0.75) .. (0,1);
      \draw (1,0) ..controls (0,0.25) and (0,0.75) .. (1,1);
      } \hspace{2mm}= \hspace{2mm} 0 \hspace{30mm}
      \tikzdiag[scale=2,xscale =-1]{
      \draw[vstdhl] (0,0) node[below]{\small $\mu_i$} ..controls (1,0.25) and (1,0.75) .. (0,1);
      \draw (1,0) ..controls (0,0.25) and (0,0.75) .. (1,1);
      } \hspace{2mm}= \hspace{2mm} 0
      \ee
      \be 
      \tikzdiag[scale=2]{ 
      \draw (0,0) -- (1,1)  ;
      \draw(1,0) --(0,1);
      \draw[vstdhl] (0,0.5) ..controls (0,0.75) and (0.5, 0.75).. (0.5, 1);
      \draw[vstdhl] (0.5,0) node[below]{\small $\mu_i$} ..controls (0.5,0.25) and (0,0.25) .. (0,0.5);
      } \hspace{2mm}
      =\hspace{2mm}
      \tikzdiag[scale=2]{ 
      \draw (0,0) -- (1,1)  ;
      \draw (1,0) --(0,1);
      \draw[vstdhl]  (0.5,0) node[below]{\small $\mu_i$} ..controls (0.5,0.25) and (1,0.25) .. (1,0.5);
      \draw[vstdhl] (1,0.5)..controls (1,0.75) and (0.5, 0.75).. (0.5, 1);
      }\hspace{2mm}  \ee
      \item The nail relations, 
      \be 
      \tikzdiag[scale=1]{
        \draw (1,0)..controls (1,0.5) and (0,0.5) .. (0,0.5);
        \draw (0,0.5) ..controls (0,0.5) and (1,0.5) .. (1,1) node[pos=.75,tikzdot]{};
        \draw[pstdhl] (0,0) node[below]{\small $\mu_i$} --(0,1) node[midway,nail]{};
        } \hspace{2mm} = \hspace{2mm} \tikzdiag[scale=1]{
        \draw (1,0)..controls (1,0.5) and (0,0.5) .. (0,0.5) node[pos=.25,tikzdot]{};
        \draw (0,0.5) ..controls (0,0.5) and (1,0.5) .. (1,1);
        \draw[pstdhl] (0,0) node[below]{\small $\mu_i$} --(0,1) node[midway,nail]{};
        } 
        \hspace{20mm} \tikzdiag[scale=1]{
        \draw (1,0)..controls (1,0.5) and (0,0.5) .. (0,0.5);
        \draw (0,0.5) ..controls (0,0.5) and (1,0.5) .. (1,2);
        \draw (2,0)..controls (2,1.5) and (0,1.5) .. (0,1.5);
        \draw (0,1.5) ..controls (0,1.5) and (2,1.5) .. (2,2);
        \draw[pstdhl] (0,0) node[below]{\small $\mu_i$} --(0,2) node[pos=.25,nail]{} node[pos=.75,nail]{} ;
        }\hspace{2mm} = -\hspace{2mm}
        \tikzdiag[scale=1]{
        \draw (1,0)..controls (1,1.5) and (0,1.5) .. (0,1.5);
        \draw (0,1.5) ..controls (0,1.5) and (1,1.5) .. (1,2);
        \draw (2,0)..controls (2,.5) and (0,.5) .. (0,.5);
        \draw (0,.5) ..controls (0,.5) and (2,.5) .. (2,2);
        \draw[pstdhl] (0,0) node[below]{\small $\mu_i$} --(0,2) node[pos=.25,nail]{} node[pos=.75,nail]{} ;
        }
      \ee
      \be \tikzdiag[scale=1]{
        \draw (1,0)..controls (1,.5) and (0,.5) .. (0,.5);
        \draw (0,.5) .. controls (1,.5) and (1,1.5) ..(0,1.5);
        \draw (0,1.5) ..controls (1,1.5) and (1,2) .. (1,2);
        \draw[pstdhl] (0,0) node[below]{\small $\mu_i$} --(0,2) node[pos=.25,nail]{} node[pos=.75,nail]{} ;
        } \hspace{2mm} = \hspace{2mm} 0 \ee
      
    \end{itemize}
    Additionally, one turns $\Tb^{\underline{\mu}}$ into a differentially graded algebra, $(\Tb^{\underline{\mu}}, d_{\underline{\mu}})$, by defining the differential $d_{\underline{\mu}}$ to be zero on dots and crossings and,
    \be \label{eq:differential} d_{\underline{\mu}} \left(
    \tikzdiag[scale=1]{
        \draw (1,0)..controls (1,0.5) and (0,0.5) .. (0,0.5);
        \draw (0,0.5) ..controls (0,0.5) and (1,0.5) .. (1,1);
        \draw[pstdhl] (0,0) node[below]{\small $\mu_1$} --(0,1) node[midway,nail]{};
        } \right) = \begin{cases}
            \tikzdiag[scale=1,scale =1]{
      \draw[stdhl] (0,0) node[below]{\small $\mu_1$} -- (0,1);
      \draw (1,0) -- (1,1) node[midway,tikzdot]{} node[midway,right]{\small $\mu_1$};
      } &\text{if }\mu_1 \in \mathbb{N} \\
            0&\text{otherwise}
        \end{cases}
        \ee
    and extending it via the Leibnitz rule.
\end{defn}
\begin{rem}
    For most of this paper, we will be focused on the case that $\mu_1$, the leftmost strand, is generic so that $d_{\underline{\mu}} =0$ and $(\Tb^{\underline{\mu}}, d_{\underline{\mu}}) = (\Tb^{\underline{\mu}}, 0)$. For this reason, we often don't write the differential and denote $(\Tb^{\underline{\mu}}, 0) \rightarrow \Tb^{\underline{\mu}}$. This will greatly simplify problems which apriori require careful analyses of the underlying dg-enhanced derived dg-categories and other flavors of abstract non-sense, and mostly deal in the familiar land of homological algebra. For instance, for a dg $\Tb^{\underline{\mu}}$-module, $(M,d_M)$, this makes it so that the graded dimension, 
    $$gdim_{q,x} (M) := \sum_{i,j,k} x^i q^j t^k H_k(M_{(i,j)},d_M)$$
    is simply,
    $$gdim_{q,x} (M) := \sum_{i,j,k} x^i q^j t^k M_{(i,j,k)}$$
    Therefore $gdim_{q,x}$ will behave well under exact sequences and allow us to make dimensional arguments. 
    
\end{rem}

The multiplication structure implies the algebra decomposes into, 
$$\Tb^{\underline{\mu}} = \bigoplus_{\rho}\Tb^{\underline{\mu}}1_\rho$$
Where the direct sum in $\rho = (b_1,b_2,...,b_r)$ is over all tuples in $\mathbb{N}^r$ (recall $r$ is the number of colored strands) and $1_\rho$ is the idempotent corresponding to the configuration of black strands and colored strands associated to such a tuple. We abuse notation a bit to include the data $\underline{\mu}$ in $\rho$. 
\[1_\rho \hspace{2mm} = \hspace{2mm }\tikzdiag[scale=1]{
\draw[pstdhl] (0,0) node[below]{\small $\mu_1$}--(0,1);
\draw (0.25,0)--(0.25,1);
\draw (0.75,0)--(0.75,1) node[midway,left]{$...$};
\tikzbrace{0.25}{0.75}{-0.1}{$b_1$}
}
\tikzdiag[scale=1]{
\draw[pstdhl] (0,0) node[below]{\small $\mu_2$}--(0,1);
\draw (0.25,0)--(0.25,1);
\draw (0.75,0)--(0.75,1) node[midway,left]{$...$} node[midway,right]{\large $\cdots$};
\tikzbrace{0.25}{0.75}{-0.1}{$b_2$}
}
\tikzdiag[scale=1]{
\draw[pstdhl] (0,0) node[below]{\small $\mu_r$}--(0,1);
\draw (0.25,0)--(0.25,1);
\draw (0.75,0)--(0.75,1) node[midway,left]{$...$};
\tikzbrace{0.25}{0.75}{-0.1}{$b_r$}
}
\]
The resulting projective (or cofibrant) modules over this algebra, $\Pb^{\underline{\mu}}_\rho$, are therefore drawn as, 
$$\Pb^{\underline{\mu}}_\rho \hspace{2mm} = \hspace{2mm} \tikzdiag[yscale=.75]{
\draw[pstdhl] (0,0) node[below]{\small $\mu_1$}--(0,1);
\draw (0.25,0)--(0.25,1);
\draw (0.75,0)--(0.75,1) node[midway,left]{$...$};
\draw[pstdhl] (1,0) node[below]{\small $\mu_2$}--(1,1);
\draw (1.25,0)--(1.25,1);
\draw (1.75,0)--(1.75,1) node[midway,left]{$...$} node[midway, right]{\large $\hspace{2mm}\cdots$};
\draw[pstdhl] (3,0) node[below]{\small $\mu_r$}--(3,1);
\draw (3.25,0)--(3.25,1);
\draw (3.75,0)--(3.75,1) node[midway,left]{$...$};
\filldraw [fill=white, draw=black] (-0.25,1) rectangle (4.25,2) node[midway] { $\Tb$};
}$$

From the dgKLRW algebra, $\Tb^{\underline{\mu}}$, one constructs the (dg-enhanced) derived dg-category of dg $\Tb^{\underline{\mu}}$-modules, which we denote $\mathcal{D}_{dg}(\Tb^{\underline{\mu}})$ (see details of the general construction in Section 3.1.5 of \cite{Na}).\\ Let us describe the categorification of the $F$ and $E$ operators in this category. Let $\Tb^{\underline{\mu}}_b$ be the subalgebra of $\Tb^{\underline{\mu}}$ on $b$ black strands (clearly $\Tb^{\underline{\mu}} = \bigoplus_{b\geq 0} \Tb^{\underline{\mu}}_b$). Then, we have, 

\be \label{eq:F_functor}\mathcal{F}_b(-) := \Tb^{\underline{\mu}}_{b+1}1_{b,1} \bigotimes_{\Tb^{\underline{\mu}}_b} -:\mathcal{D}_{dg}(\Tb^{\underline{\mu}}_b)\rightarrow \mathcal{D}_{dg}(\Tb^{\underline{\mu}}_{b+1})\ee 
\be \mathcal{E}_b(-) := q^{2b+1-|\underline{\mu}|}\hspace{2mm}1_{b,1}\Tb^{\underline{\mu}}_{b+1} \bigotimes_{\Tb^{\underline{\mu}}_{b+1}} -:\mathcal{D}_{dg}(\Tb^{\underline{\mu}})_{b+1}\rightarrow \mathcal{D}_{dg}(\Tb^{\underline{\mu}}_b)\ee 

We also define the grading shift functor, 
$$\mathcal{K}_b(-):= q^{|\underline{\mu}|-2b}\cdot \mathbb{1}(-) :\mathcal{D}_{dg}(\Tb^{\underline{\mu}}_b)\rightarrow \mathcal{D}_{dg}(\Tb^{\underline{\mu}}_{b})$$

Of course, to extend these to the full $\Tb^{\underline{\mu}}$ algebra, we simply take, $\mathcal{F} := \bigoplus_{b\geq0}\mathcal{F}_b$, $\mathcal{E} := \bigoplus_{b\geq0}\mathcal{E}_b$, $\mathcal{K} := \bigoplus_{b\geq0}\mathcal{K}_b$. This leads to the categorification of the commutation relation $EF-FE = \frac{K-K^{-1}}{q-q^{-1}}$, 
\begin{thm}{\cite{Na}}
    There is a quasi-isomorphism
    $$Cone(\mathcal{F}\mathcal{E} \rightarrow \mathcal{E}\mathcal{F}) \cong Cone(\bigoplus_{p\geq 0} q^{2p+1}\mathcal{K} \rightarrow \bigoplus_{p\geq 0} q^{2p+1}\mathcal{K}^{-1} )$$
\end{thm}

Similarly, we have the categorification theorem, 
\begin{thm}{\cite{Na}} Let $\boldsymbol{K}^\Delta_{0,\Q}(-)$ denote the asymptotic Grothendieck group (see \cite{Na3}) with rational coefficients. Additionally, let $V(\underline{\mu}) := V(\mu_1)\otimes V(\mu_2)\otimes\cdots \otimes V(\mu_r)$, where we mean that if $\mu_i$ is generic, then $V(\mu_i) = M(\lambda)$. There is an isomorphism
    $$V(\underline{\mu}) \rightarrow \boldsymbol{K}^\Delta_{0,\Q}(\Tb^{\underline{\mu}}) $$
    sending, 
    $$F^{b_r}(\cdots F^{b_2}(F^{b_1}v_{\mu_1,0}\otimes v_{\mu_2,0})\otimes \cdots \otimes v_{\mu_r,0} ) \mapsto [\Pb^{\underline{\mu}}_\rho]$$
    
\end{thm}

\section{Standard Modules}\label{sec:Standards}
In \cite{Web}, the notion of standard modules were introduced to categorify the tensor product basis for tensors of finite-dimensional representations of the quantum group $V_{\lambda_1}\otimes V_{\lambda_2}\otimes \cdots \otimes V_{\lambda_r}$. In \cite{Na}, this notion was extended to the case when the tensor products include Verma module representations of generic complex weight. There, modules $\Sb^{\underline{\mu}}_{\rho} $ categorifying the basis elements $F^{b_r}v_{\mu_r}\otimes \cdots \otimes F^{b_1}v_{\mu_1} $ were constructed as well as the functor,
$$\mathbb{S}^{\underline{\mu}}:\mathcal{D}_{dg}(\Tb^{\mu_1} \otimes \cdots \otimes \Tb^{\mu_r}) \rightarrow \mathcal{D}_{dg}(\Tb^{\underline{\mu}} ) $$
$$\mathbb{S}^{\underline{\mu}} (-) = \Sb^{\underline{\mu}}\otimes^L_{\Tb}-$$
sending $\Pb_{b_1}^{\mu_1} \otimes \cdots \otimes \Pb_{b_r}^{\mu_r} \mapsto \Sb^{\underline{\mu}}_{\rho}$. In this section, we will briefly review these results as well as their generalization,

\be \label{eq: standard_functor}\mathbb{S}^{\underline{\mu_1}; ... ;\underline{\mu_\ell}}:\mathcal{D}_{dg}(\Tb^{\underline{\mu_1}} \otimes \cdots \otimes \Tb^{\underline{\mu_\ell}}) \rightarrow \mathcal{D}_{dg}(\Tb^{\underline{\mu}} ) \ee 

$$
\mathbb{S}^{\underline{\mu_1}; ... ;\underline{\mu_\ell}} (-) = \Sb^{\underline{\mu_1}; ... ;\underline{\mu_\ell}}\otimes^L_{\Tb}-$$

sending $\Pb_{\rho_1}^{\underline{\mu_1}} \otimes \cdots \otimes \Pb_{\rho_\ell}^{\underline{\mu_\ell}} \mapsto \Sb^{\underline{\mu_1}; ... ;\underline{\mu_\ell}}_{\rho_1;...;\rho_\ell}$.   \\ \\
The construction closely parallels that of \cite{Na}. First, fix $\underline{\mu_1},\underline{\mu_2},...,\underline{\mu_\ell}$ with each $\underline{\mu_i} = (\mu^i_1,\mu^i_2,...,\mu^i_{n_i})$. Similarly fix the black strand configurations $\rho_1, ..., \rho_\ell$ with $\rho_i = (b^i_1,...,b^i_{n_i})$. Let $J_{k,\rho_i}:= \{1,2,...,b^i_k\}$ and define, 
$$J_{\rho_i}:= \bigsqcup^{n_i}_{k=1}J_{k,\rho_i}$$
$$J_\rho = \bigsqcup^\ell_{i=2}J_{\rho_i}$$
For a subset $\mathbf{j}\subset J_\rho$, define $j_{k,\rho_i} = \mathbf{j}\cap J_{k,\rho_i} $ and $j_{\rho_i}:= \mathbf{j} \cap J_{\rho_i}$. We say $b<b'$ in $j_{\rho_i}$ if $b\in j_{k,\rho_i}$ and $b' \in j_{k',\rho_i}$ with $k\leq k'$. Let $j^i := (|j_{1,\rho_i}|,|j_{2,\rho_i}|,...,|j_{n_i,\rho_i}|)$. We consider the black strand configurations $\rho(\mathbf{j})$, defined by, 
$$\rho(\mathbf{j}) = (\rho_1 +(0,0,...,n(j^2)), \rho_2-j^2 + (0,0,...,n(j^3)),..., \rho_\ell -j^\ell)$$
where $n(j^i) = \sum^{n_i}_{k=1}|j_{k,\rho_i}|$. Then, for
shifted projectives of the kind, 
$$\Pb^{\underline{\mu}}_{\rho(\mathbf{j})} [|\mathbf{j}|]$$
where $\underline{\mu}$ is the concatenation of $\underline{\mu_1}, ... ,\underline{\mu_\ell}$, we may define maps of the following form. \\ \\
Whenever $\mathbf{j'}\subset \mathbf{j}$ such that $|\mathbf{j'}|= |\mathbf{j}|-1$, we know $\mathbf{j'}\cup\{b'\} =  \mathbf{j}$ for some $b'\in J_{k,\rho_i}$ for some $i$ and $k$. Then, there is a map, 
$$\tau_{\mathbf{j},\mathbf{j'}}:q^{deg\hspace{1mm}\tau_{\mathbf{j},\mathbf{j'}}}\Pb^{\underline{\mu}}_{\rho(\mathbf{j})} [|\mathbf{j}|]\rightarrow \Pb^{\underline{\mu}}_{\rho(\mathbf{j'})} [|\mathbf{j'}|] $$
constructed through right multiplication by the following element, 
$$\tau_{\mathbf{j},\mathbf{j'}} = \tikzdiag[yscale=2]{
\draw (-0.25,0)--(-.25,1)node[midway,left]{$...$};
\draw (-0.75,0)--(-.75,1) node[midway,left]{\large $\cdots$};
\draw[pstdhl] (0,0) node[below]{\small $\mu^i_1$}--(0,1);
\draw (0.25,0)--(0.25,1);
\draw (0.75,0)--(0.75,1) node[midway,left]{$...$};
\draw[pstdhl] (1,0) node[below]{\small $\mu^i_2$}--(1,1);
\draw (1.25,0)--(1.25,1);
\draw (1.75,0)--(1.75,1) node[pos=0.25,left]{$...$} node[midway, right]{\large $\hspace{2mm}\cdots$};
\draw[pstdhl] (3,0) node[below]{\small $\mu^i_k$}--(3,1);
\draw (3.25,0)--(3.25,1);
\draw (3.75,0)--(3.75,1) node[midway,left]{$...$} node[midway,right]{\large $\cdots$};
\draw (4,0)..controls (4,0.5) and (-1,0.5) .. (-1,1);
\tikzbrace{3.25}{3.75}{0.1}{\small $p_2$};
\tikzbrace{-.75}{-.25}{0.1}{\small $p_1$};
\tikzbraceop{1.25}{1.75}{0.9}{\small $\rho(\mathbf{j})_{i,2}$};
\tikzbraceop{.25}{.75}{0.9}{\small $\rho(\mathbf{j})_{i,1}$};
}$$

Where $p_1+p_2 = b'-1$ and $p_1 = |\{b'' \in j_{\rho_i}| b''<b'\}|$. Notice $deg\hspace{1mm}\tau_{\boldsymbol{j},\boldsymbol{j}'} = q^{-2(b'-1)-2\sum^{k-1}_{m=1}\rho(\boldsymbol{j})_{i,m}}q^{\sum^k_{m=1} \mu^i_k}$. As an extension of a result from \cite{Na}, we have, 
\begin{lem} \cite{Na}
    If $\boldsymbol{j}''' \subset \boldsymbol{j}' \subset \boldsymbol{j}$ and $\boldsymbol{j}''' \subset \boldsymbol{j}'' \subset \boldsymbol{j}$ such that $|\boldsymbol{j} | = |\boldsymbol{j}'|+1 = |\boldsymbol{j}''|+1 = |\boldsymbol{j}'''|+2$, then we have, 
    $$\tau_{\mathbf{j},\mathbf{j'}} \tau_{\mathbf{j}',\mathbf{j}'''} = \tau_{\mathbf{j},\mathbf{j}''} \tau_{\mathbf{j}'',\mathbf{j}'''} $$
\end{lem}

\begin{defn} \cite{Na}
    The standard module $\Sb^{\underline{\mu_1}; ... ;\underline{\mu_\ell}}_{\rho_1;...;\rho_\ell}$ is defined by the chain complex specified by, 
    $$\bigoplus_{\mathbf{j}\subset J_\rho} q^{\#_{\mathbf{j}}}\Pb^{\underline{\mu}}_{\rho(\mathbf{j})} [|\mathbf{j}|]$$
    With differentials given by (we omit the differentials coming from the differentially graded structure since we are again assuming $d_\mu =0$),
    $$d = \underset{}{\sum_{\boldsymbol{j}\subset J_\rho}} d_{\boldsymbol{j}}$$
    $$d_{\boldsymbol{j}} = \underset{|\boldsymbol{j}'|= |\boldsymbol{j}|-1}{\sum_{\boldsymbol{j}' \subset\boldsymbol{j}}} (-1)^{c_{\boldsymbol{j}-\boldsymbol{j}'}} \tau_{\boldsymbol{j},\boldsymbol{j}'}$$
    Above, $c_{b'} := |\{b'' \in \boldsymbol{j} | b''>b'\}|$ and $\#_{\boldsymbol{j}}$ is the uniquely determined $q,x$ grading shifts determined by the degree of the maps $\tau_{\boldsymbol{j},\boldsymbol{j}'}$. 
\end{defn}
The standard modules categorify the tensor product basis elements in the sense below, 
\begin{prop}
    In the asymptotic Grothenidieck group $\boldsymbol{K}^\Delta_{0,\Q}(-)$, we have, 
    $$[\Sb^{\underline{\mu_1}; ... ;\underline{\mu_\ell}}_{\rho_1;...;\rho_\ell}] = [\Pb^{\underline{\mu_1}}_{\rho_1}]\otimes [\Pb^{\underline{\mu_2}}_{\rho_2}] \otimes ... \otimes [\Pb^{\underline{\mu_\ell}}_{\rho_\ell}]$$
\end{prop}
\begin{proof}
    This is a standard proof by induction using the coproduct structure of decategorified $U_q(\mathfrak{sl}_2)$ (see Section \ref{sec: prelim}). 
\end{proof}

\begin{ex}
    As an example, let us consider arbitrary $\underline{\mu_1}$ and $\rho_1$ with $\underline{\mu_2} = (\lambda,\lambda)$ and $\rho_2 = (1,1)$. In this case, 
    \[
    \Sb^{\underline{\mu_1}; \underline{\mu_2}}_{\rho_1; \rho_2} = 
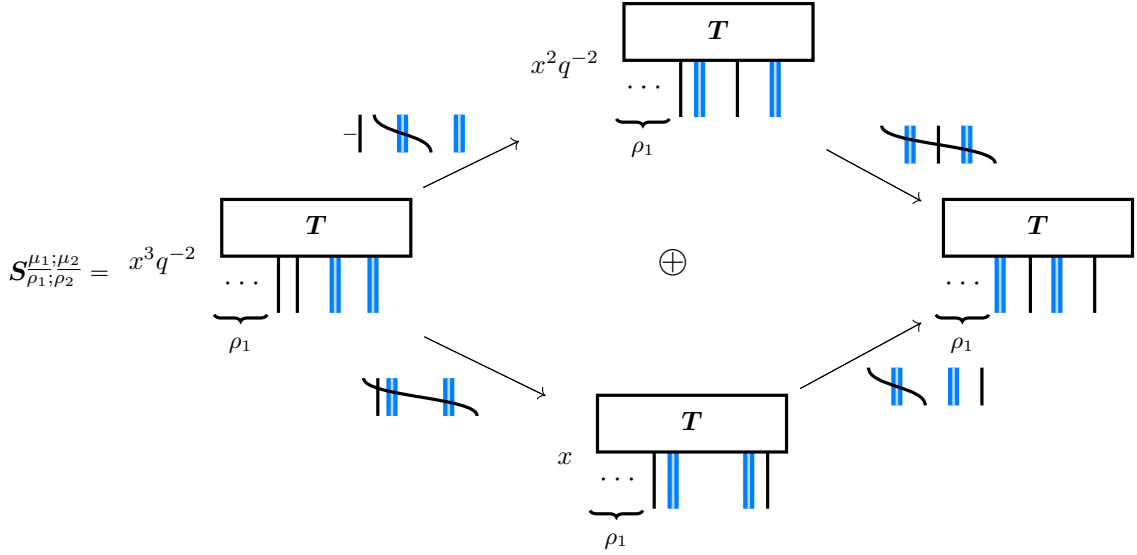
    \begin{tikzcd}[column sep=large,row sep=tiny]
& x^2 q^{-2} \hspace{2mm}\tikzdiagh[yscale=1.5]{0}{
\draw[vstdhl](0.5,0)--(.5,0.5);
\draw[vstdhl](1.5,0)--(1.5,0.5);
\draw (0.25,0)--(0.25,0.5) node[midway,left]{\large $\cdots$};
\draw (1,0)--(1,0.5);
\filldraw [fill=white, draw=black] (-0.5,.5) rectangle (2,1) node[midway] { $\Tb$};
\tikzbrace{-0.5}{0}{0.1}{\small $\rho_1$};
}
\ar{rd}
\ar[dash]{rd}{
\tikzdiag[xscale=-0.75]{
\draw[vstdhl](0.5,0)--(0.5,0.5);
\draw[vstdhl](1.5,0)--(1.5,0.5);
\draw (0,0)..controls (0,0.25) and (2,0.25)..(2,0.5);
\draw (1,0)--(1,0.5);
}
}
& \\
x^3 q^{-2}\hspace{2mm}\tikzdiagh[yscale=1.5]{0}{
\draw[vstdhl](1,0)--(1,0.5);
\draw[vstdhl](1.5,0)--(1.5,0.5);
\draw (0.25,0)--(0.25,0.5) node[midway,left]{\large $\cdots$};
\draw (0.5,0)--(0.5,0.5);
\filldraw [fill=white, draw=black] (-0.5,.5) rectangle (2,1) node[midway] { $\Tb$};
\tikzbrace{-0.5}{0}{0.1}{\small $\rho_1$};
}
\ar{ru}
\ar[dash]{ru}{-
\tikzdiag[xscale=-0.75,yscale=1]{
\draw[vstdhl](0.5,0)--(0.5,0.5);
\draw[vstdhl](1.5,0)--(1.5,0.5);
\draw (1,0)..controls (1,0.25) and (2,0.25)..(2,0.5);
\draw (2.25,0)--(2.25,0.5);
}
}
\ar{rd}
\ar[swap,dash]{rd}{
\tikzdiag[xscale=0.75,yscale=1]{
\draw[vstdhl](0.5,0)--(0.5,0.5);
\draw[vstdhl](1.5,0)--(1.5,0.5);
\draw (2,0)..controls (2,0.25) and (0,0.25)..(0,0.5);
\draw (.25,0)--(.25,0.5);
}
}
&
\bigoplus &
\tikzdiagh[yscale=1.5]{0}{
\draw[vstdhl](0.25,0)--(0.25,0.5) node[midway,left,color = black]{\large $\cdots$};
\draw[vstdhl](1,0)--(1,0.5);
\draw (0.65,0)--(0.65,0.5);
\draw (1.5,0)--(1.5,0.5);
\filldraw [fill=white, draw=black] (-0.5,.5) rectangle (2,1) node[midway] { $\Tb$};
\tikzbrace{-0.5}{0}{0.1}{\small $\rho_1$};
} \\
& 
x\hspace{2mm}\tikzdiagh[yscale=1.5]{0}{
\draw[vstdhl](0.5,0)--(0.5,0.5);
\draw[vstdhl](1.5,0)--(1.5,0.5);
\draw (0.25,0)--(0.25,0.5) node[midway,left]{\large $\cdots$};
\draw (1.75,0)--(1.75,0.5);
\filldraw [fill=white, draw=black] (-0.5,.5) rectangle (2,1) node[midway] { $\Tb$};
\tikzbrace{-0.5}{0}{0.1}{\small $\rho_1$};
}
\ar{ru}
\ar[swap,dash]{ru}{
\tikzdiag[xscale=-0.75,yscale=1]{
\draw[vstdhl](0.5,0)--(0.5,0.5);
\draw[vstdhl](1.5,0)--(1.5,0.5);
\draw (1,0)..controls (1,0.25) and (2,0.25)..(2,0.5);
\draw (0,0)--(0,0.5);
}
}
&
\end{tikzcd}
    \]
\end{ex}

One can give the modules $\Sb^{\underline{\mu_1}; ... ;\underline{\mu_\ell}} := \bigoplus_{\rho_1,...,\rho_\ell} \Sb^{\underline{\mu_1}; ... ;\underline{\mu_\ell}}_{\rho_1;...;\rho_\ell}$ the structure of a $(\Tb^{\underline{\mu}},d_\mu) - (\Tb^{\underline{\mu_1}} \otimes \Tb^{\underline{\mu_2}} \otimes \cdots \otimes \Tb^{\underline{\mu_\ell}}, d_{\mu_1} + d_{\mu_2} + \cdots + d_{\mu_\ell} )$ dg-bimodule as in Proposition 7.16 in Section 7 of \cite{Na} to construct the functor \eqref{eq: standard_functor}. The analogous results of that section follow here as well, though we do not devote the time to prove so.

\subsection{Iterated Extensions}
In \cite{Web}, it was shown that the projective modules over the KLRW algebra associated to tensors of finite dimensional irreducible representations are filtrated by standard modules. Here, we show an analogous result for the dg-enhanced KLRW algebras. Namely, the projective modules in our category have an iterated extension by the standard modules defined in \cite{Na}. \\
Let us think about what this looks like in simple examples. 
\begin{ex}{$\Pb^{\lambda,\lambda,\lambda}_{0,0,1}$}
In this case we write the iterated extension as a series of maps: 

\[
\begin{tikzcd}[row sep=0ex]
0\ar{r}
&
x^2 \hspace{2mm}\tikzdiagh[yscale=1]{0}{
\draw[vstdhl](0,0)--(0,0.5);
\draw[vstdhl](0.5,0)--(0.5,0.5);
\draw[vstdhl](1,0)--(1,0.5);
\draw (0.25,0)--(0.25,0.5);
\filldraw [fill=white, draw=black] (-0.25,.5) rectangle (1.5,1) node[midway] { $\Tb$};
}
\ar{rr}
\ar[dash]{rr}{
\tikzdiag[scale=1]{
	\draw[vstdhl](0,0)--(0,0.5);
\draw[vstdhl](0.5,0)--(0.5,0.5);
\draw[vstdhl](1,0)--(1,0.5);
\draw (0.75,0)..controls (0.75,0.25) and (0.25,0.25)..(0.25,0.5);
}
} & &
x \hspace{2mm}\tikzdiagh[yscale=1]{0}{
\draw[vstdhl](0,0)--(0,0.5);
\draw[vstdhl](0.5,0)--(0.5,0.5);
\draw[vstdhl](1,0)--(1,0.5);
\draw (0.75,0)--(0.75,0.5);
\filldraw [fill=white, draw=black] (-0.25,.5) rectangle (1.5,1) node[midway] { $\Tb$};
} 
\ar{rr}
\ar[dash]{rr}{
\tikzdiag[scale=1]{
	\draw[vstdhl](0,0)--(0,0.5);
\draw[vstdhl](0.5,0)--(0.5,0.5);
\draw[vstdhl](1,0)--(1,0.5);
\draw (1.25,0)..controls (1.25,0.25) and (0.75,0.25)..(0.75,0.5);
}
}
& &
\tikzdiagh[yscale=1]{0}{
\draw[vstdhl](0,0)--(0,0.5);
\draw[vstdhl](0.5,0)--(0.5,0.5);
\draw[vstdhl](1,0)--(1,0.5);
\draw (1.25,0)--(1.25,0.5);
\filldraw [fill=white, draw=black] (-0.25,.5) rectangle (1.5,1) node[midway] { $\Tb$};
} 
\end{tikzcd}
\]

Note that the mapping cone of the first map is exactly $\Sb^{\lambda,\lambda,\lambda}_{0,0,1}$,

\[\Sb^{\lambda,\lambda,\lambda}_{0,0,1} \cong 
\begin{tikzcd}
    x \hspace{2mm} \tikzdiagh[yscale=1]{0}{
\draw[vstdhl](0,0)--(0,0.5);
\draw[vstdhl](0.5,0)--(0.5,0.5);
\draw[vstdhl](1,0)--(1,0.5);
\draw (0.75,0)--(0.75,0.5);
\filldraw [fill=white, draw=black] (-0.25,.5) rectangle (1.5,1) node[midway] { $\Tb$};
} 
\ar{rr}
\ar[dash]{rr}{
\tikzdiag[scale=1]{
	\draw[vstdhl](0,0)--(0,0.5);
\draw[vstdhl](0.5,0)--(0.5,0.5);
\draw[vstdhl](1,0)--(1,0.5);
\draw (1.25,0)..controls (1.25,0.25) and (0.75,0.25)..(0.75,0.5);
}
}
& &
\tikzdiagh[yscale=1]{0}{
\draw[vstdhl](0,0)--(0,0.5);
\draw[vstdhl](0.5,0)--(0.5,0.5);
\draw[vstdhl](1,0)--(1,0.5);
\draw (1.25,0)--(1.25,0.5);
\filldraw [fill=white, draw=black] (-0.25,.5) rectangle (1.5,1) node[midway] { $\Tb$};
} 
\end{tikzcd}
\]

Similarly, for the cone of the 2nd map, we find, 
\[ x \Sb^{\lambda,\lambda,\lambda}_{0,1,0} \cong
\begin{tikzcd}
    x^2 \hspace{2mm}\tikzdiagh[yscale=1]{0}{
\draw[vstdhl](0,0)--(0,0.5);
\draw[vstdhl](0.5,0)--(0.5,0.5);
\draw[vstdhl](1,0)--(1,0.5);
\draw (0.25,0)--(0.25,0.5);
\filldraw [fill=white, draw=black] (-0.25,.5) rectangle (1.5,1) node[midway] { $\Tb$};
}
\ar{rr}
\ar[dash]{rr}{
\tikzdiag[scale=1]{
	\draw[vstdhl](0,0)--(0,0.5);
\draw[vstdhl](0.5,0)--(0.5,0.5);
\draw[vstdhl](1,0)--(1,0.5);
\draw (0.75,0)..controls (0.75,0.25) and (0.25,0.25)..(0.25,0.5);
}
} & &
x \hspace{2mm}\tikzdiagh[yscale=1]{0}{
\draw[vstdhl](0,0)--(0,0.5);
\draw[vstdhl](0.5,0)--(0.5,0.5);
\draw[vstdhl](1,0)--(1,0.5);
\draw (0.75,0)--(0.75,0.5);
\filldraw [fill=white, draw=black] (-0.25,.5) rectangle (1.5,1) node[midway] { $\Tb$};
} 
\end{tikzcd}
\]
And finally, the cone of the 0 map is clearly just $x^2 \Sb^{\lambda,\lambda,\lambda}_{1,0,0} \cong x^2 \Pb^{\lambda,\lambda,\lambda}_{1,0,0} $. 
\end{ex}
\begin{ex}{$\Pb^{\lambda,\lambda}_{0,2}$} In this slightly more complicated case, we consider the iterated extension,
\[
\begin{tikzcd}[row sep=large]
&&&&x^2q^{-2}\hspace{1mm}\tikzdiagh{0}{
    \draw[vstdhl](0,0)--(0,0.5);

\draw[vstdhl](1,0)--(1,0.5);
\draw (0.33,0)--(0.33,0.5);
\draw (0.66,0)--(0.66,0.5);

\filldraw [fill=white, draw=black] (-0.25,.5) rectangle (1.25,1) node[midway] { $\Tb$};
    } 
    \ar[d,
    start anchor={[xshift=8ex]},
end anchor={[xshift=12ex]}]
    \ar[dash,
    start anchor={[xshift=8ex]},
end anchor={[xshift=12ex]}]{d}{
\tikzdiagh[scale=0.75]{4ex}{
\draw[vstdhl](0,0)--(0,0.5);
\draw[vstdhl](1,0)--(1,0.5);
\draw (0.66,0)--(0.66,0.5);
\draw (1.33,0)..controls (1.33, 0.25) and (0.33,0.25) ..(0.33,0.5);
}
}
\ar[d,
    start anchor={[xshift=-1ex]},
end anchor={[xshift=-7ex]}]
    \ar[swap,dash,
    start anchor={[xshift=-1ex]},
end anchor={[xshift=-7ex]}]{d}{
\tikzdiagh[scale=0.75]{0}{
\draw[vstdhl](0,0)--(0,0.5);
\draw[vstdhl](1,0)--(1,0.5);
\draw (0.33,0)--(0.33,0.5);
\draw (1.33,0)..controls (1.33, 0.25) and (0.66,0.25) ..(0.66,0.5);
}
}
\\
    x^2  \hspace{1mm} \tikzdiagh{0}{
    \draw[vstdhl](0,0)--(0,0.5);

\draw[vstdhl](1,0)--(1,0.5);
\draw (0.33,0)--(0.33,0.5);
\draw (0.66,0)--(0.66,0.5);

\filldraw [fill=white, draw=black] (-0.25,.5) rectangle (1.25,1) node[midway] { $\Tb$};
    }
\ar{rr}
\ar[dash]{rr}{
\tikzdiagh{0}{
\draw[vstdhl](0,0)--(0,0.5);
\draw[vstdhl](1,0)--(1,0.5);
\draw (0.33,0)--(0.33,0.5);
\draw (1.33,0)..controls (1.33, 0.25) and (0.66,0.25) ..(0.66,0.5);
}
}
& & x \hspace{1mm} \tikzdiagh{0}{
    \draw[vstdhl](0,0)--(0,0.5);

\draw(1,0)--(1,0.5);
\draw (0.33,0)--(0.33,0.5);
\draw[vstdhl] (0.66,0)--(0.66,0.5);

\filldraw [fill=white, draw=black] (-0.25,.5) rectangle (1.25,1) node[midway] { $\Tb$};
    } 
\ar[rr,"\mathbf{0\bigoplus 1}"]
& & x q^{-2}\hspace{1mm}\tikzdiagh{0}{
    \draw[vstdhl](0,0)--(0,0.5);

\draw(1,0)--(1,0.5);
\draw (0.33,0)--(0.33,0.5);
\draw[vstdhl] (0.66,0)--(0.66,0.5);

\filldraw [fill=white, draw=black] (-0.25,.5) rectangle (1.25,1) node[midway] { $\Tb$};
    } 
     \bigoplus 
    x\hspace{1mm}\tikzdiagh{0}{
    \draw[vstdhl](0,0)--(0,0.5);

\draw(1,0)--(1,0.5);
\draw (0.33,0)--(0.33,0.5);
\draw[vstdhl] (0.66,0)--(0.66,0.5);

\filldraw [fill=white, draw=black] (-0.25,.5) rectangle (1.25,1) node[midway] { $\Tb$};
    } 
\ar{rr}
\ar[dash]{rr}{
}
& & \tikzdiagh{0}{
    \draw[vstdhl](0,0)--(0,0.5);
\draw(1,0)--(1,0.5);
\draw[vstdhl] (0.33,0)--(0.33,0.5);
\draw (0.66,0)--(0.66,0.5);
\filldraw [fill=white, draw=black] (-0.25,.5) rectangle (1.25,1) node[midway] { $\Tb$};
    } 
\end{tikzcd}
\]
Where the first map is the obvious one in the complex of $\Sb^{\lambda,\lambda}_{0,2}$. That is, the cone of the first map is precisely $\Sb^{\lambda,\lambda}_{0,2}$. The cone of the second map is given by, 

\[
\begin{tikzcd}[row sep=small]
x^2q^{-2}\hspace{1mm}\tikzdiagh{0}{
    \draw[vstdhl](0,0)--(0,0.5);

\draw[vstdhl](1,0)--(1,0.5);
\draw (0.33,0)--(0.33,0.5);
\draw (0.66,0)--(0.66,0.5);

\filldraw [fill=white, draw=black] (-0.25,.5) rectangle (1.25,1) node[midway] { $\Tb$};
    } 
    \ar{rr}
\ar[dash]{rr}{
\tikzdiagh[scale=0.75]{0}{
\draw[vstdhl](0,0)--(0,0.5);
\draw[vstdhl](1,0)--(1,0.5);
\draw (0.33,0)--(0.33,0.5);
\draw (1.33,0)..controls (1.33, 0.25) and (0.66,0.25) ..(0.66,0.5);
}
}
\ar{rrdd}
\ar[dash]{rrdd}{
\tikzdiagh[scale=0.75]{0}{
\draw[vstdhl](0,0)--(0,0.5);
\draw[vstdhl](1,0)--(1,0.5);
\draw (0.66,0)--(0.66,0.5);
\draw (1.33,0)..controls (1.33, 0.25) and (0.33,0.25) ..(0.33,0.5);
}
}
    &&x q^{-2}\hspace{1mm}\tikzdiagh{0}{
    \draw[vstdhl](0,0)--(0,0.5);

\draw(1,0)--(1,0.5);
\draw (0.33,0)--(0.33,0.5);
\draw[vstdhl] (0.66,0)--(0.66,0.5);

\filldraw [fill=white, draw=black] (-0.25,.5) rectangle (1.25,1) node[midway] { $\Tb$};
    } \\\bigoplus &&\bigoplus \\
   x\hspace{1mm}\tikzdiagh{0}{
    \draw[vstdhl](0,0)--(0,0.5);

\draw(1,0)--(1,0.5);
\draw (0.33,0)--(0.33,0.5);
\draw[vstdhl] (0.66,0)--(0.66,0.5);

\filldraw [fill=white, draw=black] (-0.25,.5) rectangle (1.25,1) node[midway] { $\Tb$};
    }  
    \ar[rr,"\mathbf{1}"]
    &&x\hspace{1mm}\tikzdiagh{0}{
    \draw[vstdhl](0,0)--(0,0.5);

\draw(1,0)--(1,0.5);
\draw (0.33,0)--(0.33,0.5);
\draw[vstdhl] (0.66,0)--(0.66,0.5);

\filldraw [fill=white, draw=black] (-0.25,.5) rectangle (1.25,1) node[midway] { $\Tb$};
    } 
\end{tikzcd} \cong \begin{tikzcd}
x^2q^{-2}\hspace{1mm}\tikzdiagh{0}{
    \draw[vstdhl](0,0)--(0,0.5);

\draw[vstdhl](1,0)--(1,0.5);
\draw (0.33,0)--(0.33,0.5);
\draw (0.66,0)--(0.66,0.5);

\filldraw [fill=white, draw=black] (-0.25,.5) rectangle (1.25,1) node[midway] { $\Tb$};
    } 
    \ar{rr}
\ar[dash]{rr}{
\tikzdiagh[scale=0.75]{0}{
\draw[vstdhl](0,0)--(0,0.5);
\draw[vstdhl](1,0)--(1,0.5);
\draw (0.33,0)--(0.33,0.5);
\draw (1.33,0)..controls (1.33, 0.25) and (0.66,0.25) ..(0.66,0.5);
}
}
    &&x q^{-2}\hspace{1mm}\tikzdiagh{0}{
    \draw[vstdhl](0,0)--(0,0.5);

\draw(1,0)--(1,0.5);
\draw (0.33,0)--(0.33,0.5);
\draw[vstdhl] (0.66,0)--(0.66,0.5);

\filldraw [fill=white, draw=black] (-0.25,.5) rectangle (1.25,1) node[midway] { $\Tb$};
    }
\end{tikzcd}
\]

Where by ``$\cong$", we mean quasi-isomorphism. In other words, the cone of the second map is quasi-isomorphic to $xq^{-2}\Sb^{\lambda,\lambda}_{1,1}$. The cone of the third map is clearly just $x \Sb^{\lambda,\lambda}_{1,1}$ again. The fourth map, which we omitted, is just the $0$ morphism, and its cone is simply $x^2 \Sb^{\lambda,\lambda}_{2,0}$. 
\end{ex}
In general, let us describe what standards will appear in the iterated extensions of the projective $\Pb^{\underline{\mu}}_\rho$. Let $\Phi_\rho$ be the set of permutations of black strands at the top of $1_\rho$ generated by moving black strands to blocks to their left and such that the permutations within blocks at the bottom are minimal. 
\begin{ex}
Suppose $\underline{\mu} = (\lambda, \lambda)$ and $\rho = (b_1, 2)$. Then, the elements in $\Phi_\rho$ are given by, 
\[
\begin{tikzcd}
\tikzdiagh[yscale=1.5]{0}{
\draw[vstdhl](-0.33,0)--(-0.33,0.5);
\draw[vstdhl](1,0)--(1,0.5);
\draw (0.4,0)--(0.4,0.5);
\draw (-0.1,0)--(-0.1,0.5);
\draw (1.66,0)--(1.66,0.5);
\draw (1.33,0)--(1.33,0.5);
\node at(0.18,0.25){\small $\cdots$};
\tikzbrace{-0.1}{0.4}{.1}{$b_1$};
} &
\tikzdiagh[yscale=1.5]{0}{
\draw[vstdhl](-0.33,0)--(-0.33,0.5);
\draw[vstdhl](1,0)--(1,0.5);
\draw (0.4,0)--(0.4,0.5);
\draw (-0.1,0)--(-0.1,0.5);
\draw (1.66,0)--(1.66,0.5);
\draw (1.33,0)..controls (1.33, 0.25) and (0.66,0.25) ..(0.66,0.5);
\node at(0.18,0.25){\small $\cdots$};
\tikzbrace{-0.1}{0.4}{.1}{$b_1$};
}\\ 
\tikzdiagh[yscale=1.5]{0}{
\draw[vstdhl](-0.33,0)--(-0.33,0.5);
\draw[vstdhl](1,0)--(1,0.5);
\draw (0.4,0)--(0.4,0.5);
\draw (-0.1,0)--(-0.1,0.5);
\draw (1.33,0)--(1.33,0.5);
\draw (1.66,0)..controls (1.66, 0.25) and (0.66,0.25) ..(0.66,0.5);
\node at(0.18,0.25){\small $\cdots$};
\tikzbrace{-0.1}{0.4}{.1}{$b_1$};
}
&
\tikzdiagh[yscale=1.5]{0}{
\draw[vstdhl](-0.33,0)--(-0.33,0.5);
\draw[vstdhl](1,0)--(1.33,0.5);
\draw (0.4,0)--(0.4,0.5);
\draw (-0.1,0)--(-0.1,0.5);
\draw (1.66,0).. controls (1.66,.25) and (1,.25)..(1,.5);
\draw (1.33,0)..controls (1.33, 0.25) and (0.66,0.25) ..(0.66,0.5);
\node at(0.18,0.25){\small $\cdots$};
\tikzbrace{-0.1}{0.4}{.1}{$b_1$};
}
\end{tikzcd}
\]
Note that for $b_1=0$, the top of the diagrams in $\Phi$ are precisely the standards that appear in the iterated extension of $\Pb^{\lambda, \lambda}_{0,2}$. 
\end{ex}
Indeed, at the Grothendieck group level, we have, 
\begin{lem} Let $x_\phi$ denote the string diagram associated to a permuation $\phi\in \Phi_\rho$. Additionally, let $\rho_\phi$ denote the indempotent at the top of the diagram. Then, we have,
    $$[\Pb^{\underline{\mu}}_\rho] = \sum_{\phi\in\Phi_\rho} q^{deg \hspace{1mm} x_\phi} [\Sb^{\underline{\mu}}_{\rho_\phi}]$$
\end{lem}
\begin{proof}
    This follows from Theorem 6.14 and Proposition 7.18 of \cite{Na} together with the same argument for the analogous statement in \cite{Web}.
\end{proof}
We wish to categorify this relation through the aforementioned iterated extension. 
Notice we can put a preorder on the elements of $\Phi_\rho$, just as in \cite{Na} and \cite{Web}. That is, for indempotents $\rho$ and $\rho'$, we say $\rho \leq \rho'$ if one can acquire $\rho$ my moving black strands of $\rho'$ to blocks to their left. Then, the preorder on $\Phi_\rho$ is given by preordering the indempotents at the top of the diagrams. The following Proposition works for any total order that respects this preorder. 

\begin{prop}\label{prop:it.ext.proj}
    Any projective $\Pb^{\underline{\mu}}_{\rho}$ has an iterated extension by (shifted) standards of the form, 
    \[
    \begin{tikzcd}[column sep =small]
        0=Q_0 \ar[rr,"f_1"] && Q_1 \ar[rr,"f_2"] && Q_2 & \cdots & Q_{r-1} \ar[rr,"f_r"] && Q_r = \Pb^{\underline{\mu}}_{\rho} \\
        &Cone(f_1)\ar[ru,leftarrow] \ar[lu]&&Cone(f_2)\ar[ru,leftarrow] \ar[lu]&&\cdots && Cone(f_r)\ar[ru,leftarrow] \ar[lu]
    \end{tikzcd}
    \]
Where $r = |\Phi_\rho|$ and each consecutive cone is a standard module, 
$$Cone(f_i) \cong q^{deg \hspace{1mm}x_\phi}\Sb^{\underline{\mu}}_{\rho_\phi}$$
for some $\phi \in \Phi_\rho$ that fits into the short exact sequence,
$$0\rightarrow Q_i\rightarrow Cone(f_i) \rightarrow Q_{i-1}[1] \rightarrow 0$$
\end{prop}
\begin{proof}
Starting with $Q_r = \Pb^{\underline{\mu}}_\rho$, we can simply point to the definition of $\Sb^{\underline{\mu}}_\rho$ in the previous section in order to see that there is a complex $Q_{r-1}$ together with a map $f_r$ such that $Cone(Q_{r-1}\overset{f_r}{\rightarrow} Q_r) = \Sb^{\underline{\mu}}_\rho$. We are assuming there is some total order on the elements of $\Phi_\rho$ that respects the preorder mentioned above. Therefore, we pick the next indempotent $\rho_{\phi_1}$ in the iterated extension. The indempotent $\rho_{\phi_1}$ must show up in $0$th homological degree of $Q_{r-1}$ since the total order respects the preorder. Therefore, we can identify some choice of strand $b_1 \in J_\rho$ corresponding to the projective $\Pb_{\rho_{\phi_1}}^{\underline{\mu}}$. We consider the subcomplex of $Q_{r-1}$, $\tilde{Q}_{r-1}$, consisting of all the summands with $b_1 \in \mathbf{j}$ (and we keep the same maps). In other words, 
$$\tilde{Q}_{r-1} =\underset{b_1\in \mathbf{j}}{\bigoplus_{\mathbf{j}\subset J_\rho} }q^{\#_{\mathbf{j}}}\Pb^{\underline{\mu}}_{\rho(\mathbf{j})} [|\mathbf{j}|] $$
By the iterated cone construction of standard modules in Section 7.1.2 of \cite{Na} (which is equivalent to the construction given here), there is a complex $\tilde{Q}_{r-2}$, and map $\tilde{f}_{r-1}$ such that, 
$$Cone(\tilde{Q}_{r-2} \overset{\tilde{f}_{r-1}}{\rightarrow} \tilde{Q}_{r-1}) = q^{deg\hspace{1mm} x_{\phi_1}}\Sb_{\rho_{\phi_1}}^{\underline{\mu}}$$
It also follows then, that there exists $Q_{r-2}$ such that $Cone(Q_{r-2} \overset{f_{r-1}}{\rightarrow} Q_{r-1}) \cong q^{deg\hspace{1mm} x_{\phi_1}}\Sb_{\rho_{\phi_1}}^{\underline{\mu}}$, where $f_{r-1} = \tilde{f}_{r-1} \oplus \iota$ with $\iota$ the inclusion map of the remainder complex of $Q_{r-1}$.\\
One then repeats this process, which must terminate since with each step, we are sliding black strands further and further down to the left, until we are left $\rho_{\phi_{min}}$, upon which, we are finished since $\Pb^{\underline{\mu}}_{\rho_{\phi_{min}}}$ is already a standard module. 
\end{proof}

\begin{cor}\label{cor:gdimPandS}
    $$gdim_{x,q}\Pb_\rho ^{\underline{\mu}} = \sum_{\phi\in \Phi_\rho} q^{deg\hspace{1mm}x_\phi} gdim_{x,q} \Sb^{\underline{\mu}}_{\rho_\phi}$$
\end{cor}

\section{Braiding}\label{sec:braiding}
Assuming the results of \cite{Na}, for a choice of representations $\underline{\mu} = (\mu_1, \mu_2, ..., \mu_r)$, we can assign the dgKLRW algebra $\Tb^{\underline{\mu}}$. Here, each $\mu_i$ denotes a choice of finite-dimensional representation of $\mathfrak{sl}_2$, taking values in $\mathbb{N}_{>0}$, or it may denote a generic complex highest weight $\lambda_i$ of a Verma module. In this section we will always assume $\mu_1 =\lambda$. For a choice of $w\in S_{r-1}$, we can define a braiding bimodule $\Rb_w$ of these algebras as follows, 
\begin{defn}
    For $w\in S_{r-1}$, we let $\Rb^{\underline{\mu}}_w$ be the $\Tb^{w\cdot\underline{\mu}}-\Tb^{\underline{\mu}}$-bimodule consisting of string diagrams where the colored strands in the set $\{\mu_2,...,\mu_r\}$, when read from bottom to top, trace out a reduced string diagram of the permutation $w$. We subject these diagrams to the local relations in Definition \ref{def:dgKLRWalgebra} as well as the new relations (and their mirror images), 
    \begin{align}
        \label{eq:RmoveR3p}
        \tikzdiagh[scale=2]{-1ex}{ 
    \draw (0.5,0)..controls (0.5,0.25) and (0,0.25) .. (0,0.5);
    \draw (0,0.5)..controls (0,0.75) and (0.5, 0.75).. (0.5, 1);
      \draw[pstdhl] (0,0) -- (1,1)  ;
      \draw[pstdhl] (1,0) --(0,1);
      } 
      =
      \tikzdiagh[scale=2]{-1ex}{ 
    \draw (0.5,0)..controls (0.5,0.25) and (1,0.25) .. (1,0.5);
    \draw (1,0.5)..controls (1,0.75) and (0.5, 0.75).. (0.5, 1);
      \draw[pstdhl] (0,0) -- (1,1)  ;
      \draw[pstdhl] (1,0) --(0,1);
      } 
    \end{align}
    
    \begin{align}
        \label{eq:RmoveR3pp}
        \tikzdiagh[scale=2]{-1ex}{ 
    \draw (1.25,0)..controls (1.25,0.75) and (-0.25,0.75) .. (-0.25,1);
      \draw[pstdhl] (0,0) -- (1,1)  ;
      \draw[pstdhl] (1,0) --(0,1);
      } 
      =
      \tikzdiagh[scale=2]{-1ex}{ 
    \draw (1.25,0)..controls (1.25,0.25) and (-0.25,0.25) .. (-0.25,1);
      \draw[pstdhl] (0,0) -- (1,1)  ;
      \draw[pstdhl] (1,0) --(0,1);
      } 
      \end{align}
      \begin{align}
        \label{eq:RmoveR3}
        \tikzdiagh[scale=2]{-1ex}{ 
    \draw[pstdhl] (0,0.5)..controls (0,0.75) and (0.5, 0.75).. (0.5, 1);
      \draw[pstdhl] (0,0) -- (1,1)  ;
      \draw[pstdhl] (1,0) --(0,1);
      \draw[pstdhl] (0.5,0)..controls (0.5,0.25) and (0,0.25) .. (0,0.5);
      } 
      =
      \tikzdiagh[scale=2]{-1ex}{ 
    \draw[pstdhl] (0.5,0)..controls (0.5,0.25) and (1,0.25) .. (1,0.5);
      \draw[pstdhl] (0,0) -- (1,1)  ;
      \draw[pstdhl] (1,0) --(0,1);
      \draw[pstdhl] (1,0.5)..controls (1,0.75) and (0.5, 0.75).. (0.5, 1);
      } 
      \end{align}

The left action of $\Tb^{w\cdot \underline{\mu}}$ is given by gluing diagrams at the top and the right action of $\Tb^{ \underline{\mu}}$ by gluing diagrams at the bottom. Furthermore, we set the degree of the colored crossing to, 
\begin{align}
    deg \tikzdiagh{-1ex}{ 
      \draw[vstdhl] (0,0) node[below]{$\lambda$} -- (1,1);
      \draw[vstdhl] (1,0)node[below]{$\lambda$} --(0,1);
      } 
      =
      0
\end{align}
if both strands are colored by $\lambda$, and otherwise, 
\begin{align}
    deg \tikzdiagh{-1ex}{ 
      \draw[pstdhl] (0,0) node[below]{$\mu_i$} -- (1,1);
      \draw[pstdhl] (1,0)node[below]{$\mu_j$} --(0,1);
      } 
      =
      q^{-\mu_i \mu_j \frac{1}{2} }
\end{align}
      
\end{defn}

\subsection{Basis}
Recall that the $\kk$-basis of $\Tb^{\underline{\mu}}$ can be described as follows. Let $\rho = (b_1,b_2,...,b_r)$ and $\kappa = (b'_1, b'_2,...,b_r')$ denote the data defining the idempotents at the bottom and top of a string diagram. For $b = \sum_i b_i = \sum_i b_i'$ and $w\in S_b$, we can pick a diagram $\psi_w \in 1_\kappa \Tb^{\underline{\mu}} 1_\rho$ such that,
\begin{itemize}
    \item the permutation on black strands from bottom to top read out a reduced word for $w$ (note that the identity word corresponds to drawing straight black lines from the leftmost top strand to the leftmost bottom strand, and similarly with the strand to their right, and so on... )
    \item no pair of colored and black strand cross twice
    \item $\psi_w$ is a left-adjusted diagram
\end{itemize}

Now, for each such $\psi_w $, we further specify $\underline{l} \in \Z^b_2$ and $\underline{a} \in \mathbb{N}^b $, and to the triple $(w,\underline{l},\underline{a})$ we assign the diagram $\psi_{(w,\underline{l},\underline{a})}$ defined by modifying $\psi_w$ via, 
\begin{itemize}
    \item for all $1\leq i\leq b$, add $a_i$ black dots on the $i$-th strand at the top (counted at the top)
    \item for all $1\leq i\leq b$, if $l_i =1$, nail the $i$-th black strand on the $\mu_1$ colored strand by pulling it from its leftmost position. 
\end{itemize}
Let $B^\kappa_\rho(\Tb)$ be the set of all such diagrams. 

\begin{thm}{\cite{Na}}
\label{thm:Tbasis}
    The set $B^\kappa_\rho(\Tb)$ is a $\kk$-basis for $1_\kappa \Tb^{\underline{\mu}} 1_\rho$ treated as a $\kk$-module
\end{thm}

For the braiding bimodules $\Rb_w $, we claim a very similar set forms a basis for it. Namely, we define $B^\kappa_\rho(\Rb_w)$ as follows. For $v\in S_b$, we can pick diagrams $\psi_v \in 1_\kappa \Rb_w^{\underline{\mu}} 1_\rho$, where the colored strands read from bottom to top follow a reduced expression of $w$, and the black strands a reduced word for $v$ (where no pair of colored or black strands cross twice and such that it is left-adjusted). Given $\underline{l},\underline{a}$ as above, we then add black dots and nails according to the same procedure as before. The set of all such diagrams for fixed $\kappa$ and $\rho$, we call $B^\kappa_\rho(\Rb_w)$. 
\begin{ex}
    Take, for instance, $\underline{\mu} = (\lambda, \lambda, \lambda)$, $\rho =(0,1,1)$, $\kappa = (1,1,0)$, and $w$ being an adjacent transposition of the 2nd and 3rd strands. Then set of possible $\psi_v$ has two elements:
    \[ \tikzdiagh{-1ex}{ 
    \draw (1.5,0)..controls (1.5,0.2) and (-0.5, .2) .. (-0.5, 1);
    \draw (0.5,0)..controls (0.5,0.25) and (-0.25, 0.25) ..(-0.25, .5);
    \draw (-0.25, .5)..controls (-0.25, 0.75) and (0.5,0.75) .. (0.5,1);
    \draw[vstdhl] (-1,0) node[below]{$\lambda$} -- (-1,1);
      \draw[vstdhl] (0,0) node[below]{$\lambda$} -- (1,1);
      \draw[vstdhl] (1,0)node[below]{$\lambda$} --(0,1);
      } \] 
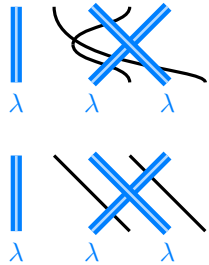
      \[ \tikzdiagh{-1ex}{ 
    \draw (1.5,0)--(0.5,1);
    \draw (0.5,0)--(-0.5,1);
    \draw[vstdhl] (-1,0) node[below]{$\lambda$} -- (-1,1);
      \draw[vstdhl] (0,0) node[below]{$\lambda$} -- (1,1);
      \draw[vstdhl] (1,0)node[below]{$\lambda$} --(0,1);
      } \] 
\end{ex}
Before proving this is a basis, we will define a useful map 
\begin{prop}\label{prop:Rbasis}
    The set $B^\kappa_\rho(\Rb_w)$ is a $\kk$-basis for $1_\kappa \Rb_w^{\underline{\mu}} 1_\rho$ treated as a $\kk$-module
\end{prop}
\begin{proof}
    The proof of this is very similar to the arguments in Section 5 of \cite{Na}. 
    First, we show $B^\kappa_\rho(\Rb_w)$ spans $1_\kappa \Rb_w^{\underline{\mu}} 1_\rho$. We wish to show any $d \in 1_\kappa \Rb_w^{\underline{\mu}} 1_\rho$ may be written as a linear combination elements of $B^\kappa_\rho(\Rb_w)$. Notice that by the KLR relations on black strands and dot sliding relation on colored strands, any diagram in $1_\kappa \Rb_w^{\underline{\mu}} 1_\rho$ is equivalent to a linear combination of diagrams where all black dots are at the top of the diagram, so we may assume all black dots are at the top. \\ \\
    We proceed by induction on the number of crossings. If $d$ has the minimum number of crossings allowed, then it must be that $d = \psi_{(1, \underline{l}, \underline{a})}$ for some choice of nails and black dots. In the induction step, it is sufficient to show $d$ is equivalent to some $\psi_{(v, \underline{l}, \underline{a})}$ modulo sums of diagrams with lower crossing number. Note that using \ref{eq:RmoveR3pp} together with similar arguments as in Section 5.1 of \cite{Na}, we may slide all nails of any black strand of $d$ such that they are positioned at the leftmost position of the black strand in question at the cost of introducing diagrams of lower crossing number. Therefore, we are reduced to the case that $d$ has no black dots or nails.\\ \\
    By similar arguments as Lemma 4.10 of \cite{Web}, we may also assume $d$ has no bigons and that applying braid group relations on colored/ black strands of $d$ does not change the class of $d$ modulo diagrams with fewer crossings. We let $v'\in S_{b+r}$ denote the permutation of black strands and colored strands of $d$ (read from bottom to top). The expression we read out for $v'$ is necessarily already a reduced word due the lack of bigons. We may then go from this expression for $v'$ to its left-adjusted representative by applying a sequence of braid relations, each of which does not change the class of $d$ modulo diagrams with fewer crossings. Calling the induced permutation on the black strands $v \in S_b$, the class of $d$ is necessarily $\psi_v$.
    \\  \\
    To show linear independence, we essentially repeat Section 5.2 of \cite{Na}. As such, we mention only the parts not contained there. Fixing the number of black strands to be $b$, we wish to define a faithful polynomial action of the perturbed algebra $\Rb^{\underline{\mu}}_{w,b} (\delta)$ on the free module $\oplus_\rho Pol_b \varepsilon_\rho$ over the polynomial ring $Pol_b = \kk [\delta] [x_1,...,x_b]\otimes \bigwedge (\omega_1,...,\omega_b)$. To the local pieces of any diagram in $1_\kappa \Rb^{\underline{\mu}}_{w,b} (\delta)1_\rho$ we assign the actions described in \cite{Na} as well as:
  \begin{align}
     \tikzdiagh{-1ex}{ 
     \node at(-0.5, 0.5){\small$\dots$};
     \node at(1.5, 0.5){\small $\dots$};
      \draw[pstdhl] (0,0) node[below]{$\mu_i$} -- (1,1);
      \draw[pstdhl] (1,0)node[below]{$\mu_j$} --(0,1);
      } \cdot f \varepsilon_\rho 
      =
      f \varepsilon_\kappa
\end{align}
We check this is well-defined on the local relations \ref{eq:RmoveR3},\ref{eq:RmoveR3p}, and \ref{eq:RmoveR3pp}. \ref{eq:RmoveR3} and \ref{eq:RmoveR3pp} are clear. For \ref{eq:RmoveR3p}, we use the shorthand $\delta = x_i^\lambda$ to write, 
\begin{align}
        \tikzdiagh[scale=1]{-1ex}{ 
        \node at(-0.5, 0.5){\small$\dots$};
     \node at(1.5, 0.5){\small $\dots$};
    \draw (0.5,0)..controls (0.5,0.25) and (0,0.25) .. (0,0.5);
    \draw (0,0.5)..controls (0,0.75) and (0.5, 0.75).. (0.5, 1);
      \draw[pstdhl] (0,0) node[below]{$\mu_k$} -- (1,1)  ;
      \draw[pstdhl] (1,0) node[below]{$\mu_l$} --(0,1);
      } \cdot f \varepsilon_\rho
      = x_i^{\mu_k} f \varepsilon_\kappa = 
      \tikzdiagh[scale=1]{-1ex}{ 
      \node at(-0.5, 0.5){\small$\dots$};
     \node at(1.5, 0.5){\small $\dots$};
    \draw (0.5,0)..controls (0.5,0.25) and (1,0.25) .. (1,0.5);
    \draw (1,0.5)..controls (1,0.75) and (0.5, 0.75).. (0.5, 1);
      \draw[pstdhl] (0,0) node[below]{$\mu_k$} -- (1,1)  ;
      \draw[pstdhl] (1,0) node[below]{$\mu_l$} --(0,1);
      } \cdot f \varepsilon_\rho
\end{align}
Thus, the polynomial action is well-defined. The rest of the argument for linear independence closely follows \cite{Na}.

\end{proof}

\begin{cor} \label{cor:Rtensor}
    If $l(w w') = l(w)+l(w')$, then $\Rb_w \otimes_\Tb \Rb_{w'} \cong \Rb_{w w'}$ as $\Tb^{ww'\cdot\underline{\mu}}-\Tb^{\underline{\mu}}$-bimodules
\end{cor}
\begin{proof}
    There is a obvious morphism $\Rb_w \otimes_\Tb \Rb_{w'} \longrightarrow \Rb_{w w'}$ given by composition. This can be shown to be both surjective and injective using Proposition \ref{prop:Rbasis}. 
\end{proof}

\subsection{Braid Group Action}
For each $w \in S_r$, we can now define the braiding functors $\Bb_w$, which will be the categorifications of the product of R-matrices. 
\begin{defn}
    We let $\Bb_w$ denote the functor $ \Bb_w: \mathcal{D}_{dg}(\Tb^{\underline{\mu}}) \rightarrow \mathcal{D}_{dg}(\Tb^{w\cdot\underline{\mu}})$ defined by, 
    $$\Bb_w(-) := \Rb_w \otimes^L_\Tb-$$
    In the case that $w$ is an adjacent transposition, say $s_i: (...,i,i+1,...) \mapsto (...,i+1,i,...)$, we will just say $\Bb_{s_i} = \Bb_i$.
\end{defn}
As is appropriate for objects categorifying R-matrices, these functors commute with 1-morphisms. 
\begin{prop}
    The functors $\Bb_i$ commute with 1-morphisms. More precisely, we have natural isomorphisms $\Bb_i \circ \mathcal{F} \cong \mathcal{F} \circ \Bb_i$ and $\Bb_i \circ \mathcal{E} \cong \mathcal{E} \circ \Bb_i$. 
\end{prop}
\begin{proof}
Recall from Section \ref{sec: prelim} that the $\mathcal{F}$ and $\mathcal{E}$ functors are given by tensor products with right (or left, respectively) projective modules, which we shall refer to as $\beta_\mathcal{F}$ and $\beta_\mathcal{E}$. Let us consider $\mathcal{F}$, and the case for $\mathcal{E}$ will follow similarly. Then, it is sufficient to show $\beta_\mathcal{F} \otimes_\Tb \Rb_i \cong \Rb_i \otimes_\Tb \beta_\mathcal{F}$. \\
First, notice that $\Rb_i \otimes_\Tb \beta_\mathcal{F}$ is the space of diagrams of $\Rb_i $ where the rightmost strand is always a black strand. Now, $\beta_\mathcal{F} \otimes_\Tb \Rb_i $ is a subspace of this module, where the rightmost strand is black \textit{and} it must have been pulled to the right above the crossing (this is because $\Rb_i$ does not interact with the rightmost black strand in the tensor product \eqref{eq:F_functor}). So, we can consider this inclusion as an injective map $\beta_\mathcal{F} \otimes_\Tb \Rb_i \rightarrow \Rb_i \otimes_\Tb \beta_\mathcal{F}$. However, by \eqref{eq:RmoveR3pp}, we can always slide this strand above and below the colored crossing, thereby showing the map is also surjective. 
\end{proof}
It will be useful for us to study the projective resolutions of these braiding bimodules in order to prove the braiding functors provide a strong categorical group action on our category. In fact, we will now explicitly describe the projective resolution of $\Bb_i(P^{\underline{\mu}}_\rho)$ for $P^{\underline{\mu}}_\rho$ an arbitrary projective module. \\ \\
We wish to find all summands and maps of a resolution, 
$$0\rightarrow\cdots \rightarrow (\boldsymbol{p}\Rb^i_\rho)_1 \overset{p_1}{\rightarrow }(\boldsymbol{p}\Rb^i_\rho)_0 \overset{p_0}{\rightarrow }\Bb_i(P^{\underline{\mu}}_\rho)\rightarrow 0$$
We shall denote the relevant projectives in the resolution only by their $(i-1)$th-$(i+1)$th strands. For instance, if $\rho$ specifies the tuple $(b_1,b_2,...,b_r)$ denoting the configuration of black strands in the idempotent $1_\rho$, then $P^{\underline{\mu}}_\rho = P^{\underline{\mu}}_{\cdots,b_{i-1},b_i,b_{i+1},\cdots } = P^{\underline{\mu}}_{b_{i-1},b_i,b_{i+1} }$. We begin by setting, 
$$(\boldsymbol{p}\Rb^i_\rho)_0 = \bigoplus_{I\in C(b_i)} q^{f(I)}P_{b_{i-1}+|I|, 0, b_{i+1}+b_i-|I|}$$
Where $C(b) = \bigcup_{b\geq k\geq 0} C(b,k)$ for $C(b,k)$ the set of length $k$ subsets of $(1,...,b)$. Furthermore, we set the map, 
$$q^{f(I)}P_{b_{i-1}+|I|, 0, b_{i+1}+b_i-|I|} \overset{p_0^I}{\longrightarrow}\Bb_i(P^{\underline{\mu}}_\rho)  $$
to be the composition with an element of the braiding bimodule, which we now describe. \\ \\
For $I \subset\{1,2,..,b\}$, let \begin{align}\label{eq:omegas}\omega^I_{b} \in NH_{b}\end{align} be an element of the nilHecke algebra on $b$ strands defined by pulling every strand in $I$ to the left and keeping the remaining strands on the right, without braiding any strands in $I$ or its complement $I^c$. For instance, if $b=4$ and $|I|=2$, then, 
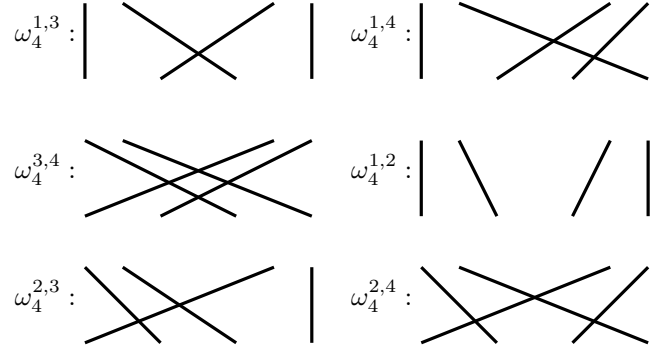
\[ \omega^{1,3}_{4}: \tikzdiagh{-1ex}{
    \draw (-1,0)-- (-1,1);
      \draw (0,0)  -- (1.5,1);
      \draw (1,0) --(-0.5,1);
      \draw (2,0)--(2,1); } \hspace{5mm}
      \omega^{1,4}_{4}: \tikzdiagh{-1ex}{
    \draw (-1,0)-- (-1,1);
      \draw (0,0)  -- (1.5,1);
      \draw (1,0) --(2,1);
      \draw (2,0)--(-0.5,1);}
      \]
      
\[ \omega^{3,4}_{4}: \tikzdiagh{-1ex}{
    \draw (-1,0)-- (1.5,1);
      \draw (0,0)  -- (2,1);
      \draw (1,0) --(-1,1);
      \draw (2,0)--(-0.5,1); }
      \hspace{5mm} 
      \omega^{1,2}_{4}: \tikzdiagh{-1ex}{
    \draw (-1,0)-- (-1,1);
      \draw (0,0)  -- (-0.5,1);
      \draw (1,0) --(1.5,1);
      \draw (2,0)--(2,1); }
      \]

\[ \omega^{2,3}_{4}: \tikzdiagh{-1ex}{
    \draw (-1,0)-- (1.5,1);
      \draw (0,0)  -- (-1,1);
      \draw (1,0) --(-0.5,1);
      \draw (2,0)--(2,1);} 
      \hspace{5mm}
       \omega^{2,4}_{4}: \tikzdiagh{-1ex}{
    \draw (-1,0)-- (1.5,1);
      \draw (0,0)  -- (-1,1);
      \draw (1,0) --(2,1);
      \draw (2,0)--(-0.5,1);} 
      \]
The maps $p_0^I$ are then composition with the element, 
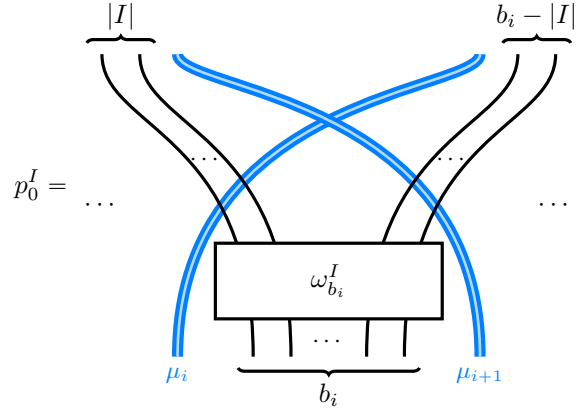
\[p_0^I = \tikzdiagh[scale=2]{0}{
	\draw[vstdhl] (0,0) node[below]{\small $\mu_{i}$}..controls (0,1.75) and (2,1.75) .. (2,2);
    \draw[vstdhl] (2,0) node[below]{\small $\mu_{i+1}$} ..controls (2,1.75) and (0,1.75) ..  (0,2);
	\draw	 	(0.5,0)..controls (0.5,1.5) and (-0.5,1.5).. (-0.5,2);
    \draw	 	(0.75,0)..controls (0.75,1.5) and (-0.25,1.5).. (-0.25,2);
    \node at(-0.5,1){$\cdots$};
    \node at(2.5,1){$\cdots$};
	\node at(1,.1){\small $\cdots$};
    \node at(1.82,1.3){\small $\cdots$};
    \node at(0.19,1.3){\small $\cdots$};
    \tikzbrace{.5}{1.5}{.1}{$b_i$};
    \tikzbraceop{-0.5}{-0.25}{1.9}{$|I|$};
    \tikzbraceop{2.25}{2.5}{1.9}{$b_i-|I|$};
	\draw (1.25,0)..controls (1.25,1.5) and (2.25,1.5) ..(2.25, 2);
    \draw (1.5,0)..controls (1.5,1.5) and (2.5,1.5) ..(2.5, 2);
	\filldraw [fill=white, draw=black] (.25,.25) rectangle (1.75,.75) node[midway] { $\omega^{I}_{b_i}$};
}
\]
Furthermore, we set $q^{f(I)} = deg(p^I_0) = q^{\mu_i |I|} q^{\mu_{i+1}(b_i - |I|) }\cdot deg\hspace{1mm}\omega_{b_i}^I$. The total map $p_0 = \sum_I p_0^I$ is easily seen to be surjective by using Prop. \ref{prop:Rbasis}. \\
Now, for constructing the rest of the resolution, it will be useful to keep track of two sets for each projective: $\jj_L$ and $\jj_R$. These will keep track of which strands are to the right of the crossing and which strands are to its left. We may also keep track of the black strands immediately after the $i$-th colored strand by using the set $\jj$. We use the shorthand $\underline{\jj} = (\jj_L,\jj,\jj_R)$. For instance, in the 0th homological degree constructed above, given $I$, $\jj_L = I$, $\jj = \{\emptyset\}$ , and $\jj_R = I^c$. Now, we can more generally consider the complex, 

$$\bigoplus_{\underline{\jj} =(\jj_L,\jj,\jj_R)} q^{\#_{\underline{\jj}}} P_{b_{i-1}+|\jj_L|, \jj , b_{i+1}+|\jj_R|} [|\jj|]\hspace{10mm}, \hspace{10mm} q^{\#_{\underline{\jj}}}:=\prod_{s\in \jj}q^{\mu_i -2(s-1)} \cdot q^{f(\jj_L)}$$

The direct sum above ranges over all $\underline{\jj} =(\jj_L,\jj,\jj_R)$ such that $\jj_L\cup \jj\cup\jj_R = \{1,2,...,b_i\}$ and such that $(\jj_L,\jj,\jj_R)$ are pairwise disjoint. Simply put, we are dividing up the $b_i$ black strands into the ones at the center, the ones at the left, and the ones at the right. As we increase the homological grading, we bring more and more strands into the middle, until $\jj_L = \jj_R = \{\emptyset\}$, at which point the homological grading is $|\jj| =b_i$ and there is only the projective module $P_{b_{i-1},b_i,b_{i+1}}$ that we started with (shifted by some $q$ and $x$ factors). There are also non-trivial differentials, $p_n$, which we now describe. \\
We say $\underline{\jj'} \rightharpoonup \underline{\jj}$ whenever $\jj \subset \jj' $ such that $|\jj| = |\jj'|-1$ and if $\jj_L' \subset\jj_L$ and $\jj_R' \subset\jj_R$. In this case, we may assume $\jj' = \jj \cup\{b'\}$ for some $b'\in \{1,...,b_i\}$. It also follows that either $|\jj_R'| = |\jj_R|-1$ and $|\jj_L'| = |\jj_L|$ OR that $|\jj_L| = |\jj_L|-1$ and $|\jj_R'| = |\jj_R|$. Let's assume the former as the two cases are nearly identical. Then, we define a map,
$$q^{\#_{\underline{\jj'}}}P_{\underline{\jj}'} \overset{\tau^R_{\underline{\jj'}, \underline{\jj}}}{\longrightarrow}q^{\#_{\underline{\jj}}} P_{\underline{\jj}}$$
through composition with the element, 
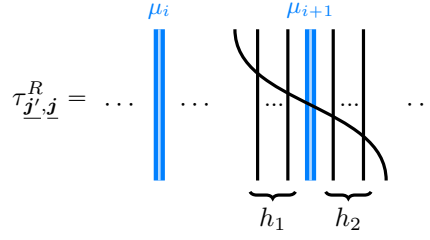
\[\tau^R_{\underline{\jj'}, \underline{\jj}} = \tikzdiagh[scale=2]{0}{
\draw[vstdhl] (0,0) --(0,1)node[above]{\small $\mu_{i}$};
\draw[vstdhl] (1,0) --(1,1)node[above]{\small $\mu_{i+1}$};
\draw (1.5,0)..controls (1.5,0.5) and (0.5,0.5)..(0.5,1);
\draw (0.65,0)--(0.65,1);
\draw (0.85,0)--(0.85,1);
\node at(0.76,.5){\small $...$};
\draw (1.15,0)--(1.15,1);
\draw (1.35,0)--(1.35,1);
\node at(1.26,.5){\small $...$};
\tikzbrace{.7}{0.8}{.1}{$h_1$};
\tikzbrace{1.2}{1.3}{.1}{$h_2$};
\node at(0.25,.5){$\cdots$};
\node at(-0.25,.5){$\cdots$};
\node at(1.75,.5){$\cdots$};
}\]
Where $h_1 := \#\{k\in \jj | k<b'\}$ and $h_2:= b'-1-h_1$. Similarly, if instead $\underline{\jj'} \rightharpoonup \underline{\jj}$ such that $|\jj_L| = |\jj_L|-1$ and $|\jj_R'| = |\jj_R|$, then we define the map, 
$$q^{\#_{\underline{\jj'}}}P_{\underline{\jj}'} \overset{\tau^L_{\underline{\jj'}, \underline{\jj}}}{\longrightarrow}q^{\#_{\underline{\jj}}} P_{\underline{\jj}}$$
via the diagram,

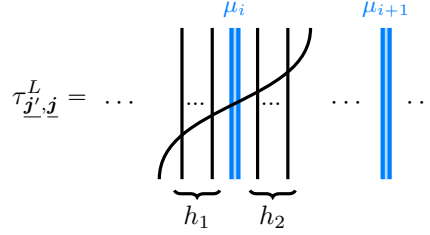
\[\tau^L_{\underline{\jj'}, \underline{\jj}} = \tikzdiagh[scale=2,xscale=-1]{0}{
\draw[vstdhl] (0,0) --(0,1)node[above]{\small $\mu_{i+1}$};
\draw[vstdhl] (1,0) --(1,1)node[above]{\small $\mu_{i}$};
\draw (1.5,0)..controls (1.5,0.5) and (0.5,0.5)..(0.5,1);
\draw (0.65,0)--(0.65,1);
\draw (0.85,0)--(0.85,1);
\node at(0.76,.5){\small $...$};
\draw (1.15,0)--(1.15,1);
\draw (1.35,0)--(1.35,1);
\node at(1.26,.5){\small $...$};
\tikzbraceop{.7}{0.8}{-0.6}{$h_2$};
\tikzbraceop{1.2}{1.3}{-0.6}{$h_1$};
\node at(0.25,.5){$\cdots$};
\node at(-0.25,.5){$\cdots$};
\node at(1.75,.5){$\cdots$};
}\]
Where now the numbers $h_1,h_2$ are determined through $h_1 := \#\{k\in \jj | k>b' \}$ and $h_2:= b'-1-h_1$. In general, we can set $\Tilde{\tau}_{\underline{\jj'}, \underline{\jj}}$ to be the map $\tau^L_{\underline{\jj'}, \underline{\jj}}$ or $\tau^R_{\underline{\jj'}, \underline{\jj}}$ depending on which direction the strand $b'$ comes from. Now, for $\underline{\jj} \rightharpoonup \underline{\jj}'$ let $b_{\underline{\jj},\underline{\jj'}} := \# \{k\in \jj |k>\jj-\jj'\}$. Note that $\jj-\jj'$ is always a one-set element in this case. We can finally write down the differentials as,
$$p_n = \underset{|\jj|=n}{\sum_{\underline{\jj}}}d_{\underline{\jj}}$$
$$d_{\underline{\jj}} = \underset{\underline{\jj} \rightharpoonup \underline{\jj}'}{\sum_{\underline{\jj}'}} (-1)^{b_{\underline{\jj},\underline{\jj'}}} \Tilde{\tau}_{\underline{\jj}, \underline{\jj}'}$$
\begin{lem}
    $p_np_{n-1} =0 $ for all $n$.
\end{lem}
\begin{proof}
    It is sufficient to show that starting at any projective labeled by $\underline{\jj}$ and applying the differentials gives 0. Therefore, 
    $$d_{\underline{\jj}}p_{n-1} = \sum_{\underline{\jj^\circ}, \underline{\jj'},\underline{\jj'^\circ}} (-1)^{b_{\underline{\jj},\underline{\jj^\circ}} +b_{\underline{\jj'},\underline{\jj'^\circ}}}\Tilde{\tau}_{\underline{\jj},\underline{\jj^\circ}} \Tilde{\tau}_{\underline{\jj'},\underline{\jj'^\circ}}$$
    $$= \sum_{ \underline{\jj'},\underline{\jj'^\circ}} (-1)^{b_{\underline{\jj},\underline{\jj'}} +b_{\underline{\jj'},\underline{\jj'^\circ}}} \Tilde{\tau}_{\underline{\jj},\underline{\jj'}} \Tilde{\tau}_{\underline{\jj'},\underline{\jj'^\circ}}$$
    We note that if $\underline{\jj} \rightharpoonup \underline{\jj'} \rightharpoonup \underline{\jj'^\circ } $ and $\underline{\jj} \rightharpoonup \underline{\jj''} \rightharpoonup \underline{\jj'^\circ } $, then by very similar arguments to the proof of Lemma 7.1 of \cite{Na}, we can show $\Tilde{\tau}_{\underline{\jj},\underline{\jj'}} \Tilde{\tau}_{\underline{\jj'},\underline{\jj'^\circ}} = \Tilde{\tau}_{\underline{\jj},\underline{\jj''}} \Tilde{\tau}_{\underline{\jj''},\underline{\jj'^\circ}}$. Therefore, we can rewrite the above expression,
    $$d_{\underline{\jj}}p_{n-1} = \sum_{\underline{\jj'^\circ}} \bigg(\sum_{ \underline{\jj'}} (-1)^{b_{\underline{\jj},\underline{\jj'}} +b_{\underline{\jj'},\underline{\jj'^\circ}}} \bigg)\Tilde{\tau}_{\underline{\jj},\underline{\jj''}} \Tilde{\tau}_{\underline{\jj''},\underline{\jj'^\circ}}$$
    for some fixed $\underline{\jj''}$. For a fixed $\underline{\jj'^\circ}$, we must have $\underline{\jj} = \underline{\jj'^\circ}\cup\{b_1,b_2\}$. Therefore, there are only two possible $\underline{\jj'}$'s that contribute to the sum. Assuming $b_2>b_1$ or $b_2<b_1$, it is easy to check that $b_{\underline{\jj},\underline{\jj'}} +b_{\underline{\jj'},\underline{\jj'^\circ}}$ evaluated on both of these choices differ by 1. This implies the result. 
\end{proof}
And so, we finally arrive at the result, 
\begin{prop}\label{prop:Rresolution}
    Let $\boldsymbol{p}\Rb^i_\rho$ denote the chain complex defined by, 
    $$\bigoplus_{\underline{\jj} =(\jj_L,\jj,\jj_R)} q^{\#_{\underline{\jj}}} P_{b_{i-1}+|\jj_L|, \jj , b_{i+1}+|\jj_R|} [|\jj|]$$
    with differentials, 
    $$p_n = \underset{|\jj|=n}{\sum_{\underline{\jj}}}d_{\underline{\jj}}$$
    $$d_{\underline{\jj}} = \underset{\underline{\jj} \rightharpoonup \underline{\jj}'}{\sum_{\underline{\jj}'}} (-1)^{b_{\underline{\jj},\underline{\jj'}}} \Tilde{\tau}_{\underline{\jj}, \underline{\jj}'}$$
    Then $\boldsymbol{p}\Rb^i_\rho$ is a projective resolution for $\Bb_i(P^{\underline{\mu}}_\rho)$.
\end{prop}
Below, we list the first few non-trivial examples, 
\begin{ex}{$b_i=1$}
\[
\begin{tikzcd}[row sep=small]
& x \hspace{2mm} \tikzdiagh{0}{
    \draw[vstdhl](0,0)--(0,0.5);
    \draw(1.25,0)--(1.25,0.5);
    \draw[vstdhl] (1,0)--(1,0.5);
    \filldraw [fill=white, draw=black] (-0.5,.5) rectangle (1.5,1) node[midway] { $\Tb$};
    } \ar{rd}
\ar[dash]{rd}{
\tikzdiagh[scale=0.75]{0}{
\draw[vstdhl](0,0)..controls (0,.5) and (1,.5) .. (1,1);
\draw[vstdhl](1,0)..controls (1,.5) and (0,.5) .. (0,1);
\draw (.5,0)..controls (.5, 0.5) and (1.33,0.5) ..(1.33,1);
}
}& \\
    x^2 \hspace{2mm} \tikzdiagh{0}{
    \draw[vstdhl](0,0)--(0,0.5);
    \draw(.5,0)--(.5,0.5);
    \draw[vstdhl] (1,0)--(1,0.5);
    \filldraw [fill=white, draw=black] (-0.5,.5) rectangle (1.5,1) node[midway] { $\Tb$};
    }
    \ar{rd}
    \ar[dash,swap]{rd}{-
\tikzdiag[scale=0.75,yscale=-1,xscale=-1]{
\draw[vstdhl](0,0)--(0,0.5);
\draw[vstdhl](1,0)--(1,0.5);
\draw (.5,0)..controls (.5, 0.25) and (1.33,0.25) ..(1.33,0.5);
} }
\ar{ru}
    \ar[dash]{ru}{
\tikzdiag[scale=0.75,yscale=-1]{
\draw[vstdhl](0,0)--(0,0.5);
\draw[vstdhl](1,0)--(1,0.5);
\draw (.5,0)..controls (.5, 0.25) and (1.33,0.25) ..(1.33,0.5);
} }
    &
    \bigoplus
    &
    \tikzdiagh{0}{
    \draw[vstdhl](0,0) node[below]{\small $\mu_i$}--(0,0.5) node[color=black, pos=.5,left=2mm]{$\cdots$};
    \draw(.5,0)--(.5,0.5);
    \draw[vstdhl] (1,0)--(1,0.5) node[color=black, pos=.5,right=2mm]{$\cdots$};
    \filldraw [fill=white, draw=black] (-0.5,.5) rectangle (1.5,1) node[midway] { $\Rb_i$};
    } \\
    & x \hspace{2mm} \tikzdiagh{0}{
    \draw[vstdhl](0,0)--(0,0.5);
    \draw(-.25,0)--(-.25,0.5);
    \draw[vstdhl] (1,0)--(1,0.5);
    \filldraw [fill=white, draw=black] (-0.5,.5) rectangle (1.5,1) node[midway] { $\Tb$};
    } \ar{ru}
    \ar[dash,swap]{ru}{
\tikzdiagh[scale=0.75]{0}{
\draw[vstdhl](0,0)..controls (0,.5) and (1,.5) .. (1,1);
\draw[vstdhl](1,0)..controls (1,.5) and (0,.5) .. (0,1);
\draw (.5,0)..controls (.5, 0.5) and (-.33,0.5) ..(-.33,1);
}
}
\end{tikzcd}
\]
\end{ex}

\begin{ex}{$b_i=2$}
\[
\begin{tikzcd}[row sep=tiny,column sep=large]
    &
    x^3\hspace{2mm}\tikzdiag[xscale=1]{
    \draw(0,0)--(0,0.5);
    \draw[vstdhl] (.33,0)--(.33,0.5);
    \draw(.66,0)--(.66,0.5);
    \draw[vstdhl] (1,0)--(1,0.5);
    \filldraw [fill=white, draw=black] (-.25,.5) rectangle (1.25,1) node[midway] { $\Tb$};
    } 
    \ar[yshift=-3pt]{r}
    \ar[dash, yshift=-3pt]{r}{-
\tikzdiag[xscale=0.5,yscale=0.5]{
\draw[vstdhl](0,0)--(0,1);
\draw[vstdhl](1,0)--(1,1);
\draw (-.5,0)..controls (-.5, 0.5) and (.5,0.5) ..(.5,1);
\draw (-1,0)--(-1,1);
}
}
\ar[xshift=3mm,yshift=-3mm]{rdd}
    \ar[pos=0.7,dash, xshift=3mm,yshift=-3mm]{rdd}{
\tikzdiag[xscale=0.5,yscale=0.5]{
\draw[vstdhl](0,0)--(0,1);
\draw[vstdhl](1,0)--(1,1);
\draw (1.5,0)..controls (1.5, 0.5) and (.5,0.5) ..(.5,1);
\draw (-.5,0)--(-.5,1);
}
}
    & 
    x^2\hspace{2mm}\tikzdiag[xscale=1]{
    \draw(0,0)--(0,0.5);
    \draw(.33,0)--(.33,0.5);
    \draw[vstdhl](.66,0)--(.66,0.5);
    \draw[vstdhl] (1,0)--(1,0.5);
    \filldraw [fill=white, draw=black] (-.25,.5) rectangle (1.25,1) node[midway] { $\Tb$};
    } 
    \ar{rddd}
    \ar[dash]{rddd}{
\tikzdiag[xscale=0.75,yscale=0.75]{
\draw[vstdhl](0,0)..controls (0,.75) and (1,.75) ..(1,1);
\draw[vstdhl](1,0) ..controls (1,.75) and (0,.75) .. (0,1);
\draw (.66,0)..controls (.66, 0.5) and (-.33,0.5) ..(-.33,1);
\draw (.33,0)..controls (.33, 0.5) and (-.66,0.5) ..(-.66,1);
}
}\\
    &\bigoplus & \bigoplus \\
    &
    x^3 q^{-2}\hspace{2mm}\tikzdiag[xscale=1]{
    \draw(0,0)--(0,0.5);
    \draw[vstdhl] (.33,0)--(.33,0.5);
    \draw(.66,0)--(.66,0.5);
    \draw[vstdhl] (1,0)--(1,0.5);
    \filldraw [fill=white, draw=black] (-.25,.5) rectangle (1.25,1) node[midway] { $\Tb$};
    }
    \ar[yshift=-3pt]{ruu}
    \ar[dash, yshift=-3pt]{ruu}{-
\tikzdiag[xscale=0.5,yscale=0.5]{
\draw[vstdhl](0,0)--(0,1);
\draw[vstdhl](1,0)--(1,1);
\draw (-1,0)..controls (-1, 0.5) and (.5,0.5) ..(.5,1);
\draw (-.5,0)--(-.5,1);
}
}
\ar[yshift=-3pt]{rdd}
    \ar[dash, yshift=-3pt,pos=.25]{rdd}{
\tikzdiag[xscale=0.5,yscale=0.5]{
\draw[vstdhl](0,0)--(0,1);
\draw[vstdhl](1,0)--(1,1);
\draw (1.5,0)..controls (1.5, 0.5) and (.5,0.5) ..(.5,1);
\draw (-.5,0)--(-.5,1);
}
}
    & 
    x^2\hspace{2mm}\tikzdiag[xscale=1]{
    \draw(0,0)--(0,0.5);
    \draw[vstdhl] (.33,0)--(.33,0.5);
    \draw[vstdhl](.66,0)--(.66,0.5);
    \draw (1,0)--(1,0.5);
    \filldraw [fill=white, draw=black] (-.25,.5) rectangle (1.25,1) node[midway] { $\Tb$};
    }
    \ar{rd}
    \ar[dash,pos=0.25,swap]{rd}{
\tikzdiag[xscale=0.75,yscale=0.75]{
\draw[vstdhl](0,0)..controls (0,.75) and (1,.75) ..(1,1);
\draw[vstdhl](1,0) ..controls (1,.75) and (0,.75) .. (0,1);
\draw (.66,0)..controls (.66, 0.5) and (1.33,0.5) ..(1.33,1);
\draw (.33,0)..controls (.33, 0.5) and (-.33,0.5) ..(-.33,1);
}
}
    \\
    x^4 q^{-2}\hspace{2mm}\tikzdiag[xscale=1]{
    \draw[vstdhl](0,0)--(0,0.5);
    \draw (.33,0)--(.33,0.5);
    \draw(.66,0)--(.66,0.5);
    \draw[vstdhl] (1,0)--(1,0.5);
    \filldraw [fill=white, draw=black] (-.25,.5) rectangle (1.25,1) node[midway] { $\Tb$};
    } 
    \ar{ruuu}
    \ar[dash]{ruuu}{-
\tikzdiag[xscale=0.5,yscale=0.5]{
\draw[vstdhl](0,0)--(0,1);
\draw[vstdhl](1,0)--(1,1);
\draw (-.5,0)..controls (-.5, 0.5) and (.66,0.5) ..(.66,1);
\draw (.33,0)--(.33,1);
}
}
\ar{ru}
    \ar[dash,swap]{ru}{
\tikzdiag[xscale=0.5,yscale=0.5]{
\draw[vstdhl](0,0)--(0,1);
\draw[vstdhl](1,0)--(1,1);
\draw (-.5,0)..controls (-.5, 0.5) and (.33,0.5) ..(.33,1);
\draw (.66,0)--(.66,1);
}
}
\ar{rd}
    \ar[dash]{rd}{-
\tikzdiag[xscale=-0.5,yscale=0.5]{
\draw[vstdhl](0,0)--(0,1);
\draw[vstdhl](1,0)--(1,1);
\draw (-.5,0)..controls (-.5, 0.5) and (.66,0.5) ..(.66,1);
\draw (.33,0)--(.33,1);
}
}
\ar{rddd}
    \ar[dash,swap]{rddd}{
\tikzdiag[xscale=-0.5,yscale=0.5]{
\draw[vstdhl](0,0)--(0,1);
\draw[vstdhl](1,0)--(1,1);
\draw (-.5,0)..controls (-.5, 0.5) and (.33,0.5) ..(.33,1);
\draw (.66,0)--(.66,1);
}
}
    &\bigoplus & \bigoplus &
    \tikzdiag[xscale=1]{
    \draw[vstdhl](0,0) node[below]{\small $\mu_i$}--(0,0.5);
    \draw (.33,0)--(.33,0.5);
    \draw(.66,0)--(.66,0.5);
    \draw[vstdhl] (1,0)--(1,0.5);
    \filldraw [fill=white, draw=black] (-.25,.5) rectangle (1.25,1) node[midway] { $\Rb_i$};
    } \\
    &
    x^3\hspace{2mm}\tikzdiag[xscale=1]{
    \draw[vstdhl] (0,0)--(0,0.5);
    \draw (.33,0)--(.33,0.5);
    \draw[vstdhl](.66,0)--(.66,0.5);
    \draw (1,0)--(1,0.5);
    \filldraw [fill=white, draw=black] (-.25,.5) rectangle (1.25,1) node[midway] { $\Tb$};
    } 
    \ar[yshift=-3pt]{ruu}
    \ar[dash, yshift=-3pt,pos=0.25]{ruu}{-
\tikzdiag[xscale=0.5,yscale=0.5]{
\draw[vstdhl](0,0)--(0,1);
\draw[vstdhl](1,0)--(1,1);
\draw (-.5,0)..controls (-.5, 0.5) and (.5,0.5) ..(.5,1);
\draw (1.5,0)--(1.5,1);
}
}
\ar[xshift=1mm,yshift=2mm]{rdd}
    \ar[pos=0.1,dash, xshift=  1mm,yshift=2mm]{rdd}{
\tikzdiag[xscale=0.5,yscale=0.5]{
\draw[vstdhl](0,0)--(0,1);
\draw[vstdhl](1,0)--(1,1);
\draw (1.5,0)..controls (1.5, 0.5) and (.5,0.5) ..(.5,1);
\draw (2,0)--(2,1);
}
}
    & 
    x^2 q^{-2}\hspace{2mm}\tikzdiag[xscale=1]{
    \draw(0,0)--(0,0.5);
    \draw[vstdhl](.33,0)--(.33,0.5);
    \draw[vstdhl](.66,0)--(.66,0.5);
    \draw (1,0)--(1,0.5);
    \filldraw [fill=white, draw=black] (-.25,.5) rectangle (1.25,1) node[midway] { $\Tb$};
    } 
    \ar[yshift=-3mm,xshift=3mm]{ru}
    \ar[dash,pos=0.25,yshift=-3mm,xshift=3mm]{ru}{
\tikzdiag[xscale=0.75,yscale=0.75]{
\draw[vstdhl](0,0)..controls (0,.75) and (1,.75) ..(1,1);
\draw[vstdhl](1,0) ..controls (1,.75) and (0,.75) .. (0,1);
\draw (.66,0)..controls (.66, 0.5) and (-.33,0.5) ..(-.33,1);
\draw (.33,0)..controls (.33, 0.5) and (1.33,0.5) ..(1.33,1);
}
}
    \\
    &\bigoplus & \bigoplus \\
    &
    x^3 q^{-2}\hspace{2mm}\tikzdiag[xscale=1]{
    \draw[vstdhl](0,0)--(0,0.5);
    \draw (.33,0)--(.33,0.5);
    \draw[vstdhl](.66,0)--(.66,0.5);
    \draw(1,0)--(1,0.5);
    \filldraw [fill=white, draw=black] (-.25,.5) rectangle (1.25,1) node[midway] { $\Tb$};
    }
    \ar[xshift=2mm, yshift=-3pt]{ruu}
    \ar[dash, xshift=2mm, yshift=-3pt,pos=.1]{ruu}{-
\tikzdiag[xscale=0.5,yscale=0.5]{
\draw[vstdhl](0,0)--(0,1);
\draw[vstdhl](1,0)--(1,1);
\draw (-.5,0)..controls (-.5, 0.5) and (.5,0.5) ..(.5,1);
\draw (1.5,0)--(1.5,1);
}
}
\ar[yshift=-3pt]{r}
    \ar[dash, yshift=-3pt,swap]{r}{
\tikzdiag[xscale=0.5,yscale=0.5]{
\draw[vstdhl](0,0)--(0,1);
\draw[vstdhl](1,0)--(1,1);
\draw (2,0)..controls (2, 0.5) and (.5,0.5) ..(.5,1);
\draw (1.5,0)--(1.5,1);
}
}
    & 
    x^2\hspace{2mm}\tikzdiag[xscale=1]{
    \draw[vstdhl](0,0)--(0,0.5);
    \draw[vstdhl] (.33,0)--(.33,0.5);
    \draw(.66,0)--(.66,0.5);
    \draw (1,0)--(1,0.5);
    \filldraw [fill=white, draw=black] (-.25,.5) rectangle (1.25,1) node[midway] { $\Tb$};
    } 
    \ar{ruuu}
    \ar[dash,swap]{ruuu}{
\tikzdiag[xscale=-0.75,yscale=0.75]{
\draw[vstdhl](0,0)..controls (0,.75) and (1,.75) ..(1,1);
\draw[vstdhl](1,0) ..controls (1,.75) and (0,.75) .. (0,1);
\draw (.66,0)..controls (.66, 0.5) and (-.33,0.5) ..(-.33,1);
\draw (.33,0)..controls (.33, 0.5) and (-.66,0.5) ..(-.66,1);
}
}
\end{tikzcd}
\]
\end{ex}
\begin{rem}
These two examples first appeared in an unpublished draft of Dupont and Naisse, and it is worth remarking that we find the same projective resolutions they do in these cases. Interestingly, the draft aimed to categorify the Lawrence-Krammer-Bigelow representation of the braid group using these resolutions of the first few weight spaces. It would be interesting to see if this can be done using the results in this paper. 
\end{rem}

For the purpose of proving that the braiding functors fulfill the braid group relations, it will be useful to study iterated extensions by standards that the braiding bimodules admit. This is largely inspired by the proof of Lemma 6.12 of \cite{Web}. As in Section \ref{sec:Standards}, we introduce the set $\Phi^w_\rho$, which is acquired from $\Phi_\rho$ by permuting the colored strands according to $w\in S_r$ whilst keeping the blocks of black strands to the right of each colored strand together. We take the string diagram representatives $y_\phi$ of the permutations $\phi \in \Phi^w_\rho$ to be diagrams realizing $\phi$ with the minimal number of crossings. 
\begin{ex}
    Consider $\underline{\mu} = (\lambda,\lambda,\lambda,\lambda )$ and $\rho = (1, 1, 1, 0)$. In this case, the set $\Phi_\rho$ has 6 elements, 
    
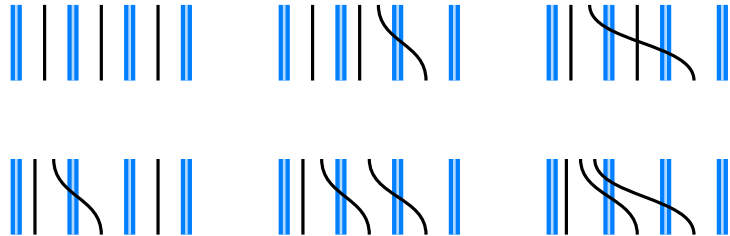
\[
    \begin{tikzcd}
        \tikzdiagh[yscale=2, xscale=0.75]{0}{
\draw[vstdhl](0,0)--(0,0.5);
\draw (0.5,0)--(0.5, 0.5);
\draw[vstdhl](1,0)--(1,0.5);
\draw (1.5,0)--(1.5, 0.5);
\draw[vstdhl](2,0)--(2,0.5);
\draw (2.5,0)--(2.5, 0.5);
\draw[vstdhl](3,0)--(3,0.5);
}
&
\tikzdiagh[yscale=2, xscale=0.75]{0}{
\draw[vstdhl](0,0)--(0,0.5);
\draw (0.5,0)--(0.5, 0.5);
\draw[vstdhl](1,0)--(1,0.5);
\draw (1.33,0)--(1.33, 0.5);
\draw[vstdhl](2,0)--(2,0.5);
\draw (2.5,0)..controls (2.5,0.25) and (1.66, 0.25)..(1.66, 0.5);
\draw[vstdhl](3,0)--(3,0.5);
}
&
\tikzdiagh[yscale=2, xscale=0.75]{0}{
\draw[vstdhl](0,0)--(0,0.5);
\draw (0.33,0)--(0.33, 0.5);
\draw[vstdhl](1,0)--(1,0.5);
\draw (1.5,0)--(1.5, 0.5);
\draw[vstdhl](2,0)--(2,0.5);
\draw (2.5,0)..controls (2.5,0.25) and (.66, 0.25)..(.66, 0.5);
\draw[vstdhl](3,0)--(3,0.5);
}\\
\tikzdiagh[yscale=2, xscale=0.75]{0}{
\draw[vstdhl](0,0)--(0,0.5);
\draw[vstdhl](1,0)--(1,0.5);
\draw[vstdhl](2,0)--(2,0.5);
\draw[vstdhl](3,0)--(3,0.5);
\draw (0.33,0)--(0.33, 0.5);
\draw (2.5,0)--(2.5, 0.5);
\draw (1.5,0)..controls (1.5,0.25) and (.66, 0.25)..(.66, 0.5);
}
&
\tikzdiagh[yscale=2, xscale=0.75]{0}{
\draw[vstdhl](0,0)--(0,0.5);
\draw[vstdhl](1,0)--(1,0.5);
\draw[vstdhl](2,0)--(2,0.5);
\draw[vstdhl](3,0)--(3,0.5);
\draw (0.33,0)--(0.33, 0.5);
\draw (2.5,0)..controls (2.5,0.25) and (1.5, 0.25)..(1.5, 0.5);
\draw (1.5,0)..controls (1.5,0.25) and (.66, 0.25)..(.66, 0.5);
}
&
\tikzdiagh[yscale=2, xscale=0.75]{0}{
\draw[vstdhl](0,0)--(0,0.5);
\draw[vstdhl](1,0)--(1,0.5);
\draw[vstdhl](2,0)--(2,0.5);
\draw[vstdhl](3,0)--(3,0.5);
\draw (0.25,0)--(0.25, 0.5);
\draw (2.5,0)..controls (2.5,0.25) and (.75, 0.25)..(.75, 0.5);
\draw (1.5,0)..controls (1.5,0.25) and (.5, 0.25)..(.5, 0.5);
}
\end{tikzcd}
\]
Consider $w = s_2 s_3$. Then $\Phi^w_\rho$ has the representative elements, 
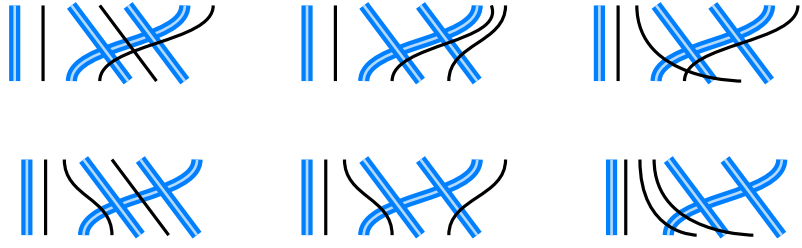
\[
\begin{tikzcd}
    \tikzdiagh[yscale=1, xscale=0.75]{0}{
    \draw[vstdhl] (0,0)--(0,1);
    \draw[vstdhl] (1,0)..controls (1,0.5) and (3,0.5) ..(3,1);
    \draw[vstdhl] (2,0)--(1,1);
    \draw[vstdhl] (3,0)--(2,1);
    \draw (0.5,0)--(0.5, 1);
    \draw(1.5,0)..controls (1.5,0.5) and (3.5,0.5) ..(3.5,1);
    \draw (2.5,0) --(1.5,1);
    }
    &
    \tikzdiagh[yscale=1, xscale=0.75]{0}{
    \draw[vstdhl] (0,0)--(0,1);
    \draw[vstdhl] (1,0)..controls (1,0.5) and (3,0.5) ..(3,1);
    \draw[vstdhl] (2,0)--(1,1);
    \draw[vstdhl] (3,0)--(2,1);
    \draw (0.5,0)--(0.5, 1);
    \draw(1.5,0)..controls (1.5,0.5) and (3.5,0.5) ..(3.25,1);
    \draw (2.5,0)..controls (2.5,0.5) and (3.5,0.5) ..(3.5,1);
    }
    &
    \tikzdiagh[yscale=1, xscale=0.75]{0}{
    \draw[vstdhl] (0,0)--(0,1);
    \draw[vstdhl] (1,0)..controls (1,0.5) and (3,0.5) ..(3,1);
    \draw[vstdhl] (2,0)--(1,1);
    \draw[vstdhl] (3,0)--(2,1);
    \draw (0.33,0)--(0.33, 1);
    \draw(1.5,0)..controls (1.5,0.5) and (3.5,0.5) ..(3.5,1);
    \draw (2.5,0)..controls (2.5,0) and (0.66,0) ..(0.66,1);
    }\\
    \tikzdiagh[yscale=1, xscale=0.75]{0}{
    \draw[vstdhl] (0,0)--(0,1);
    \draw[vstdhl] (1,0)..controls (1,0.5) and (3,0.5) ..(3,1);
    \draw[vstdhl] (2,0)--(1,1);
    \draw[vstdhl] (3,0)--(2,1);
    \draw (0.33,0)--(0.33, 1);
    \draw(1.5,0)..controls (1.5,0.5) and (0.66,0.5) ..(0.66,1);
    \draw (2.5,0) --(1.5,1);
    }
    &
    \tikzdiagh[yscale=1, xscale=0.75]{0}{
    \draw[vstdhl] (0,0)--(0,1);
    \draw[vstdhl] (1,0)..controls (1,0.5) and (3,0.5) ..(3,1);
    \draw[vstdhl] (2,0)--(1,1);
    \draw[vstdhl] (3,0)--(2,1);
    \draw (0.33,0)--(0.33, 1);
    \draw(1.5,0)..controls (1.5,0.5) and (.66,0.5) ..(.66,1);
    \draw (2.5,0)..controls (2.5,0.5) and (3.5,0.5) ..(3.5,1);
    }
    &
    \tikzdiagh[yscale=1, xscale=0.75]{0}{
    \draw[vstdhl] (0,0)--(0,1);
    \draw[vstdhl] (1,0)..controls (1,0.5) and (3,0.5) ..(3,1);
    \draw[vstdhl] (2,0)--(1,1);
    \draw[vstdhl] (3,0)--(2,1);
    \draw (0.25,0)--(0.25, 1);
    \draw(1.5,0)..controls (1.5,0) and (.5,0) ..(.5,1);
    \draw (2.5,0)..controls (2.5,0) and (0.75,0) ..(0.75,1);
    }
\end{tikzcd}
\]
\end{ex}
Now, we can more clearly describe the desired iterated extension. 
\begin{cor}\label{cor:R.it.ext.}
$\Bb_i(P^{\underline{\mu}}_\rho)$ has a length $|\Phi^w_\rho|$ iterated extension, where each $\phi\in \Phi^w_\rho$ appears once in the sense that consecutive cones are quasi-isomorphic to $q^{deg \hspace{1mm}y_\phi}\cdot \Sb^{\underline{\mu}}_{\rho_\phi}$ for each $\phi$. 
\end{cor}
\begin{proof}
    It is sufficient to prove this for the case $length(\underline{\mu})=3$, $b_3 =0$, and $i=2$, since we can use Proposition \ref{prop:it.ext.proj} to extend to the general case. In this case, it is immediate from Proposition \ref{prop:Rresolution}, where the chain complexes for $q^{deg \hspace{1mm}y_\phi}\cdot \Sb^{\underline{\mu}}_{\rho_\phi}$ are, by observation, honest sub-complexes of $\boldsymbol{p}\Rb^i_\rho$, and we have an honest filtration compatible with the preorder. 
\end{proof}

Moreover, this set of elements is a very useful generating set for the braiding bimodules. 
\begin{lem}\label{lem:gen.set}
$\Phi^w_\rho$ is a generating set for $\Rb_{w}1_{\rho}$ as a right $\Tb^{w(\underline{\mu})}$-module. 
\end{lem}
\begin{proof}
    Again, as in Proposition \ref{prop:Rbasis}, we induct on the number of black/black and black/colored crossings. The base case is similar, so we proceed to the induction step. Let $d \in \Rb_{w}1_{\rho}$ be any diagram. Let us imagine $y=\frac{1}{2}$ denotes the vertical halfway point of the diagram. It is enough to show that $d$ is equal to a sum of diagrams with $y_\phi$'s on their bottom half modulo diagrams with fewer crossings. First, isotope all colored crossings to be below the $y=\frac{1}{2}$ line. As in Prop. \ref{prop:Rbasis}, we can move all dots and nails above $y=\frac{1}{2}$ at the cost of introducing diagrams with fewer crossings. The remainder of the proof follows just as in \cite{Web}, Lemma 6.12.
\end{proof}

Now, we prove a technical lemma that will be useful in the proof of R3 invariance. 
\begin{lem}\label{lem:it.exts.tensor}
    Let $M,N$ be $\Tb$ bimodules with iterated extensions by standards as right and left modules respectively. If the canonical map $M\otimes_\Tb^LN\rightarrow M\otimes_\Tb N$ is surjective, then it is also an isomorphism. 
\end{lem}
\begin{proof}
Let $\phi: M\otimes_\Tb^LN\rightarrow M\otimes_\Tb N$ be the canonical morphism in question. We have a short exact sequence, 
$$0\rightarrow ker\hspace{1mm}\phi \rightarrow M\otimes_\Tb^LN\rightarrow M\otimes_\Tb N\rightarrow 0$$
And we hope to show $ker\hspace{1mm} \phi =0$. Our strategy will be very similar to the one in \cite{Web}, with one extra step. Loosely speaking, to successfully apply the methods of \cite{Web}, we will show that the claims follows in the KLRW algebra in the absence of white dots and then argue that their structure allows one to conclude the stronger claim.\\
Let $I$ be the ideal of $\Tb$ generated by all diagrams with white dots. Then $\Tb/I = \Tilde{\Tb}$ is the ring acquired from $\Tb$ by introducing the local relation, 
\begin{align}
	\tikzdiag[yscale=1]{
		\draw (0,0.5)..controls (1,0.5) and (1,0.75) .. (1,1);
\draw (0,0.5)..controls (1,0.5) and (1,0.25) .. (1,0);
\draw[vstdhl] (0,0)--(0,1) node[midway,nail]{};
	} 
	\ &= 0
\end{align}
We will use tildes to denote the respective modules over $\Tilde{\Tb}$. The iterated extensions pass to $\Tilde{M}$ and $\Tilde{N}$ and the consecutive cones now consist of $\{\Tilde{\Sb}^{\underline{\mu}}_\rho\}_{\rho \in \mathcal{M}}$ and $\{\dot{\Tilde{\Sb}}^{\underline{\mu}}_\kappa\}_{\kappa \in \mathcal{N}}$ respectively for indexing sets $\mathcal{M},\mathcal{N}$. Note that for $\tilde{ker \hspace{1mm} \phi} =0$, it is sufficient that $H_k(\Tilde{M}\otimes^L \Tilde{N}) = 0 $ for $k>0$. First, we show that $H_k(\Tilde{\Sb}^{\underline{\mu}}_\rho\otimes^L \Tilde{N}) =0$ for $k>0$. The iterated extension for $\tilde{N}$, 
$$0\rightarrow F^1_N\rightarrow F^2_N\rightarrow \cdots \rightarrow F^r = \tilde{N}$$
and the associated exact triangles, 
$$0\rightarrow F^i_N \rightarrow \dot{\Tilde{\Sb}}^{\underline{\mu}}_\kappa \rightarrow F^{i+1}_N[1]\rightarrow 0$$
induce a long exact sequence on the homology groups for all $i$, 
$$\cdots \rightarrow H_{k}(\Tilde{\Sb}^{\underline{\mu}}_\rho\otimes^L F^{i+1}_N)\rightarrow H_k(\Tilde{\Sb}^{\underline{\mu}}_\rho\otimes^L F^i_N )\rightarrow H_k(\Tilde{\Sb}^{\underline{\mu}}_\rho\otimes^L \dot{\Tilde{\Sb}}^{\underline{\mu}}_\kappa )\rightarrow H_{k-1}(\Tilde{\Sb}^{\underline{\mu}}_\rho\otimes^L F^{i+1}_N)\rightarrow\cdots $$
Note that the results of Section 7 of \cite{Na} imply that $H_k(\Tilde{\Sb}^{\underline{\mu}}_\rho\otimes^L \dot{\Tilde{\Sb}}^{\underline{\mu}}_\kappa ) \cong 0$ for all $k>0$ and $\kappa, \rho$. Therefore, for $k>0$, we get $ H_{k}(\Tilde{\Sb}^{\underline{\mu}}_\rho\otimes^L F^{i+1}_N)\cong H_k(\Tilde{\Sb}^{\underline{\mu}}_\rho\otimes^L F^i_N )$, so that after iterating finitely many times, we find $ H_{k}(\Tilde{\Sb}^{\underline{\mu}}_\rho\otimes^L \tilde{N})\cong 0$ for $k>0$. Now, we repeat the mirror argument for the iterated extension of $\tilde{M}$ to find, $H_k(\Tilde{M}\otimes^L \Tilde{N}) = 0 $ for $k>0$ and therefore $\Tilde{M}\otimes^L \Tilde{N} \cong \Tilde{M}\otimes \Tilde{N}$. This implies $ker\hspace{1mm} \phi \otimes_\Tb  \tilde{\Tb} =0$. However, this is only so if $ker\hspace{1mm} \phi = ker\hspace{1mm} \phi \cdot I $. However, $ker\hspace{1mm} \phi \cdot I^n$ must have diagrams with at least $n$ white dots (if they are non-zero). Thus, by finiteness of string diagrams, we must have $ker\hspace{1mm} \phi =0$, implying the claim. 

\end{proof}

\begin{lem}\label{lem:RRit.ext.}
    For $i\neq j$, $\Bb_i \circ \Bb_j \cong \Bb_{s_i s_j}$. Furthermore, $\Bb_{s_i s_j}(\Pb^{\underline{\mu}}_\rho)$ has an iterated extension by standards determined by $\Phi^w_\rho$ in the sense above. 
\end{lem}

\begin{proof}
     By Corollary \ref{cor:R.it.ext.}, we may consider iterated extensions of $\Rb_i$ as a right module and $\Rb_j$ as a left module (the ladder follows by noticing, as in \cite{Web}, $ \Rb_{s_j^{-1}} \cong \dot{\Rb}_{s_j} $) where consecutive cones are quasi-isomorphic to standard modules. Now, we wish to examine $\Rb_i \otimes^L \Rb_j \otimes^L \Pb^{\underline{\mu}}_\rho \cong \Rb_i \otimes^L \Rb_j1_\rho$. As usual, denote $\rho = (b_1,b_2,...,b_r)$. Applying Proposition \ref{prop:Rresolution} and tensoring with $\Rb_i$, we get a natural map between the complex, 
    $$\bigoplus_{\underline{\jj} =(\jj_L,\jj,\jj_R)} q^{\#_{\underline{\jj}}} \Rb_i\otimes P_{b_{j-1}+|\jj_L|, \jj , b_{j+1}+|\jj_R|} [|\jj|]$$
    and $\Rb_{s_i s_j}$ via the maps, 
    $$\cdots \rightarrow \bigoplus_{I\in C(b_j)}q^{f(I)}\Rb_i 1_{\rho_I}\overset{\sum_I p^I_0}{\longrightarrow } \Rb_{s_1 s_j} 1_\rho $$
    by viewing the diagrams of $\Rb_i 1_{\rho_I}$ composed with $p^I_0$ as diagrams of $\Rb_{s_1 s_j}$. Note here that we used the shorthand $\rho_I = (b_1,...,b_{j-1}+|I|,0,b_{j+1}+b_j-|I|,b_{j+2},...,b_r)$. This map is easily seen to be surjective (for instance, by noting $\Rb_i \otimes-$ is a right exact functor). Applying Lemma \ref{lem:it.exts.tensor} to these facts, we find that $\Rb_i\otimes^L \Rb_j \cong \Rb_i\otimes \Rb_j$. \\
    To show that $\Rb_{s_is_j}$ has an iterated extension by standards as in the claim, it is again sufficient to consider the case $length(\underline{\mu})=4$,  $b_4 =0$. A projective resolution for $\Rb_{s_is_j}1_\rho$ may be easily acquired by taking $\Rb_i \otimes^L \boldsymbol{p}\Rb^\rho_j$ and then resolving all summands in the complex according to Proposition \ref{prop:Rresolution}. Call the resulting resolution $\boldsymbol{p}\Rb^\rho_{s_is_j}$. From the structure of the differentials and the summands, we can find an iterated extension along the arguments of Proposition \ref{prop:it.ext.proj}, 
    $$Q_0\overset{f_1}{\longrightarrow }Q_1\overset{f_2}{\longrightarrow }\cdots \overset{f_{|\Phi_\rho|}}{\longrightarrow } Q_{|\Phi_\rho|} = \boldsymbol{p}\Rb^\rho_{s_is_j} \cong \Rb_{s_is_j}1_\rho $$
    such that $Cone(f_k) \cong q^{deg \hspace{1mm} y_{\phi_k}}\Sb^{\underline{\mu}}_{\rho_{\phi_k}}$ for $\{\phi_k\}_k = \Phi^{s_is_j}_\rho$ some indexing of $\Phi^{s_is_j}_\rho$ compatible with the partial order from Section \ref{sec:Standards}. The non-trivial part of the claim then comes down to showing $Q_0 \cong 0$. We note that the iterated extension above shows that, $$gdim_{q,x} \hspace{1mm}\Bb(\Pb^{\underline{\mu}}_\rho)\geq \sum_{\phi\in \Phi^{s_is_j}_\rho} q^{deg \hspace{1mm} y_{\phi}}gdim_{q,x} \hspace{1mm}\Sb^{\underline{\mu}}_{\rho_{\phi}}$$
    Therefore, the final part of the claim will follow once we show the complimentary inequality also holds. 
    We focus on the exact sequence from before, 
    $$\bigoplus_{I\in C(b_j)}q^{f(I)}\Rb_i 1_{\rho_I}\overset{\sum_I p^I_0}{\longrightarrow } \Rb_{s_i s_j} 1_\rho \rightarrow 0$$
    Suppose for the moment that for all $I\in C(b_j)$, there exists $\hat{\phi}^I \in \Phi^{s_i}_{\rho_I}$ such that, 
    $$\hat{\Phi}^{s_i}_{\rho_I} \equiv \{\phi \in \Phi^{s_i}_{\rho_I}| \hat{\phi_I}\lesssim \phi\}$$
    $$p_0\bigg (\bigcup_I \hat{\Phi}^{s_i}_{\rho_I}\bigg)= \Phi^{s_is_j}_\rho$$
    $$p_0\bigg (\hat{\Phi}^{s_i}_{\rho_I}\bigg) \bigcap p_0\bigg (\hat{\Phi}^{s_i}_{\rho_J}\bigg) = \emptyset, \hspace{5mm} I\neq J$$
    Then we can form $\Tb$ modules $Q_{\hat{\phi}_I}$ such that the cones $\Rb_i 1_{\rho_I}/Q_{\hat{\phi}_I}$ have iterated extensions governed by $\hat{\Phi}^{s_i}_{\rho_I}$ (this can be done, for instance, by taking the iterated extension of $\Rb_i 1_{\rho_I}$ described by Corollary \ref{cor:R.it.ext.}, permuting the terms such that the first $|\hat{\Phi}^{s_i}_{\rho_I}|$ consecutive cones are the standards associated with $\hat{\Phi}^{s_i}_{\rho_I}$ and letting $Q_{\hat{\phi}_I}$ be the term in the iterated extension such that $Cone (Q_{\hat{\phi}_I} \rightarrow Q_{k_{\hat{\phi}_I}}) \cong \Sb^{\underline{\mu}}_{\rho_{\hat{\phi}_I}}$). Furthermore, by Lemma \ref{lem:gen.set}, the sequence, 
    $$\bigoplus_{I\in C(b_j)}q^{f(I)}\Rb_i 1_{\rho_I}/Q_{\hat{\phi}_I}\overset{\sum_I p^I_0}{\longrightarrow } \Rb_{s_i s_j} 1_\rho \rightarrow 0$$
    remains exact and we therefore conclude,
    $$gdim_{q,x} \hspace{1mm}\Bb(\Pb^{\underline{\mu}}_\rho)\leq \sum_{\phi\in \Phi^{s_is_j}_\rho} q^{deg \hspace{1mm} y_{\phi}}gdim_{q,x} \hspace{1mm}\Sb^{\underline{\mu}}_{\rho_{\phi}}$$
    Now, we show that such an element $\hat{\phi}^I$ always exists. An arbitrary element of $\phi \in \Phi^{s_i s_j}_\rho$ (for the relevant case at hand) may be identified with an element $\phi\in \Z^{b_2}_2 \Z^{b_3}_3$ by considering string diagrams of $\phi$. For instance, if $w = s_3s_2$, the string diagrams could be written as, 
    \[
    \tikzdiagh[yscale=2, xscale=2]{0}{
    \draw[bund] (0.25,0) node[below]{$b_1$}  --(0.25, 1);
    \draw[bund] (1.5,0) node[below]{$b_2$} --(1.5,0.15);
    \draw[bund](1.6,0.15)..controls (1.6,0.5) and (3.25,0.5) ..(3.25,1);
    \draw[bund](1.4,0.15)..controls (1.4,0.5) and (.5,0.5) ..(.5,1);
    \filldraw [fill=white, draw=black] (1.3,.10) rectangle (1.7,.25) node[midway] { $\omega_2$};
    \draw[bund] (2.5,0) node[below]{$b_3$} --(2.5,0.15);
    \draw[bund](2.4,0.15)..controls (2.4,0.5) and (.75,0.5) ..(.75,1);
    \draw[bund](2.55,0.15)..controls (2.55,0.5) and (1.5,0.5) ..(1.5,1);
    \draw[bund](2.6,0.15)..controls (2.6,0.5) and (3.5,0.5) ..(3.5,1);
    \filldraw [fill=white, draw=black] (2.3,.10) rectangle (2.7,.25) node[midway] { $\omega_3$};
    \draw[vstdhl] (0,0)--(0,1);
    \draw[vstdhl] (1,0)..controls (1,0.5) and (3,0.5) ..(3,1);
    \draw[vstdhl] (2,0) ..controls (2,0.5) and (1,0.5) .. (1,1);
    \draw[vstdhl] (3,0) ..controls (3,.5) and (2,.5) .. (2,1);}
    \]
     Where $\omega_i \in NH_{b_i}$ are the elements described in Eq. \eqref{eq:omegas} and specified by the representative in $\Z^{b_2}_2 \Z^{b_3}_3$. The 2 black lines represent parallel, unbraided bundles of black strands. The element $\omega_2$ specifies the set $I$ for which the diagram above may be rewritten as the product, 
     \begin{align}
     \tikzdiag[yscale=2, xscale=2]{
    \draw[bund] (0.25,0) node[below]{$b_1+|I|$}  --(0.25, 1);
    \draw[bund] (2.5,0) node[below]{$b_3+b_2-|I|$} --(2.5,0.15);
    \draw[bund](2.4,0.15)..controls (2.4,0.5) and (.75,0.5) ..(.75,1);
    \draw[bund](2.55,0.15)..controls (2.55,0.5) and (1.5,0.5) ..(1.5,1);
    \draw[bund](2.6,0.15)..controls (2.6,0.5) and (3.5,0.5) ..(3.5,1);
    \filldraw [fill=white, draw=black] (2.3,.10) rectangle (2.7,.30) node[midway] {$\omega_3'$};
    \draw[vstdhl] (0,0)--(0,1);
    \draw[vstdhl] (2,0)..controls (2,0.5) and (3,0.5) ..(3,1);
    \draw[vstdhl] (1,0)--(1,1);
    \draw[vstdhl] (3,0) ..controls (3,.5) and (2,.5) .. (2,1);
    }
    \cdot \hspace{4mm}p_0^I
    \end{align}
The diagram on the left is always an element of $\Phi^{s_i}_{\rho_I}$ and this procedure allows us to define the sets $\hat{\Phi}^{s_i}_{\rho_I}\subset \Phi^{s_i}_{\rho_I}$ that appear in the construction above. It is immediate through this factorization that,
$$p_0\bigg (\bigcup_I \hat{\Phi}^{s_i}_{\rho_I}\bigg)= \Phi^{s_is_j}_\rho$$
$$p_0\bigg (\hat{\Phi}^{s_i}_{\rho_I}\bigg) \bigcap p_0\bigg (\hat{\Phi}^{s_i}_{\rho_J}\bigg) = \emptyset, \hspace{5mm} I\neq J$$
We are left to show $\hat{\Phi}^{s_i}_{\rho_I}$ has a `minimal' element $\hat{\phi}_I$ such that $\hat{\Phi}^{s_i}_{\rho_I}$ may be redefined by taking all elements $\phi \gtrsim \hat{\phi}_I$. But this element can be easily spotted from the procedure above and we recognize it as, 
\begin{align}
    \hat{\phi}_I = 
     \tikzdiagh[yscale=2, xscale=2]{0}{
    \draw[vstdhl] (0,0)--(0,1);
    \draw[vstdhl] (2,0)..controls (2,0.5) and (3,0.5) ..(3,1);
    \draw[vstdhl] (1,0)--(1,1);
    \draw[vstdhl] (3,0) ..controls (3,.5) and (2,.5) .. (2,1);
    \draw (0.20,0)  --(0.20, 1);
    \draw (0.40,0)  --(0.40, 1);
    \draw (2.20,0)..controls (2.20,0.5) and (3.20,0.5) .. (3.20,1);
    \draw (2.40,0)..controls (2.40,0.5) and (3.40,0.5) .. (3.40,1);
    \draw (2.60,0)..controls (2.60,0.5) and (0.6,0.5) .. (0.6,1);
    \draw (2.80,0)..controls (2.80,0.5) and (0.8,0.7) .. (0.8,1);
    \tikzbrace{.2}{.4}{.1}{$b_1+|I|$};
    \tikzbrace{2.2}{2.3}{.1}{$b_2-|I|$};
    \tikzbrace{2.7}{2.8}{.1}{$b_3$};
    \node at(0.3,.5){$...$};
    \node at(0.74,.9){$...$};
    \node at(3.28,.9){$...$};
    \node at(2.3,.1){$...$};
    \node at(2.68,.1){$...$};
    }
\end{align}
The same argument works for $s_2s_3$ and we are done. 

\end{proof}

\begin{thm} For $i,j>1$, we have natural isomorphisms, 
$$\Bb_i \circ \Bb_j \cong \Bb_j \circ \Bb_i \hspace{10mm} |i-j|>1$$
$$\Bb_i \circ \Bb_{i+1} \circ \Bb_i \cong \Bb_{i+1} \circ \Bb_i \circ \Bb_{i+1}$$
\end{thm}
\begin{proof}
    The sliding of distant crossings is obvious. To prove the R3 relation, by Corollary \ref{cor:Rtensor}, it is sufficient to show $\Rb_i \otimes^L \Rb_{i+1} \otimes^L \Rb_i \cong \Rb_i \otimes \Rb_{i+1} \otimes \Rb_i$ as well as the analagous statement for the right-hand side. But again, we can use Lemma \ref{lem:it.exts.tensor} with Lemma \ref{lem:RRit.ext.} to show this, and we are done.
\end{proof}
\subsection{Inverse Braiding}
If $\Bb_{\sigma_i} = \Rb_i \otimes^L -$ is the functor we wish to associate to a negative crossing, its inverse, $\Bb^{-1}_{\sigma_i}$, is what we wish to associate to a positive crossing. As explained in \cite{Web}, the appropiate object to consider is, 
$$\Bb^{-1}_{\sigma_i} := RHom_{\Tb^{\sigma_i\cdot \underline{\mu}}}(\Rb_{\sigma_i},-)$$
\begin{prop}\label{prop:Inverse_Braiding}
    $$\Bb^{-1}_{\sigma_i} \circ \Bb_{\sigma_i} \cong \Bb_{\sigma_i}\circ \Bb^{-1}_{\sigma_i}\cong \mathbb{1}$$
\end{prop}
\begin{proof}
    It is sufficient to show this on any projective. Let $\Pb^{\underline{\mu}}_\rho$ be a projective with $b$ total black strands. In which case, $\Bb_i$ is simply a derived tensor product with $\Rb^{b}_i$ (and similarly for $\Bb^{-1}_i$), the $\Rb_i$ submodule on $b$ black strands. It is therefore sufficient to show $\Rb^b_i \otimes^L -$ is an equivalence. The cofibrant replacement of $\Rb^b_i$ has finite length by Proposition \ref{prop:Rresolution}, so it is perfect. Now, consider the map, 
    $$Hom(\Tb^{\underline{\mu}}1_\rho, \Tb^{\underline{\mu}}1_{\rho'}) \longrightarrow Hom(\Rb_i^b 1_\rho, \Rb_i^b 1_{\rho'})$$

    given by considering an element of $1_\rho \Tb^{\underline{\mu}}1_{\rho'}$ as map on the right hand side through composition. This map is easily seen to be injective. Similarly, 
    $$Hom(\Tb^{\underline{\mu}}, \Tb^{\underline{\mu}}1_{\rho'}) \longrightarrow Hom(\Rb_i^b, \Rb_i^b 1_{\rho'})$$
    is injective. Furthermore, notice that we found in the proof of Lemma \ref{lem:RRit.ext.} the identity, 
    $$gdim_{x,q} \hspace{1mm} (\Rb^b_i1_\rho) = \sum_{\phi\in \Phi^{w_i}_\rho} gdim_{x,q} \hspace{1mm} \Sb^{\underline{\mu}}_{\rho_\phi}$$ 
    Since $gdim_{x,q} \hspace{1mm} (R^b_i1_\rho) = gdim_{x,q} \hspace{1mm} Hom(\Rb^b_i \otimes \Tb^{\underline{\mu}}, \Rb^b_i1_\rho)$, then Corollary \ref{cor:gdimPandS} implies, 
    $$gdim_{x,q} \hspace{1mm}Hom(\Tb^{\underline{\mu}}, \Tb^{\underline{\mu}}1_{\rho'}) = gdim_{x,q} \hspace{1mm} Hom(\Rb_i^b, \Rb_i^b 1_{\rho'})$$

    and the map above must then be an isomorphism, thereby showing $\Rb^b_i \otimes^L -$ is an equivalence. 
    
\end{proof}

\section{Finite Dimensional Strands}\label{sec:fund_strands}
Thus far, we have worked explicitly with braids of strands colored by highest weights of Verma module representations or finite-dimensional irreducible representations of $\mathfrak{sl_2}$. Of course, if our goal is to arrive at knot invariants, we must also develop a framework for including cups and caps for the knots we wish to study. As in \cite{Web}, this is done by introducing cap and cup functors. In this section, we will do this for the case where the strands are colored by the fundamental representation of $\mathfrak{sl_2}$, closely following the steps in \cite{Na2} and \cite{Web2} and extending it to our framework. Of course, we choose the fundamental representation for simplicity, but generalizing the results in this section to arbitrary finite-dimensional representations is straightforward using the more detailed steps in \cite{Web}. The introduction of knots colored by fundamental representation will allow us to connect our theory to the more ambitious goal of defining Khovanov homology for knot and braid complements. 
\subsection{Cups and Caps}
We begin by defining the bimodules that will give rise to the cup and cap functors. We mostly follow \cite{Na2} and \cite{Web}. Let $\underline{\lambda} = (\lambda_1,\lambda_2,...,\lambda_r )$ and $\underline{\lambda^+_i} = (\lambda_1, \lambda_2,...,\lambda_i, 1^*, 1, \lambda_{i+1},...,\lambda_r)$. 
\begin{defn}\label{defn:cup}
    For $r\geq i \geq 1$, we let the cup bimodule $\Cb^{\underline{\lambda}^+_i}_{\underline{\lambda}}$ be the $\Tb^{\underline{\lambda}^+_i} - \Tb^{\underline{\lambda}}$-bimodule generated by diagrams,
    
    \[
    \tikzdiagh[yscale=2,xscale=1.5]{0}{
    \draw[pstdhl] (0,0) node[below]{\small $\lambda_{1}$} -- (0,1);
    \draw[pstdhl] (1,0) node[below]{\small $\lambda_{2}$} -- (1,1);
    \draw[pstdhl] (2,0) node[below]{\small $\lambda_{i}$} -- (2,1);
    \draw (3.5,0.5)--(3.5,1);
    \draw[stdhl] (4,1) node[above]{\small $1$} arc (0:-180:0.5) node[above]{\small $1^*$};
    \draw[pstdhl] (5,0) node[below]{\small $\lambda_{i+1}$} -- (5,1);
    \draw[pstdhl] (6,0) node[below]{\small $\lambda_{r}$} -- (6,1);
    \node at(1.5,.5){\large $...$};  
    \node at(5.5,.5){\large $...$};
    \draw (0.33,0)--(0.33,1);
    \draw (0.66,0) -- (0.66,1); 
    \node at(0.5,.5){$...$};
    \draw (2.33, 0) -- (2.33,1);
    \node at(2.5,.5){$...$};
    \draw (2.66, 0) -- (2.66,1);
    \draw (4.33, 0) -- (4.33,1);
    \node at(4.5,.5){$...$};
    \draw (4.66, 0) -- (4.66,1);
    \draw (6.33, 0) -- (6.33,1);
    \node at(6.5,.5){$...$};
    \draw (6.66, 0) -- (6.66,1);
    \tikzbrace{0.33}{0.66}{.1}{$b_1$};
    \tikzbrace{2.33}{2.66}{.1}{$b_i$};
    \tikzbrace{4.33}{4.66}{.1}{$b_i'$};
    \tikzbrace{6.33}{6.66}{.1}{$b_r$};
    }
    \]
Subject to the dg-enhanced KLRW relations of Section \ref{sec: prelim} as well as the new cup local relations, 
\be 
\tikzdiag[scale=1]{
\draw[stdhl] (0,0.5) --(0,1);
\draw[stdhl] (1,0.5) -- (1,1);
\draw[stdhl] (1,0.5) arc (0:-180:0.5);
\draw (0.5,0) ..controls (0.5,0.5) and (-0.5,0.5) .. (-0.5, 1);
} \hspace{3mm}= \hspace{3mm}\tikzdiag[scale=1]{
\draw[stdhl] (0,0.5) --(0,1);
\draw[stdhl] (1,0.5) -- (1,1);
\draw[stdhl] (1,0.5) arc (0:-180:0.5);
\draw (0.5,0) ..controls (0.5,0.5) and (1.5,0.5) .. (1.5, 1);
}\hspace{3mm} = \hspace{3mm} 0
\ee 

\be
\tikzdiag[scale=1]{
\draw[stdhl] (0,0.5) --(0,1);
\draw[stdhl] (1,0.5) -- (1,1);
\draw[stdhl] (1,0.5) arc (0:-180:0.5);
\draw (0.5,0)--(0.5, 1);
\draw (1.5,-0.5) ..controls (1.5,0.5) and (-0.5,0.5) .. (-0.5, 1);
} = - \tikzdiag[scale=1]{
\draw[stdhl] (0,0.5) --(0,1);
\draw[stdhl] (1,0.5) -- (1,1);
\draw[stdhl] (1,0.5) arc (0:-180:0.5);
\draw (0.5,0)--(0.5, 1);
\draw (1.5,-0.5) ..controls (1,-0.25) and (-0.75,-0.25) .. (-0.5, 1);
}\label{eq:cup_slide}
\ee

\be
\tikzdiag[xscale=-1]{
\draw[stdhl] (0,0.5) --(0,1);
\draw[stdhl] (1,0.5) -- (1,1);
\draw[stdhl] (1,0.5) arc (0:-180:0.5);
\draw (0.5,0)--(0.5, 1);
\draw (1.5,-0.5) ..controls (1.5,0.5) and (-0.5,0.5) .. (-0.5, 1);
} = \tikzdiag[xscale=-1]{
\draw[stdhl] (0,0.5) --(0,1);
\draw[stdhl] (1,0.5) -- (1,1);
\draw[stdhl] (1,0.5) arc (0:-180:0.5);
\draw (0.5,0)--(0.5, 1);
\draw (1.5,-0.5) ..controls (1,-0.25) and (-0.75,-0.25) .. (-0.5, 1);
}\label{eq:cup_slide2}
\ee

Where we let $(b_1,b_2,...,b_i, b_i' , b_{i+1},..., b_r)$ range over all $\mathbb{N}^{r+1}$. Additionally, we set the degree of the cup to, 
\be deg \hspace{2mm }\tikzdiag[xscale=1,yscale=1.25]{
\draw[stdhl] (1,0.5) arc (0:-180:0.5);
\draw (0.5,0)--(0.5, 0.5);
} \hspace{2mm}= 1\ee
\end{defn}

Similarly, we can define the cap bimodule $\Kb_{\underline{\lambda}^+_i}^{\underline{\lambda}}$ to be the $\Tb^{\underline{\lambda}}-\Tb^{\underline{\lambda}^+_i} $-bimodule given by, 

$$\Kb_{\underline{\lambda}^+_i}^{\underline{\lambda}} := q^{-1 }\dot{\Cb}^{\underline{\lambda}^+_i}_{\underline{\lambda}}[-1]$$

We interpret this as simply setting the degree of the cap to, 
\be deg \hspace{2mm }\tikzdiag[xscale=1,yscale=-1.25]{
\draw[stdhl] (1,0.5) arc (0:-180:0.5);
\draw (0.5,0)--(0.5, 0.5);
} \hspace{2mm}= q^{-1} h^{-1}\ee

We define the cup and cap functors ($\Cbb$ and $\Kbb$ respectively) just as in \cite{Na2}, 

$$\Cbb^{\underline{\lambda}^+_i}_{\underline{\lambda}} : = \Cb^{\underline{\lambda}^+_i}_{\underline{\lambda}} \otimes^L_\Tb - $$

$$\Kbb_{\underline{\lambda}^+_i}^{\underline{\lambda}} : = \Kb_{\underline{\lambda}^+_i}^{\underline{\lambda}} \otimes^L_\Tb - $$

.
\begin{rem}
We will largely use simply $\Cb_i$ and $\Kb_i$ when $\underline{\lambda}$ is understood. 
\end{rem} 
We may diagrammatically identify these functors with the cups and caps below,
\[
\begin{tikzcd}[row sep=large, column sep = large]
\tikzdiagh[scale=1]{0}{
\draw[stdhl,mid<-] (2,0) node[above]{$1$}  arc (0:-180:1) node[above]{$1^*$};
} \ar[rr,rightsquigarrow]&& \mbox{\large $\Cbb$} \\
\tikzdiagh[yscale=-1]{0}{
\draw[stdhl,mid->] (2,0) node[below]{$1$}  arc (0:-180:1) node[below]{$1^*$};
} \ar[rr,rightsquigarrow]&& \mbox{\large $\Kbb$} 
\end{tikzcd}
\]
Immediately, we can see, 

\begin{prop}
    The functors $\Cbb_i$ and $\Kbb_i$ commute with 1-morphisms. 
\end{prop}
\begin{proof}
    We take $\Cbb_i$ and the result for $\Kbb$ follows by adjunction. It is sufficient to show this for the $\mathcal{F}$ and $\mathcal{E}$ functors. Take $\mathcal{F}$, for instance. Recall that $\mathcal{F}$ acts by a derived tensor product with a cofibrant module, so that the result follows if $\beta_F \otimes \Cbb_i \cong \Cbb_i \otimes \beta_F$. But this follows at the diagram level since we can always slide black strands above the cup to the bottom and vice versa using the relation \ref{eq:cup_slide2}. 
\end{proof}

Following \cite{Na2}, we can write the basis for these bimodules as follows. Let $g^i_\ell$ be the following diagram, 

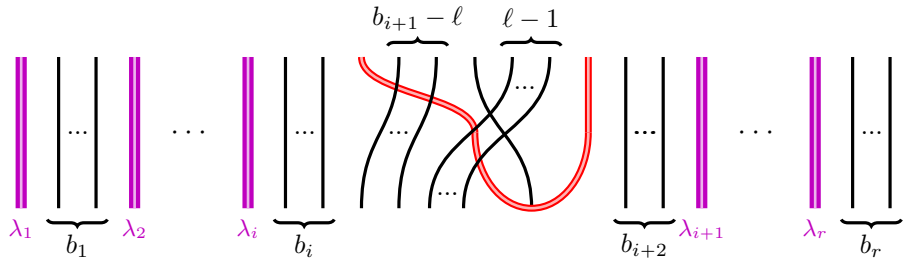
\[
    \tikzdiagh[yscale=2,xscale=1.5]{0}{
    \draw[pstdhl] (0,0) node[below]{\small $\lambda_{1}$} -- (0,1);
    \draw[pstdhl] (1,0) node[below]{\small $\lambda_{2}$} -- (1,1);
    \draw[pstdhl] (2,0) node[below]{\small $\lambda_{i}$} -- (2,1);
    \draw (4.5,0)..controls (4.5,0.5)  and (4,.5) ..(4,1);
    \draw[stdhl] (4,.5)..controls (4,.75) and (3,.75) ..(3,1);
    \draw[stdhl] (5,.5)--(5,1);
    \draw[stdhl] (5,0.5) arc (0:-180:0.5);
    \draw[pstdhl] (6,0) node[below]{\small $\lambda_{i+1}$} -- (6,1);
    \draw[pstdhl] (7,0) node[below]{\small $\lambda_{r}$} -- (7,1);
    \node at(1.5,.5){\large $\cdots$};  
    \node at(5.5,.5){\large $...$};
    \draw (0.33,0)--(0.33,1);
    \draw (0.66,0) -- (0.66,1); 
    \node at(0.5,.5){$...$};
    \draw (2.33, 0) -- (2.33,1);
    \node at(2.5,.5){$...$};
    \draw (2.66, 0) -- (2.66,1);
    \draw (3,0)..controls (3,.5) and (3.33,.5) .. (3.33,1);
    \draw (3.33,0)..controls (3.33,.5) and (3.66,.5) .. (3.66,1);
    \draw (3.6,0)..controls (3.6,.5) and (4.33,.5) .. (4.33,1);
    \draw (3.9,0)..controls (3.9,.5) and (4.66,.5) .. (4.66,1);
    \draw (5.33, 0) -- (5.33,1);
    \node at(5.5,.5){$...$};
    \node at(3.33,.5){$...$};
    \node at(3.76,.1){$...$};
    \node at(4.44,.8){$...$};
    \node at(6.5,.5){\large $\cdots$};
    \draw (5.66, 0) -- (5.66,1);
    \draw (7.33, 0) -- (7.33,1);
    \node at(7.5,.5){$...$};
    \draw (7.66, 0) -- (7.66,1);
    \tikzbrace{0.33}{0.66}{.1}{$b_1$};
    \tikzbrace{2.33}{2.66}{.1}{$b_i$};
    \tikzbrace{5.33}{5.66}{.1}{$b_{i+2}$};
    \tikzbrace{7.33}{7.66}{.1}{$b_r$};
    \tikzbraceop{3.33}{3.66}{.9}{$b_{i+1}-\ell$};
    \tikzbraceop{4.33}{4.66}{.9}{$\ell-1$};
    }
    \]

This is an element lying in, 
$$g^i_\ell \in 1_\kappa \Cb_i 1_{\kappa'}$$
where, 
$$\kappa = (b_1,b_2,..., b_i,b_{i+1}, b_{i+2}, b_{i+3}, ... , b_r)$$
$$\kappa' = (b_1,b_2,..., b_{i-1}, b_i+b_{i+1}+ b_{i+2}-1, b_{i+3}, ... , b_r)$$
As such, $g^i_\ell$ defines a map via left multiplication, 
$$g^i_\ell : 1_{\kappa'} \Tb 1_\rho \longrightarrow 1_\kappa \Cb_i 1_\rho$$

Furthermore, this element gives us a nice basis for the bimodules. Recall that the set $B^\kappa_\rho (\Tb^{\underline{\lambda}})$ is a basis for $1_\kappa \Tb^{\underline{\lambda}}1_\rho$ (Theorem \ref{thm:Tbasis}). Then, a basis for $\Cb_i$ can be established by the following result. 
\begin{prop}\label{prop:cup_basis}
The set 
$$B^\kappa_\rho (\Cb_i) := \bigsqcup^{b_{i+1}}_{\ell=1} g^i_\ell \cdot B^{\kappa'}_\rho (\Tb^{\underline{\lambda}})$$ 
is a basis for $1_\kappa\Cb_i1_\rho$.
\end{prop}
\begin{proof}
    The proof is almost identical to Corollary 5.2 of \cite{Na2}, changing details to our case. 
\end{proof}

To do calculations, it will be necessary to know the projective resolutions of $\Cb_i$. Consider a projective of the form $\Pb^{\underline{\lambda}}_\rho $ (with $\rho = (b_1,...,b_r)$) and the following chain complex, 

\be \begin{tikzcd}[row sep=tiny,column sep=small]
&q \hspace{2mm}\tikzdiagh[yscale=1.5]{0}{
\draw[pstdhl](0,0) node[below]{$\lambda_i$} --(0,0.5);
\draw[stdhl](0.5,0)--(.5,0.5);
\draw[stdhl](1.5,0)--(1.5,0.5);
\draw (2,0)--(2,.5);
\draw (2.25,0)--(2.25,.5);
\draw (1.75,0)--(1.75,0.5);
\filldraw [fill=white, draw=black] (-0.25,.5) rectangle (2.5,1) node[midway] { $\Tb$};
\node at (-0.3,.25) {$...$};
\node at (2.13,.25) {\small $..$};
\node at (2.6,.25) {$...$};
\tikzbrace{2.1}{2.2}{0.1}{$b_i$}
}
\ar{rd}
\ar[dash]{rd}{
\tikzdiag[scale=1]{
\draw[pstdhl](0,0)--(0,0.5);
\draw[stdhl](0.5,0)--(0.5,0.5);
\draw[stdhl](1,0)--(1,0.5);
\draw (1.25,0.5)..controls (1.25,0.25) and (0.75,0.25)..(0.75,0);
}
}
& &\\
     q^2 \hspace{2mm}\tikzdiagh[yscale=1.5]{0}{
\draw[pstdhl](0,0) node[below]{$\lambda_i$} --(0,0.5);
\draw[stdhl](0.5,0)--(.5,0.5);
\draw[stdhl](1.5,0)--(1.5,0.5);
\draw (2,0)--(2,.5);
\draw (2.25,0)--(2.25,.5);
\draw (1,0)--(1,0.5);
\filldraw [fill=white, draw=black] (-0.25,.5) rectangle (2.5,1) node[midway] { $\Tb$};
\node at (-0.3,.25) {$...$};
\node at (2.13,.25) {\small $..$};
\node at (2.6,.25) {$...$};
\tikzbrace{2.1}{2.2}{0.1}{$b_i$}
}
\ar{ru}
\ar[dash]{ru}{-\hspace{1mm}
\tikzdiag[scale=1]{
\draw[pstdhl](0,0)--(0,0.5);
\draw[stdhl](0.5,0)--(0.5,0.5);
\draw[stdhl](1,0)--(1,0.5);
\draw (1.25,0)..controls (1.25,0.25) and (0.75,0.25)..(0.75,0.5);
}
}
\ar{rd}
\ar[swap,dash]{rd}{
\tikzdiag[scale=1]{
\draw[pstdhl](0,0)--(0,0.5);
\draw[stdhl](0.5,0)--(0.5,0.5);
\draw[stdhl](1,0)--(1,0.5);
\draw (.25,0)..controls (.25,0.25) and (0.75,0.25)..(0.75,0.5);
}
}
& \bigoplus &
\tikzdiagh[yscale=1.5]{0}{
\draw[pstdhl](0,0) node[below]{$\lambda_i$} --(0,0.5);
\draw[stdhl](0.5,0)--(.5,0.5);
\draw[stdhl](1.5,0)--(1.5,0.5);
\draw (2,0)--(2,.5);
\draw (2.25,0)--(2.25,.5);
\draw (1,0)--(1,0.5);
\filldraw [fill=white, draw=black] (-0.25,.5) rectangle (2.5,1) node[midway] { $\Tb$};
\node at (-0.3,.25) {$...$};
\node at (2.13,.25) {\small $..$};
\node at (2.6,.25) {$...$};
\tikzbrace{2.1}{2.2}{0.1}{$b_i$}
}
\ar{rrr}
\ar[dash]{rrr}{
\tikzdiag[yscale=2]{
\draw[pstdhl](0,0)--(0,0.25);
\draw[stdhl](1,0.25) arc (0:-180:0.25);
\draw (.75,0)--(0.75,0.25);
}
}
&&&
\tikzdiagh[yscale=1.5]{0}{
\draw[pstdhl](0,0) node[below]{$\lambda_i$} --(0,0.5);
\draw (.33,0)--(.33,.5);
\draw (.66,0)--(.66,.5);
\filldraw [fill=white, draw=black] (-0.25,.5) rectangle (1,1) node[midway] { $\Cb_i$};
\node at (-0.3,.25) {$...$};
\node at (.5,.25) {\small $..$};
\tikzbrace{.33}{.66}{0.1}{$b_i$}
}\\
&  q \hspace{2mm}\tikzdiagh[yscale=1.5]{0}{
\draw[pstdhl](0,0) node[below]{$\lambda_i$} --(0,0.5);
\draw[stdhl](0.5,0)--(.5,0.5);
\draw[stdhl](1.5,0)--(1.5,0.5);
\draw (2,0)--(2,.5);
\draw (2.25,0)--(2.25,.5);
\draw (.25,0)--(.25,0.5);
\filldraw [fill=white, draw=black] (-0.25,.5) rectangle (2.5,1) node[midway] { $\Tb$};
\node at (-0.3,.25) {$...$};
\node at (2.13,.25) {\small $..$};
\node at (2.6,.25) {$...$};
\tikzbrace{2.1}{2.2}{0.1}{$b_i$}
} 
\ar{ru}
\ar[swap,dash]{ru}{
\tikzdiag[scale=1]{
\draw[pstdhl](0,0)--(0,0.5);
\draw[stdhl](0.5,0)--(0.5,0.5);
\draw[stdhl](1,0)--(1,0.5);
\draw (.75,0)..controls (.75,0.25) and (0.25,0.25)..(0.25,0.5);
}
}
&&
\end{tikzcd} \label{eq:cup_proj_res}
\ee 

\begin{prop}
The complex above \eqref{eq:cup_proj_res} is a projective resolution for $\Cbb_i(\Pb^{\underline{\lambda}}_\rho)$. 
\end{prop}
\begin{proof}
Again, we refer to \cite{Na2}, since the statement can be proven similarly. 
\end{proof}

We can similarly work out the projective resolutions of the cap bimodules. Consider the map, 
\be \begin{tikzcd}[row sep=tiny]
q^{-1}\bigoplus_{[b_{i+1}]_q} \hspace{2mm}\tikzdiagh[yscale=1.5]{0}{
\draw[pstdhl](0,0) --(0,0.5);
\draw (.33,0)--(.33,.5);
\draw (.66,0)--(.66,.5);
\filldraw [fill=white, draw=black] (-0.25,.5) rectangle (1,1) node[midway] { $\Tb$};
\node at (-0.3,.25) {$...$};
\node at (.5,.25) {\small $..$};
\tikzbrace{.33}{.66}{0.1}{$b_i+b_{i+1}+b_{i+2}-1$}
}[-1]
\ar{rr}
&&
\tikzdiagh[yscale=1.5]{0}{
\draw[pstdhl](0,0) node[below]{$\lambda_i$} --(0,0.5);
\draw[stdhl](1,0)--(1,0.5);
\draw[stdhl](2,0)--(2,0.5);
\draw (.33,0)--(.33,.5);
\draw (.66,0)--(.66,0.5);
\draw (1.33,0)--(1.33,.5);
\draw (1.66,0)--(1.66,0.5);
\draw (2.33,0)--(2.33,.5);
\draw (2.66,0)--(2.66,0.5);
\filldraw [fill=white, draw=black] (-0.25,.5) rectangle (2.9,1) node[midway] { $\Kb_i$};
\node at (-0.3,.25) {$...$};
\node at (1.5,.25) {\small $..$};
\node at (.5,.25) {\small $..$};
\node at (2.5,.25) {\small $..$};
\node at (2.9,.25) {$...$};
\tikzbrace{1.33}{1.66}{0.1}{$b_{i+1}$};
\tikzbrace{.33}{.66}{0.1}{$b_{i}$};
\tikzbrace{2.33}{2.66}{0.1}{$b_{i+2}$};
}
\end{tikzcd} \label{eq:cap_proj_res}
\ee 

given by right multiplication on the summand in the appropriate degree by $\dot{g}^i_\ell$. That is, by the diagram, 
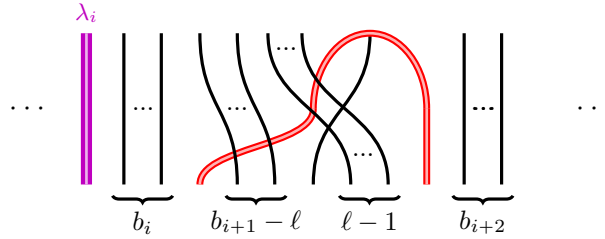
\[
    \tikzdiagh[yscale=-2,xscale=1.5]{0}{
    \draw[pstdhl] (2,0) node[above]{\small $\lambda_{i}$} -- (2,1);
    \draw (4.5,0)..controls (4.5,0.5)  and (4,.5) ..(4,1);
    \draw[stdhl] (4,.5)..controls (4,.75) and (3,.75) ..(3,1);
    \draw[stdhl] (5,.5)--(5,1);
    \draw[stdhl] (5,0.5) arc (0:-180:0.5);
    \node at(1.5,.5){\large $\cdots$};  
    \node at(5.5,.5){\large $...$};
    \draw (2.33, 0) -- (2.33,1);
    \node at(2.5,.5){$...$};
    \draw (2.66, 0) -- (2.66,1);
    \draw (3,0)..controls (3,.5) and (3.33,.5) .. (3.33,1);
    \draw (3.33,0)..controls (3.33,.5) and (3.66,.5) .. (3.66,1);
    \draw (3.6,0)..controls (3.6,.5) and (4.33,.5) .. (4.33,1);
    \draw (3.9,0)..controls (3.9,.5) and (4.66,.5) .. (4.66,1);
    \draw (5.33, 0) -- (5.33,1);
    \node at(5.5,.5){$...$};
    \node at(3.33,.5){$...$};
    \node at(3.76,.1){$...$};
    \node at(4.44,.8){$...$};
    \node at(6.5,.5){\large $\cdots$};
    \draw (5.66, 0) -- (5.66,1);
    \tikzbrace{2.33}{2.66}{1.6}{$b_i$};
    \tikzbrace{5.33}{5.66}{1.6}{$b_{i+2}$};
    \tikzbrace{3.33}{3.66}{1.6}{$b_{i+1}-\ell$};
    \tikzbrace{4.33}{4.66}{1.6}{$\ell-1$};
    }
    \]

\begin{prop}
The map described above \eqref{eq:cap_proj_res} is an isomorphism. 
\end{prop}
\begin{proof}
    This is immediate from taking the mirror along the horizontal of Proposition \ref{prop:cup_basis}. 
\end{proof}
\subsection{Ribbon Structure}
Now, we are ready to prove topological invariance of the red strands under Reidemeister 1 moves (recall that R2 and R3 have been shown in Section \ref{sec:braiding}). 
\be
\tikzdiagh[yscale=2]{0}{
\draw[pstdhl](0,0) node[below]{$\lambda_i$}--(0,1);
\draw[pstdhl](3,0) node[below]{$\lambda_{i+2}$}--(3,1);
\draw[stdhl] (1,0.5)..controls (1,0.1) and (2,0.1) .. (2,1);
\draw[stdhl](2,0) node[below]{$1$}..controls (2,0.9) and (1,0.9)..(1,0.5);
\node[color=mypurple] at (-0.5,0.5) {$\cdots$};
\node[color=mypurple] at (3.5,0.5) {$\cdots$};
}\hspace{3mm}
\sim \hspace{3mm}
\tikzdiagh[yscale=2]{0}{
\draw[pstdhl](0,0) node[below]{$\lambda_i$}--(0,1);
\draw[pstdhl](3,0) node[below]{$\lambda_{i+2}$}--(3,1);
\draw[stdhl](1.5,0) node[below]{$1$} --(1.5,1);
\node[color=mypurple] at (-0.5,0.5) {$\cdots$};
\node[color=mypurple] at (3.5,0.5) {$\cdots$};
}\label{eq:R1}
\ee
That is, under an R1 move, the ribbon structure of the quantum knot invariants in question becomes prominent in the form of overall degree shifts. We closely follow \cite{Web} for most of this discussion. \\
Let us begin by defining the cups and caps with their orientations reversed. 
\begin{defn}
Let $\bar{\Cb}^{\underline{\lambda}^+_i}_{\underline{\lambda}}$ be the bimodule defined in the same way as $\Cb^{\underline{\lambda}^+_i}_{\underline{\lambda}}$, except we flip $1^*$ and $1$, so that we send $(1^*,1)\rightarrow (1,1^*)$ in Definition \ref{defn:cup}. Then, we define, 

$$\bar{\Cbb}^{\underline{\lambda}^+_i}_{\underline{\lambda}} := \bar{\Cb}^{\underline{\lambda}^+_i}_{\underline{\lambda}} \otimes^L_\Tb -$$

Define $\bar{\Kb}_{\underline{\lambda}^+_i}^{\underline{\lambda}}$ in the same way and subsequently, 
$$\bar{\Kbb}_{\underline{\lambda}^+_i}^{\underline{\lambda}}:= \bar{\Kb}_{\underline{\lambda}^+_i}^{\underline{\lambda}} \otimes^L_\Tb -$$

\end{defn}
We diagrammatically identify these with the following pictures, 
\[
\begin{tikzcd}[row sep=large, column sep = large]
\tikzdiagh[scale=1]{0}{
\draw[stdhl,mid->] (2,0) node[above]{$1^*$}  arc (0:-180:1) node[above]{$1$};
} \ar[rr,rightsquigarrow]&& \mbox{\large $\bar{\Cbb}$} \\
\tikzdiagh[yscale=-1]{0}{
\draw[stdhl,mid<-] (2,0) node[below]{$1^*$}  arc (0:-180:1) node[below]{$1$};
} \ar[rr,rightsquigarrow]&& \mbox{\large $\bar{\Kbb}$} 
\end{tikzcd}
\]
We wish to show these four cups/caps defined thus far form a consistent ribbon structure. To do so, it is necessary to study their behavior under braiding. 
\[
\tikzdiag[scale=1]{
\draw[stdhl,mid<-] (2,0)  arc (0:-180:1);
\draw[stdhl] (0,0)..controls (0,1) and (2,1) .. (2,2);
\draw[stdhl] (2,0) ..controls (2,1) and (0,1) .. (0,2);
} \hspace{4mm}=\hspace{4mm} q^{\frac{3}{2}} h \tikzdiag[scale=1]{
\draw[stdhl,mid->] (2,0)  arc (0:-180:1);}
\]
\begin{prop}\label{prop:cup_r1}
    $\Bb_{i+1} \circ \Cbb_i \cong q^{\frac{3}{2}}\bar{\Cbb}_i[1]$ 
\end{prop}
\begin{proof}
Take any projective $P^{\underline{\lambda}}_\rho$, apply $\Cbb_i$ and resolve the find the complex \eqref{eq:cup_proj_res}. We opt not to draw the strands to the left and to the right of the cup since they have no bearing on the calculation and write only the red strands coming from the cup. 
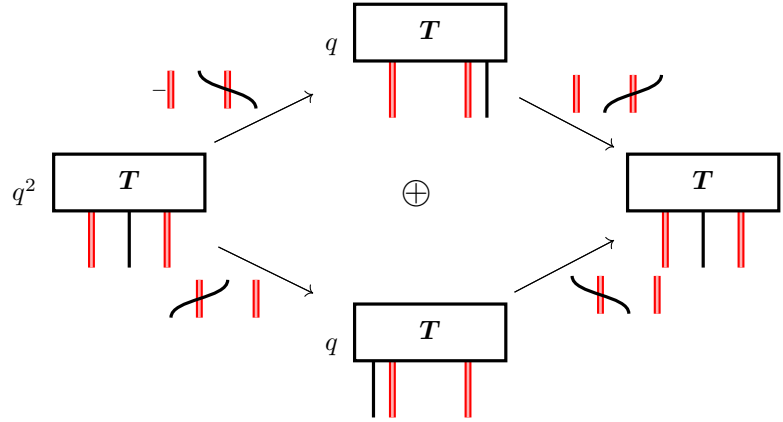
\[
\begin{tikzcd}[column sep=large,row sep=tiny]
& q\hspace{2mm}\tikzdiagh[yscale=1.5]{0}{
\draw[stdhl](0.5,0)--(.5,0.5);
\draw[stdhl](1.5,0)--(1.5,0.5);
\draw (1.75,0)--(1.75,0.5);
\filldraw [fill=white, draw=black] (0,.5) rectangle (2,1) node[midway] { $\Tb$};
}
\ar{rd}
\ar[dash]{rd}{
\tikzdiag[xscale=0.75]{
\draw[stdhl](0.5,0)--(0.5,0.5);
\draw[stdhl](1.5,0)--(1.5,0.5);
\draw (1,0)..controls (1,0.25) and (2,0.25)..(2,0.5);
}
}
& \\
q^2\hspace{2mm}\tikzdiagh[yscale=1.5]{0}{
\draw[stdhl](0.5,0)--(.5,0.5);
\draw[stdhl](1.5,0)--(1.5,0.5);
\draw (1,0)--(1,0.5);
\filldraw [fill=white, draw=black] (0,.5) rectangle (2,1) node[midway] {$\Tb$};
}
\ar{ru}
\ar[dash]{ru}{-
\tikzdiag[xscale=0.75,yscale=-1]{
\draw[stdhl](0.5,0)--(0.5,0.5);
\draw[stdhl](1.5,0)--(1.5,0.5);
\draw (1,0)..controls (1,0.25) and (2,0.25)..(2,0.5);
}
}
\ar{rd}
\ar[swap,dash]{rd}{
\tikzdiag[xscale=-0.75,yscale=-1]{
\draw[stdhl](0.5,0)--(0.5,0.5);
\draw[stdhl](1.5,0)--(1.5,0.5);
\draw (1,0)..controls (1,0.25) and (2,0.25)..(2,0.5);
}
}
&
\bigoplus &
\tikzdiagh[yscale=1.5]{0}{
\draw[stdhl](0.5,0)--(.5,0.5);
\draw[stdhl](1.5,0)--(1.5,0.5);
\draw (1,0)--(1,0.5);
\filldraw [fill=white, draw=black] (0,.5) rectangle (2,1) node[midway] { $\Tb$};
} \\
& 
q\hspace{2mm}\tikzdiagh[yscale=1.5]{0}{
\draw[stdhl](0.5,0)--(.5,0.5);
\draw[stdhl](1.5,0)--(1.5,0.5);
\draw (0.25,0)--(0.25,0.5);
\filldraw [fill=white, draw=black] (0,.5) rectangle (2,1) node[midway] { $\Tb$};
}
\ar{ru}
\ar[swap,dash]{ru}{
\tikzdiag[xscale=-0.75,yscale=1]{
\draw[stdhl](0.5,0)--(0.5,0.5);
\draw[stdhl](1.5,0)--(1.5,0.5);
\draw (1,0)..controls (1,0.25) and (2,0.25)..(2,0.5);
}
}
&
\end{tikzcd}
\]
We can braid this complex and resolve using Proposition \ref{prop:Rresolution}. The result is given by, 

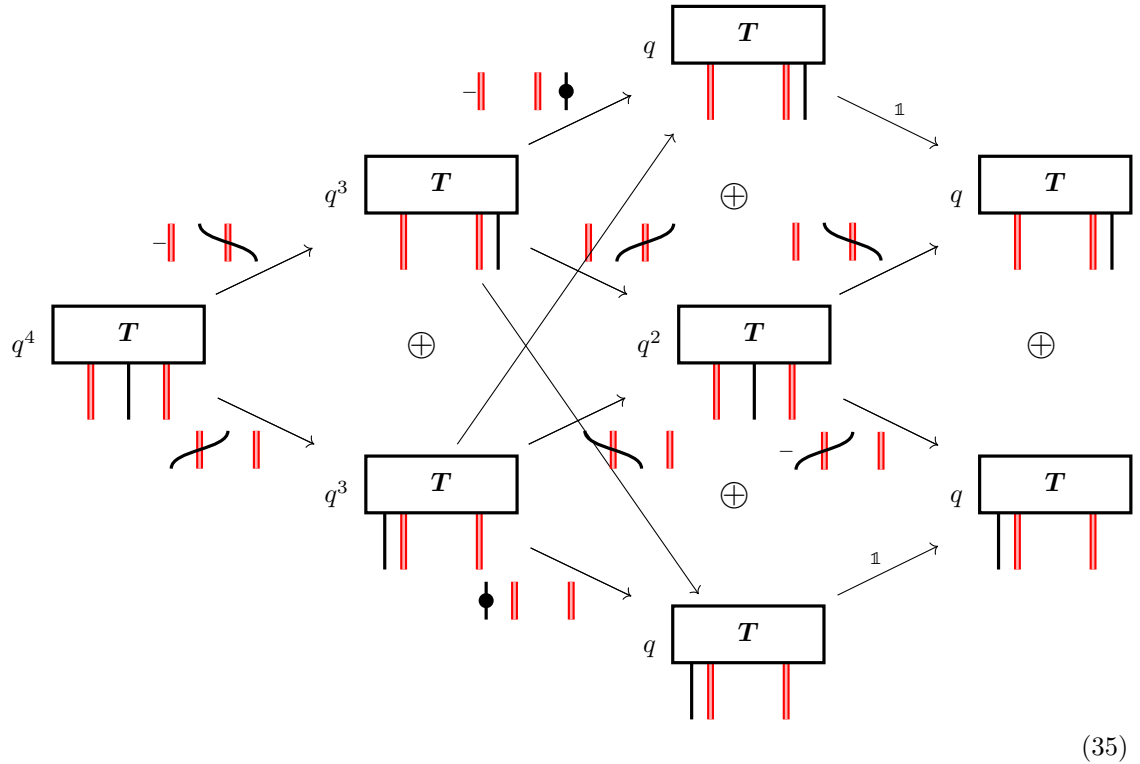
\be 
\begin{tikzcd}[column sep=large,row sep=tiny]
& & q\hspace{2mm}\tikzdiagh[yscale=1.5]{0}{
\draw[stdhl](0.5,0)--(.5,0.5);
\draw[stdhl](1.5,0)--(1.5,0.5);
\draw (1.75,0)--(1.75,0.5);
\filldraw [fill=white, draw=black] (0,.5) rectangle (2,1) node[midway] { $\Tb$};
}
\ar{rd}{\mathbb{1}}
& \\
& q^3\hspace{2mm}\tikzdiagh[yscale=1.5]{0}{
\draw[stdhl](0.5,0)--(.5,0.5);
\draw[stdhl](1.5,0)--(1.5,0.5);
\draw (1.75,0)--(1.75,0.5);
\filldraw [fill=white, draw=black] (0,.5) rectangle (2,1) node[midway] { $\Tb$};
}
\ar{rd}
\ar[dash]{rd}{
\tikzdiag[xscale=0.75]{
\draw[stdhl](0.5,0)--(0.5,0.5);
\draw[stdhl](1.5,0)--(1.5,0.5);
\draw (1,0)..controls (1,0.25) and (2,0.25)..(2,0.5);
}
}
\ar{rddd}
\ar{ru}
\ar[dash]{ru}{-
\tikzdiag[xscale=0.75]{
\draw[stdhl](0.5,0)--(0.5,0.5);
\draw[stdhl](1.5,0)--(1.5,0.5);
\draw (2,0)..controls (2,0.25) and (2,0.25)..(2,0.5) node[pos=0.5, tikzdot]{};
}
}
& \bigoplus &
q\hspace{2mm}\tikzdiagh[yscale=1.5]{0}{
\draw[stdhl](0.5,0)--(.5,0.5);
\draw[stdhl](1.5,0)--(1.5,0.5);
\draw (1.75,0)--(1.75,0.5);
\filldraw [fill=white, draw=black] (0,.5) rectangle (2,1) node[midway] { $\Tb$};
}\\
q^4\hspace{2mm}\tikzdiagh[yscale=1.5]{0}{
\draw[stdhl](0.5,0)--(.5,0.5);
\draw[stdhl](1.5,0)--(1.5,0.5);
\draw (1,0)--(1,0.5);
\filldraw [fill=white, draw=black] (0,.5) rectangle (2,1) node[midway] {$\Tb$};
}
\ar{ru}
\ar[dash]{ru}{-
\tikzdiag[xscale=0.75,yscale=-1]{
\draw[stdhl](0.5,0)--(0.5,0.5);
\draw[stdhl](1.5,0)--(1.5,0.5);
\draw (1,0)..controls (1,0.25) and (2,0.25)..(2,0.5);
}
}
\ar{rd}
\ar[swap,dash]{rd}{
\tikzdiag[xscale=-0.75,yscale=-1]{
\draw[stdhl](0.5,0)--(0.5,0.5);
\draw[stdhl](1.5,0)--(1.5,0.5);
\draw (1,0)..controls (1,0.25) and (2,0.25)..(2,0.5);
}
}
&
\bigoplus &q^2\hspace{2mm}
\tikzdiagh[yscale=1.5]{0}{
\draw[stdhl](0.5,0)--(.5,0.5);
\draw[stdhl](1.5,0)--(1.5,0.5);
\draw (1,0)--(1,0.5);
\filldraw [fill=white, draw=black] (0,.5) rectangle (2,1) node[midway] { $\Tb$};
} 
\ar{rd}
\ar[swap,dash]{rd}{-
\tikzdiag[xscale=-0.75,yscale=-1]{
\draw[stdhl](0.5,0)--(0.5,0.5);
\draw[stdhl](1.5,0)--(1.5,0.5);
\draw (1,0)..controls (1,0.25) and (2,0.25)..(2,0.5);
}
}
\ar{ru}
\ar[dash]{ru}{
\tikzdiag[xscale=0.75,yscale=-1]{
\draw[stdhl](0.5,0)--(0.5,0.5);
\draw[stdhl](1.5,0)--(1.5,0.5);
\draw (1,0)..controls (1,0.25) and (2,0.25)..(2,0.5);
}
}
& 
\bigoplus \\
& 
q^3\hspace{2mm}\tikzdiagh[yscale=1.5]{0}{
\draw[stdhl](0.5,0)--(.5,0.5);
\draw[stdhl](1.5,0)--(1.5,0.5);
\draw (0.25,0)--(0.25,0.5);
\filldraw [fill=white, draw=black] (0,.5) rectangle (2,1) node[midway] { $\Tb$};
}
\ar{ru}
\ar[swap,dash]{ru}{
\tikzdiag[xscale=-0.75]{
\draw[stdhl](0.5,0)--(0.5,0.5);
\draw[stdhl](1.5,0)--(1.5,0.5);
\draw (1,0)..controls (1,0.25) and (2,0.25)..(2,0.5);
}
}
\ar{ruuu}
\ar{rd}
\ar[swap,dash]{rd}{
\tikzdiag[xscale=-0.75,yscale=1]{
\draw[stdhl](0.5,0)--(0.5,0.5);
\draw[stdhl](1.5,0)--(1.5,0.5);
\draw (2,0)..controls (2,0.25) and (2,0.25)..(2,0.5) node[pos=.5,tikzdot]{};
}
}
& \bigoplus & q\hspace{2mm}\tikzdiagh[yscale=1.5]{0}{
\draw[stdhl](0.5,0)--(.5,0.5);
\draw[stdhl](1.5,0)--(1.5,0.5);
\draw (0.25,0)--(0.25,0.5);
\filldraw [fill=white, draw=black] (0,.5) rectangle (2,1) node[midway] { $\Tb$};
} \\
&&
q\hspace{2mm}\tikzdiagh[yscale=1.5]{0}{
\draw[stdhl](0.5,0)--(.5,0.5);
\draw[stdhl](1.5,0)--(1.5,0.5);
\draw (0.25,0)--(0.25,0.5);
\filldraw [fill=white, draw=black] (0,.5) rectangle (2,1) node[midway] { $\Tb$};
}
\ar{ru}{\mathbb{1}}&
\end{tikzcd}
\label{eq:braided_cup_res}\ee

all multiplied by all overall $q^{-\frac{1}{2}}$ factor coming from the red-red crossing. Above, we don't write two of the degree 2 differentials for the sake of legibility, but the reader should note they are obvious maps with signs consistent with the rest of the diagram. \\
First, note that the 0th homology is clearly trivial. Now, for $H_1$, we can compute the kernel of $d_1$ to find, 
\[
Ker\hspace{1mm} d_1 = \tikzdiag[yscale=1.5,scale=0.75]{
\draw[stdhl](0.5,0)--(.5,0.5);
\draw[stdhl](1.5,0)--(1.5,0.5);
\draw (1.8,0)..controls (1.8,.25) and (1,.25) ..(1,.5);
\filldraw [fill=white, draw=black] (0,.5) rectangle (2,0.8) node[midway] {$-x$};
}
\oplus \tikzdiag[yscale=1.5,scale=0.75]{
\draw[stdhl](0.5,0)--(.5,0.5);
\draw[stdhl](1.5,0)--(1.5,0.5);
\draw (1,0)--(1,0.5);
\filldraw [fill=white, draw=black] (0,.5) rectangle (2,0.8) node[midway] {$x$};
}
\oplus \tikzdiag[yscale=1.5,scale=0.75,xscale=-1]{
\draw[stdhl](0.5,0)--(.5,0.5);
\draw[stdhl](1.5,0)--(1.5,0.5);
\draw (1.8,0)..controls (1.8,.25) and (1,.25) ..(1,.5);
\filldraw [fill=white, draw=black] (0,.5) rectangle (2,0.8) node[midway] {$x$};
}
\]

Where $x$ ranges over all elements of $\hspace{2mm}\tikzdiag[yscale=1.5,scale=0.5]{
\draw[stdhl](0.5,0)--(.5,0.5);
\draw[stdhl](1.5,0)--(1.5,0.5);
\draw (1,0)--(1,0.5);
\filldraw [fill=white, draw=black] (0,.5) rectangle (2,1) node[midway] {$\Tb$};
}$. On the other hand, the image of $d_2$ is easily seen to be, 

\[
Im\hspace{1mm} d_2 = \tikzdiag[yscale=1.5,scale=0.75]{
\draw[stdhl](0.5,0)--(.5,0.5);
\draw[stdhl](1.5,0)--(1.5,0.5);
\draw (1.8,0)--(1.8,0.5) node[pos=0.5,tikzdot]{};
\filldraw [fill=white, draw=black] (0,.5) rectangle (2,0.8) node[midway] {$-z$};
}\oplus
\tikzdiag[yscale=1.5,scale=0.75]{
\draw[stdhl](0.5,0)--(.5,0.5);
\draw[stdhl](1.5,0)--(1.5,0.5);
\draw (1,0)..controls (1,.25) and (1.8,.25) ..(1.8,0.5);
\filldraw [fill=white, draw=black] (0,.5) rectangle (2,0.8) node[midway] {$z$};
} 
\oplus 
\tikzdiag[yscale=1.5,scale=0.75]{
\draw[stdhl](0.5,0)--(.5,0.5);
\draw[stdhl](1.5,0)--(1.5,0.5);
\draw (.2,0)..controls (.2,.25) and (1.8,.25) ..(1.8,0.5);
\filldraw [fill=white, draw=black] (0,.5) rectangle (2,0.8) node[midway] {$z$};
} \hspace{2mm} + \hspace{2mm}
\tikzdiag[yscale=1.5,scale=0.75,xscale=-1]{
\draw[stdhl](0.5,0)--(.5,0.5);
\draw[stdhl](1.5,0)--(1.5,0.5);
\draw (.2,0)..controls (.2,.25) and (1.8,.25) ..(1.8,0.5);
\filldraw [fill=white, draw=black] (0,.5) rectangle (2,0.8) node[midway] {$-y$};
}\oplus 
\tikzdiag[yscale=1.5,scale=0.75,xscale=-1]{
\draw[stdhl](0.5,0)--(.5,0.5);
\draw[stdhl](1.5,0)--(1.5,0.5);
\draw (1,0)..controls (1,.25) and (1.8,.25) ..(1.8,0.5);
\filldraw [fill=white, draw=black] (0,.5) rectangle (2,0.8) node[midway] {$y$};
} \oplus 
\tikzdiag[yscale=1.5,scale=0.75,xscale=-1]{
\draw[stdhl](0.5,0)--(.5,0.5);
\draw[stdhl](1.5,0)--(1.5,0.5);
\draw (1.8,0)--(1.8,0.5) node[pos=0.5,tikzdot]{};
\filldraw [fill=white, draw=black] (0,.5) rectangle (2,0.8) node[midway] {$y$};
}
\]

Where $z$ ranges over all elements of $\hspace{2mm}\tikzdiag[yscale=1.5,scale=0.5]{
\draw[stdhl](0.5,0)--(.5,0.5);
\draw[stdhl](1.5,0)--(1.5,0.5);
\draw (1.8,0)--(1.8,0.5);
\filldraw [fill=white, draw=black] (0,.5) rectangle (2,1) node[midway] {$\Tb$};
}\hspace{1mm}$ and $y$ ranges over $\hspace{2mm}\tikzdiag[xscale=-1,yscale=1.5,scale=0.5]{
\draw[stdhl](0.5,0)--(.5,0.5);
\draw[stdhl](1.5,0)--(1.5,0.5);
\draw (1.8,0)--(1.8,0.5);
\filldraw [fill=white, draw=black] (0,.5) rectangle (2,1) node[midway] {$\Tb$};
}$. We can therefore see that if we view $x\in H_1$ as an element of $Ker\hspace{1mm}d_1$, $x$ will be trivial in $H_1$ if it is of the form, 
\[ x=\tikzdiag[yscale=1.5,scale=0.75]{
\draw[stdhl](0.5,0)--(.5,0.5);
\draw[stdhl](1.5,0)--(1.5,0.5);
\draw (1,0)..controls (1,.25) and (1.8,.25) ..(1.8,0.5);
\filldraw [fill=white, draw=black] (0,.5) rectangle (2,0.8) node[midway] {$x'$};
} \hspace{5mm} or \hspace{5mm} x=\tikzdiag[xscale=-1, yscale=1.5,scale=0.75]{
\draw[stdhl](0.5,0)--(.5,0.5);
\draw[stdhl](1.5,0)--(1.5,0.5);
\draw (1,0)..controls (1,.25) and (1.8,.25) ..(1.8,0.5);
\filldraw [fill=white, draw=black] (0,.5) rectangle (2,0.8) node[midway] {$x'$};
}\]

As a result, we conclude, 
\[\frac{Ker\hspace{1mm}d_1}{Im \hspace{1mm}d_2} \cong \frac{\tikzdiag[yscale=1.5,scale=0.5]{
\draw[stdhl](0.5,0)--(.5,0.5);
\draw[stdhl](1.5,0)--(1.5,0.5);
\draw (1,0)--(1,0.5);
\filldraw [fill=white, draw=black] (0,.5) rectangle (2,1) node[midway] {$\Tb$};
}}{Im\hspace{1mm} \bigg(\tikzdiag[xscale=0.75,scale=0.8]{
\draw[stdhl](0.5,0)--(0.5,0.5);
\draw[stdhl](1.5,0)--(1.5,0.5);
\draw (1,0)..controls (1,0.25) and (2,0.25)..(2,0.5);
} + \tikzdiag[xscale=-0.75,scale=0.8]{
\draw[stdhl](0.5,0)--(0.5,0.5);
\draw[stdhl](1.5,0)--(1.5,0.5);
\draw (1,0)..controls (1,0.25) and (2,0.25)..(2,0.5);
}\bigg)}\]

This suggests the complex \eqref{eq:braided_cup_res} is quasi-isomorphic to, 
\[
\begin{tikzcd}[column sep=large,row sep=tiny]
& q^3\hspace{2mm}\tikzdiagh[yscale=1.5]{0}{
\draw[stdhl](0.5,0)--(.5,0.5);
\draw[stdhl](1.5,0)--(1.5,0.5);
\draw (1.75,0)--(1.75,0.5);
\filldraw [fill=white, draw=black] (0,.5) rectangle (2,1) node[midway] { $\Tb$};
}
\ar{rd}
\ar[dash]{rd}{
\tikzdiag[xscale=0.75]{
\draw[stdhl](0.5,0)--(0.5,0.5);
\draw[stdhl](1.5,0)--(1.5,0.5);
\draw (1,0)..controls (1,0.25) and (2,0.25)..(2,0.5);
}
}
& \\
q^4\hspace{2mm}\tikzdiagh[yscale=1.5]{0}{
\draw[stdhl](0.5,0)--(.5,0.5);
\draw[stdhl](1.5,0)--(1.5,0.5);
\draw (1,0)--(1,0.5);
\filldraw [fill=white, draw=black] (0,.5) rectangle (2,1) node[midway] {$\Tb$};
}
\ar{ru}
\ar[dash]{ru}{-
\tikzdiag[xscale=0.75,yscale=-1]{
\draw[stdhl](0.5,0)--(0.5,0.5);
\draw[stdhl](1.5,0)--(1.5,0.5);
\draw (1,0)..controls (1,0.25) and (2,0.25)..(2,0.5);
}
}
\ar{rd}
\ar[swap,dash]{rd}{
\tikzdiag[xscale=-0.75,yscale=-1]{
\draw[stdhl](0.5,0)--(0.5,0.5);
\draw[stdhl](1.5,0)--(1.5,0.5);
\draw (1,0)..controls (1,0.25) and (2,0.25)..(2,0.5);
}
}
&
\bigoplus &
q^2\hspace{1mm}\tikzdiagh[yscale=1.5]{0}{
\draw[stdhl](0.5,0)--(.5,0.5);
\draw[stdhl](1.5,0)--(1.5,0.5);
\draw (1,0)--(1,0.5);
\filldraw [fill=white, draw=black] (0,.5) rectangle (2,1) node[midway] { $\Tb$};
} \ar{r} & 0 \\
& 
q^3\hspace{2mm}\tikzdiagh[yscale=1.5]{0}{
\draw[stdhl](0.5,0)--(.5,0.5);
\draw[stdhl](1.5,0)--(1.5,0.5);
\draw (0.25,0)--(0.25,0.5);
\filldraw [fill=white, draw=black] (0,.5) rectangle (2,1) node[midway] { $\Tb$};
}
\ar{ru}
\ar[swap,dash]{ru}{
\tikzdiag[xscale=-0.75,yscale=1]{
\draw[stdhl](0.5,0)--(0.5,0.5);
\draw[stdhl](1.5,0)--(1.5,0.5);
\draw (1,0)..controls (1,0.25) and (2,0.25)..(2,0.5);
}
}
&
\end{tikzcd}
\]
Again, multiplied by an overall factor of $q^{-\frac{1}{2}}$. And indeed, after checking $Ker\hspace{1mm} d_2$ (which is easy and we leave it as an exercise), we find that the quasi-isomorphism holds. Notably, the complex above is precisely the cofibrant replacement for $q^{\frac{3}{2}}\bar{\Cbb}_i[1]$.

\end{proof}

For the ribbon structure to be consistent, we must be able to reverse the orientation of the cups and caps through braiding in the following way, 
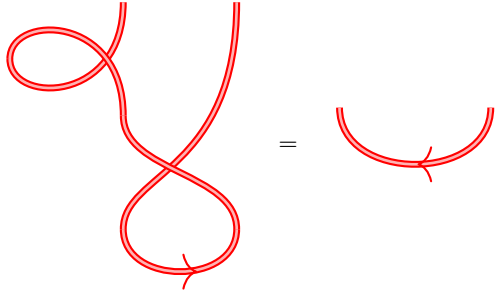
\[
\tikzdiag[scale=1.5]{
\draw[stdhl] (1,0.5)..controls (1,0.1) and (2,0.1) .. (2,1);
\draw[stdhl](3,1) ..controls (3,-.5) and (2,-.5)..(2,-1);
\draw[stdhl,mid<-] (3,-1) ..controls (3,-1.5) and (2,-1.5) ..(2,-1);
\draw[stdhl](2,0)..controls (2,-.5) and (3,-.5) ..(3,-1);
\draw[stdhl](2,0) ..controls (2,0.9) and (1,0.9)..(1,0.5) ;
}
\hspace{4mm}= \hspace{4mm}\tikzdiag[scale=2
]{
\draw[stdhl,mid->] (3,-1) ..controls (3,-1.5) and (2,-1.5) ..(2,-1);
}
\]

This translates to, 
\begin{prop}
    $\Kbb_i \circ\Bb_{i+2}\circ \Cbb_i\circ \Bb_{i+1}\circ \Cbb_i \cong \bar{\Cbb_i}$
\end{prop}
\begin{proof}
By Proposition \ref{prop:cup_r1}, it is sufficient to show, $\Kbb_i \circ\Bb_{i+2}\circ \Cbb_i\circ \bar{\Cbb}_i \cong q^{-\frac{3}{2}}\bar{\Cbb}_i[-1]$. For this case, this is very simple as $\Kbb_i \circ\Bb_{i+2}\circ \Cbb_i$ applied to each projective in the resolution of $\bar{\Cb}_i$ yields itself shifted by $q^{-\frac{3}{2}} h^{-1}$ (one can check this by repeating the same techniques of braiding and resolving using Prop \ref{prop:Rresolution} we have been using in this section) and the differentials just carry over to the shifted complex. 
\end{proof}

We also check the S-move:
\[
\tikzdiag[scale=2]{
\draw[stdhl,mid->] (0,0)..controls (0,2) and (1,2) ..(1,1) ..controls (1,0) and (2,0) ..(2,2);
} \hspace{4mm}=\hspace{4mm} \tikzdiag[scale=2]{
\draw[stdhl,mid->] (0,0) -- (0,2);
}
\]

\begin{lem}
    $\bar{\Kbb}_{i-1} \circ \Cbb_i \cong \mathbb{1}$
\end{lem}
\begin{proof}
Take the resolution of $\Cbb_i$ on any projective \eqref{eq:cup_proj_res}. Notice only the bottom projective in degree 1 is not killed by the cap being applied to its left. Its grading is precisely $q h$, which is exactly canceled by the degree of the cap, leaving us with the original projective. 
\end{proof}

The last move we must check is the ``pitchfork'' move from \cite{Web}, 
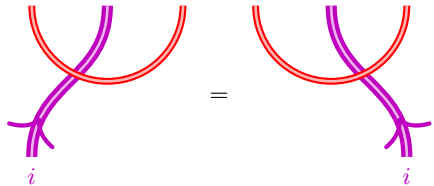
\[
\tikzdiag[scale=2]{
\draw[pstdhl,lower->] (0,0) node[below]{\small $i$} ..controls (0,0.5) and (0.5,0.5) .. (0.5,1);
\draw[stdhl] (1,1) arc (0:-180:0.5);
} \hspace{2mm} = \hspace{2mm}\tikzdiag[scale=2,xscale=-1]{
\draw[pstdhl,lower->] (0,0) node[below]{\small $i$} ..controls (0,0.5) and (0.5,0.5) .. (0.5,1);
\draw[stdhl] (1,1) arc (0:-180:0.5);
}
\]
\begin{lem}
With the labeling of strands and the braiding permutations being understood as in the diagram above, the pitchfork move is obeyed, 
$$\Bb_\sigma \circ \Cbb_i \cong \Bb^{-1}_{\sigma'} \circ \Cbb_{i-1}$$
\end{lem}
\begin{proof}
    We focus on the case that the purple strand is colored by a Verma module, since this is the only new case not considered in \cite{Web}. By Proposition \ref{prop:Inverse_Braiding}, the statement is equivalent to $\Bb_{\sigma'}\circ \Bb_\sigma \circ \Cbb_i \cong \Cbb_{i-1}$. Without loss of generality, assume $b_{i-1}=0$ and consider the resolution of $\Cb_i$, 
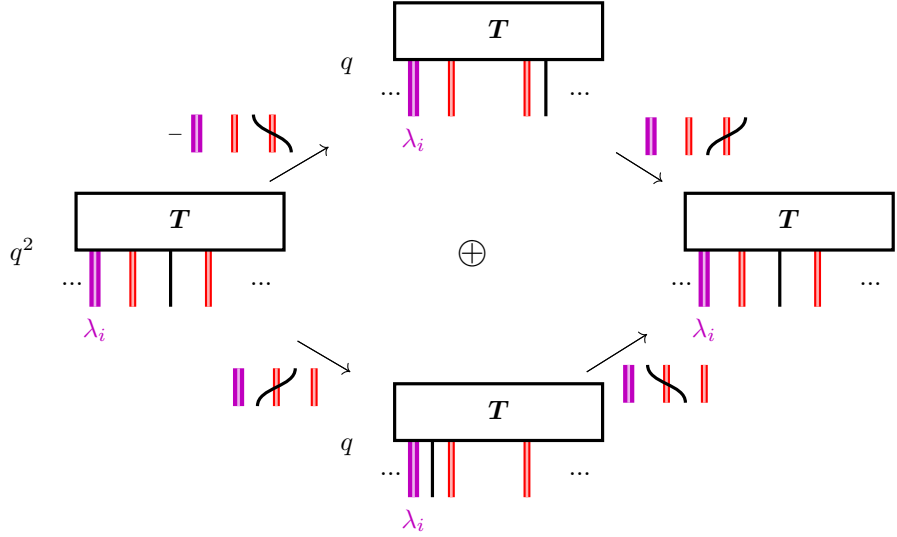
    \[
    \begin{tikzcd}[row sep=tiny,column sep=small]
&q \hspace{2mm}\tikzdiagh[yscale=1.5]{0}{
\draw[pstdhl](0,0) node[below]{$\lambda_i$} --(0,0.5);
\draw[stdhl](0.5,0)--(.5,0.5);
\draw[stdhl](1.5,0)--(1.5,0.5);
\draw (1.75,0)--(1.75,0.5);
\filldraw [fill=white, draw=black] (-0.25,.5) rectangle (2.5,1) node[midway] { $\Tb$};
\node at (-0.3,.25) {$...$};
\node at (2.2,.25) {$...$};

}
\ar{rd}
\ar[dash]{rd}{
\tikzdiag[scale=1]{
\draw[pstdhl](0,0)--(0,0.5);
\draw[stdhl](0.5,0)--(0.5,0.5);
\draw[stdhl](1,0)--(1,0.5);
\draw (1.25,0.5)..controls (1.25,0.25) and (0.75,0.25)..(0.75,0);
}
}
& &\\
     q^2 \hspace{2mm}\tikzdiagh[yscale=1.5]{0}{
\draw[pstdhl](0,0) node[below]{$\lambda_i$} --(0,0.5);
\draw[stdhl](0.5,0)--(.5,0.5);
\draw[stdhl](1.5,0)--(1.5,0.5);
\draw (1,0)--(1,0.5);
\filldraw [fill=white, draw=black] (-0.25,.5) rectangle (2.5,1) node[midway] { $\Tb$};
\node at (-0.3,.25) {$...$};
\node at (2.2,.25) {$...$};
}
\ar{ru}
\ar[dash]{ru}{-\hspace{1mm}
\tikzdiag[scale=1]{
\draw[pstdhl](0,0)--(0,0.5);
\draw[stdhl](0.5,0)--(0.5,0.5);
\draw[stdhl](1,0)--(1,0.5);
\draw (1.25,0)..controls (1.25,0.25) and (0.75,0.25)..(0.75,0.5);
}
}
\ar{rd}
\ar[swap,dash]{rd}{
\tikzdiag[scale=1]{
\draw[pstdhl](0,0)--(0,0.5);
\draw[stdhl](0.5,0)--(0.5,0.5);
\draw[stdhl](1,0)--(1,0.5);
\draw (.25,0)..controls (.25,0.25) and (0.75,0.25)..(0.75,0.5);
}
}
& \bigoplus &
\tikzdiagh[yscale=1.5]{0}{
\draw[pstdhl](0,0) node[below]{$\lambda_i$} --(0,0.5);
\draw[stdhl](0.5,0)--(.5,0.5);
\draw[stdhl](1.5,0)--(1.5,0.5);
\draw (1,0)--(1,0.5);
\filldraw [fill=white, draw=black] (-0.25,.5) rectangle (2.5,1) node[midway] { $\Tb$};
\node at (-0.3,.25) {$...$};
\node at (2.2,.25) {$...$};
}
\\
&  q \hspace{2mm}\tikzdiagh[yscale=1.5]{0}{
\draw[pstdhl](0,0) node[below]{$\lambda_i$} --(0,0.5);
\draw[stdhl](0.5,0)--(.5,0.5);
\draw[stdhl](1.5,0)--(1.5,0.5);
\draw (.25,0)--(.25,0.5);
\filldraw [fill=white, draw=black] (-0.25,.5) rectangle (2.5,1) node[midway] { $\Tb$};
\node at (-0.3,.25) {$...$};
\node at (2.2,.25) {$...$};
} 
\ar{ru}
\ar[swap,dash]{ru}{
\tikzdiag[scale=1]{
\draw[pstdhl](0,0)--(0,0.5);
\draw[stdhl](0.5,0)--(0.5,0.5);
\draw[stdhl](1,0)--(1,0.5);
\draw (.75,0)..controls (.75,0.25) and (0.25,0.25)..(0.25,0.5);
}
}
&&
\end{tikzcd}
    \]
Going counter clockwise starting at the projective on the far right, label each projective with $\bigcirc_1, \bigcirc_2, \bigcirc_3, \bigcirc_4$. We can view the resolution as a nested cone $Cone\bigg(Cone(\bigcirc_3\rightarrow \bigcirc_4)\longrightarrow Cone(\bigcirc_2\rightarrow \bigcirc_1)\bigg)$ with the maps as shown above. We apply the first braiding functor to $Cone(\bigcirc_3\rightarrow \bigcirc_4)$ and resolve. The induced map on the projective resolutions is also displayed, 
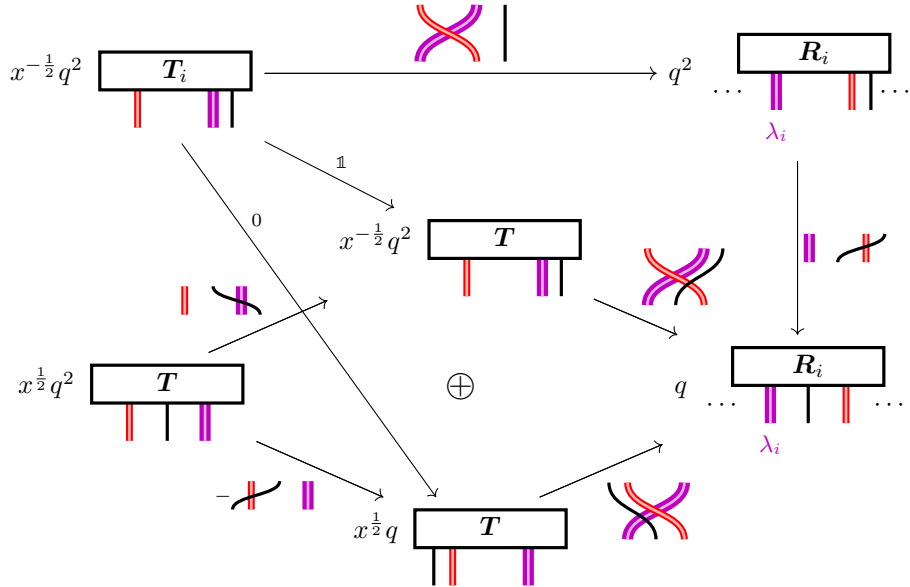
\[
\begin{tikzcd}[row sep=small]
x^{-\frac{1}{2}}q^2\hspace{2mm}\tikzdiagh{0}{
    \draw[stdhl](0,0) --(0,0.5);
    \draw(1.25,0)--(1.25,0.5);
    \draw[pstdhl] (1,0)--(1,0.5) ;
    \filldraw [fill=white, draw=black] (-0.5,.5) rectangle (1.5,1) node[midway] { $\Tb_i$};
    }
    \ar{rr}
\ar[dash]{rr}{
\tikzdiagh[scale=0.75]{0}{
\draw[pstdhl](0,0)..controls (0,.5) and (1,.5) .. (1,1);
\draw[stdhl](1,0)..controls (1,.5) and (0,.5) .. (0,1);
\draw (1.5,0)--(1.5,1);
}
}
\ar{rdd}{\mathbb{1}}
\ar[pos=0.25]{rdddd}{0}
    &&
    q^2\hspace{1mm}\tikzdiagh{0}{
    \draw[pstdhl](0,0) node[below]{\small $\lambda_i$}--(0,0.5) node[color=black, pos=.5,left=2mm]{$\cdots$};
    \draw(1.25,0)--(1.25,0.5);
    \draw[stdhl] (1,0)--(1,0.5) node[color=black, pos=.5,right=2mm]{$\cdots$};
    \filldraw [fill=white, draw=black] (-0.5,.5) rectangle (1.5,1) node[midway] { $\Rb_i$};
    }
    \ar{ddd}
    \ar[dash]{ddd}{
\tikzdiag[scale=0.75]{
\draw[pstdhl](0,0)--(0,0.5);
\draw[stdhl](1,0)--(1,0.5);
\draw (.5,0)..controls (.5, 0.25) and (1.33,0.25) ..(1.33,0.5);
} }
    \\ \\
& x^{-\frac{1}{2}}q^2 \hspace{2mm} \tikzdiagh{0}{
    \draw[stdhl](0,0)--(0,0.5);
    \draw(1.25,0)--(1.25,0.5);
    \draw[pstdhl] (1,0)--(1,0.5);
    \filldraw [fill=white, draw=black] (-0.5,.5) rectangle (1.5,1) node[midway] { $\Tb$};
    } \ar{rd}
\ar[dash]{rd}{
\tikzdiagh[scale=0.75]{0}{
\draw[pstdhl](0,0)..controls (0,.5) and (1,.5) .. (1,1);
\draw[stdhl](1,0)..controls (1,.5) and (0,.5) .. (0,1);
\draw (.5,0)..controls (.5, 0.5) and (1.33,0.5) ..(1.33,1);
}
}& \\
    x^\frac{1}{2} q^2\hspace{2mm} \tikzdiagh{0}{
    \draw[stdhl](0,0)--(0,0.5);
    \draw(.5,0)--(.5,0.5);
    \draw[pstdhl] (1,0)--(1,0.5);
    \filldraw [fill=white, draw=black] (-0.5,.5) rectangle (1.5,1) node[midway] { $\Tb$};
    }
    \ar{rd}
    \ar[dash,swap]{rd}{-
\tikzdiag[scale=0.75,yscale=-1,xscale=-1]{
\draw[pstdhl](0,0)--(0,0.5);
\draw[stdhl](1,0)--(1,0.5);
\draw (.5,0)..controls (.5, 0.25) and (1.33,0.25) ..(1.33,0.5);
} }
\ar{ru}
    \ar[dash]{ru}{
\tikzdiag[scale=0.75,yscale=-1]{
\draw[stdhl](0,0)--(0,0.5);
\draw[pstdhl](1,0)--(1,0.5);
\draw (.5,0)..controls (.5, 0.25) and (1.33,0.25) ..(1.33,0.5);
} }
    &
    \bigoplus
    &
    q\hspace{1mm}\tikzdiagh{0}{
    \draw[pstdhl](0,0) node[below]{\small $\lambda_i$}--(0,0.5) node[color=black, pos=.5,left=2mm]{$\cdots$};
    \draw(.5,0)--(.5,0.5);
    \draw[stdhl] (1,0)--(1,0.5) node[color=black, pos=.5,right=2mm]{$\cdots$};
    \filldraw [fill=white, draw=black] (-0.5,.5) rectangle (1.5,1) node[midway] { $\Rb_i$};
    } \\
    & x^\frac{1}{2} q \hspace{2mm} \tikzdiagh{0}{
    \draw[stdhl](0,0)--(0,0.5);
    \draw(-.25,0)--(-.25,0.5);
    \draw[pstdhl] (1,0)--(1,0.5);
    \filldraw [fill=white, draw=black] (-0.5,.5) rectangle (1.5,1) node[midway] { $\Tb$};
    } \ar{ru}
    \ar[dash,swap]{ru}{
\tikzdiagh[scale=0.75]{0}{
\draw[pstdhl](0,0)..controls (0,.5) and (1,.5) .. (1,1);
\draw[stdhl](1,0)..controls (1,.5) and (0,.5) .. (0,1);
\draw (.5,0)..controls (.5, 0.5) and (-.33,0.5) ..(-.33,1);
}
}
\end{tikzcd}\]
Using this, it is not difficult to show that the resulting complex, i.e. $\Rb_{i+1} \otimes^L \Rb_i \otimes^L Cone(\bigcirc_3\rightarrow \bigcirc_4)$, is quasi-isomorphic to, 
\[ \begin{tikzcd}
 q^2\hspace{2mm} \tikzdiagh{0}{
    \draw[stdhl](0,0)--(0,0.5);
    \draw(.33,0)--(.33,0.5);
    \draw[stdhl] (0.66,0)--(0.66,.5);
    \draw[pstdhl] (1,0)--(1,0.5);
    \filldraw [fill=white, draw=black] (-0.5,.5) rectangle (1.5,1) node[midway] { $\Tb$};
    }
    \ar{rr}
    \ar[dash,swap]{rr}{
\tikzdiag[scale=1]{
\draw[pstdhl](1,0)--(1,0.5);
\draw[stdhl](0,0)--(0,0.5);
\draw[stdhl](.5,0)--(.5,0.5);
\draw (-.25,0)..controls (-.25, 0.25) and (.25,0.25) ..(.25,0.5);
} } && q \hspace{2mm} \tikzdiagh{0}{
    \draw[stdhl](0,0)--(0,0.5);
    \draw(-.25,0)--(-.25,0.5);
    \draw[stdhl] (0.66,0)--(0.66,.5);
    \draw[pstdhl] (1,0)--(1,0.5);
    \filldraw [fill=white, draw=black] (-0.5,.5) rectangle (1.5,1) node[midway] { $\Tb$};
    }
\end{tikzcd}\]
Notice that the second application of $\Bb_i$ only shifts the complex by $x^{-\frac{1}{2}}$ since there are no projectives with black strands to braid after the quasi-isomorphism. Repeating a very similar argument, one finds that for $\Rb_{i+1} \otimes^L \Rb_i \otimes^L Cone(\bigcirc_2\rightarrow \bigcirc_1)$, we have a quasi-isomorphism to, 

\[ \begin{tikzcd}
 q\hspace{2mm} \tikzdiagh{0}{
    \draw[stdhl](0,0)--(0,0.5);
    \draw[stdhl] (.33,0)--(.33,0.5);
    \draw (0.66,0)--(0.66,.5);
    \draw[pstdhl] (1,0)--(1,0.5);
    \filldraw [fill=white, draw=black] (-0.5,.5) rectangle (1.5,1) node[midway] { $\Tb$};
    }
    \ar{rr}
    \ar[dash,swap]{rr}{
\tikzdiag[scale=1]{
\draw[pstdhl](1,0)--(1,0.5);
\draw[stdhl](0,0)--(0,0.5);
\draw[stdhl](.5,0)--(.5,0.5);
\draw (.25,0)..controls (.25, 0.25) and (.75,0.25) ..(.75,0.5);
} } &&  \tikzdiagh{0}{
    \draw[stdhl](0,0)--(0,0.5);
    \draw(.33,0)--(.33,0.5);
    \draw[stdhl] (0.66,0)--(0.66,.5);
    \draw[pstdhl] (1,0)--(1,0.5);
    \filldraw [fill=white, draw=black] (-0.5,.5) rectangle (1.5,1) node[midway] { $\Tb$};
    }
\end{tikzcd}\]
If one also tracks the maps between these two chain complexes through the braiding, the resulting cone is precisely the resolution for $\Cb_{i-1}$ mentioned several times throughout this section
\end{proof}

\subsection{Khovanov Homology For Braid Complements} \label{sec:khov_hom_braids}
In \cite{Web}, an invariant, $\Phi(-)$ , associated to a colored $(p,q)-$tangle, $T$, was defined. There, $\Phi(T)$ was a functor $\mathcal{V}^{\underline{\lambda}} \rightarrow \mathcal{V}^{\underline{\mu}}$, where $\underline{\lambda}$ and $\underline{\mu}$ are the appropiate list of weights assigned to the bottom and top of the tangle, respectively, and $\mathcal{V}^{\underline{\lambda}}$ was the relevant category associated to the tensor product specified by $\underline{\lambda}$. Of course, when $p=q=0$, $\Phi(T)$ is a homological invariant we can assign to knots. We wish to extend this construction to the case of strands colored by Verma module representations of generic weight. This would, in principle, define a theory of Khovanov homology on knot complements in $S^3$. However, to do so in full generality, we need to work out what the cups and caps for Verma colored strands are. This will be the focus of the subsequent paper. For this section, we will consider tangles where the Verma colored strands form a braid. In other words, we concern ourselves with diagrams where there are no blue cups or caps. 

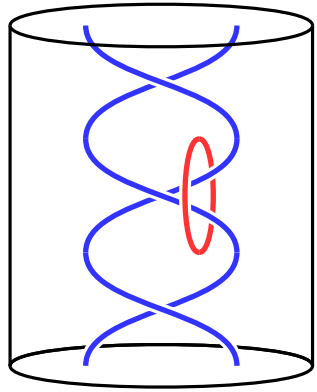
\[
\tikzdiagh[yscale=-0.75]{0}{
\draw[rstdhl] (1.5, 2) ..controls (1.75,2) and (1.75,4) ..(1.5,4);
\draw (-1,6) ..controls (-1,5.5) and (3,5.5) .. (3,6);
\draw (-1,6) ..controls (-1,6.5) and (3,6.5) .. (3,6);
\draw (-1,0)--(-1,6);
\draw (3,0)--(3,6);
\draw (-1,6) ..controls (-1,5.5) and (3,5.5) .. (3,6);
\draw[bstdhl](2,0)..controls (2,1) and (0,1) ..(0,2);
\draw[bstdhl] (0,0)..controls (0,1) and (2,1)..(2,2);
\draw[bstdhl] (2,2)..controls (2,3) and (0,3) ..(0,4);
\draw[bstdhl](0,2)..controls (0,3) and (2,3) ..(2,4);
\draw[bstdhl] (2,4)..controls (2,5) and (0,5) ..(0,6);
\draw[bstdhl](0,4)..controls (0,5) and (2,5) ..(2,6);
\draw[rstdhl] (1.5, 2) ..controls (1.25,2) and (1.25,4) ..(1.5,4);
\draw (-1,0) ..controls (-1,-.5) and (3,-.5) .. (3,0);
\draw (-1,0) ..controls (-1,.5) and (3,.5) .. (3,0);
}
\]

\begin{defn} \label{def:tangle_functor}
    Let $D$ be the diagram of an oriented tangle whose components are colored by either highest weight Verma module representations of $\mathfrak{sl}_2$ or the fundamental representation of $\mathfrak{sl}_2$. Assume further that $D$ has no cups or caps in the strands colored by Verma modules. Let $\underline{\mu} = (\mu_1,\mu_2,...,\mu_N)$ be the strands at the bottom of the tangle and $\underline{\nu} = (\nu_1,\nu_2,...,\nu_M)$ the strands at the top. Then, we let $\Phi(D): \mathcal{D}_{dg}(\Tb^{\underline{\mu}}) \rightarrow \mathcal{D}_{dg}(\Tb^{\underline{\nu}})$\textbf{ be the functor associated to $D$} by the following rules. Reading $D$ from bottom to top, we compose the sequence of functors,
    \begin{itemize}
        \item To the far left of the diagram, add an upward moving, vertical, Verma-colored strand to the diagram, not braiding with any other strand
        \item To a negative crossing between the $i$th and $i+1$th strands, we associate $\Bb_i$
        \item To a positive crossing between the $i$th and $i+1$th strand, we associate $\Bb^{-1}_i$
    \end{itemize}
    Now, referring strictly to the caps and caps colored by fundamental representation, we have,
    \begin{itemize}
        \item To a clockwise oriented cup, we associate $\bar{\Cbb}$
        \item To a counter-clockwise oriented cup, we associate $\Cbb$
        \item To a clockwise oriented cup, we associate $\bar{\Kbb}$
        \item To a counter-clockwise oriented cup, we associate $\Kbb$ 
    \end{itemize}
\end{defn}
\begin{rem}
    Note that we have often worked with the case where all Verma strands are colored by the same highest weight $\lambda$, but this is not a requirement. We can treat each blue strand to have its own unique highest weight $\lambda_i$. In which case, the presence of each strand introduces a new grading $x_i$ to the dg-KLRW algebra and the resulting functors corresponding to the holonomy around the punctured strand. 
\end{rem}
The results of the previous sections immediately imply, 

\begin{thm}
    Let $D$ and $D'$ be 2 diagrams for the same isotopy class of ribbon tangle of the kind admissible by Def. \ref{def:tangle_functor}. Then, $\Phi(D) \cong \Phi(D')$. 
\end{thm}

Of course, we can also make the functor an honest invariant of oriented tangles by correcting for the ribbon structure,
$$\Phi'(D) := \prod^F_{i=1}(q^{\frac{3}{2}}h)^{w_i} \Phi(D)$$
Here $w_i = n^+_i-n_i^-$ is the writhe of the $i$th component colored by fundamental and $F$ is total number of components colored by the fundamental representation. 

\begin{rem}
In \cite{Park}, to get well-defined quantum knot invariants from Verma modules, one must restrict to the case where the knots in question are braid-positive (or negative). One might therefore worry that we have no such restriction on the braid word $\beta$ above. This is simply a consequence of the fact that we have not yet `traced' over the category. That is, for $D$ a diagram of a braid colored by Vermas $\beta$ and a link in its complement colored by fundamentals, $L$, $\Phi(D)$ categorifies linear maps $\hat{Z}(\beta, L): V^{\underline{\mu}} \rightarrow V^{\beta(\underline{\mu})}$ (see Section \ref{sec:intro}). For any projective $\Pb^{\underline{\mu}}_\rho$, applying the functor $\Phi(D)$ consists of a finite composition of functors with finite projective resolutions, so it is always well-defined. Of course, we expect that once the cups and caps are introduced to our category in the upcoming paper, we will also succumb to putting restrictions on $\beta$ in the context of knot homologies. 
\end{rem}

\section{An Example}\label{sec:example}
In this section, we will use the constructions in the previous sections to explicitly compute knot homologies in the case that the braid traced out by the Verma module colored strands is a single strand. In Section \ref{sec:KhS3}, we briefly review $S^3$ and the connection between \cite{Web} and Khovanov Homology \cite{Kh} (which are related by an overall normalization factor). In Section \ref{sec:AKh}, we consider the homology theory constructed in this paper in the case that the Verma colored strands (blue strands) trace out an unknot. We compare some calculations in this theory to Annular Khovanov Homology as constructed in \cite{APS} and propose that the theories are very closely related and conjecture their dimensions are, in fact, the same. \\ \\
We will use the following notation for this section. Let $D$ be a diagram as in Definition \ref{def:tangle_functor}. We can separate 
\be D=V\cup L \label{eq:diag_decomp}\ee
where $V$ is the diagram corresponding to all the blue strands and $L$ is the diagram corresponding to all the red strands. We will always assume $L$ is a link here. In the case that $V$ is empty, we write, 

$$\mathcal{H}^{S^3}(L,\kk):=\Phi(D) $$

\begin{rem}For all of the sections below, we consider only embedded knots colored by the fundamental representation of $\mathfrak{sl_2}$. 
\end{rem}

\subsection{\texorpdfstring{Knot Homology in $S^3$}{Knot Homology in S3}}\label{sec:KhS3}
The precise connection between Khovanov homology \cite{Kh} and the KLRW approach due to Webster \cite{Web} was carefully worked out in \cite{Web2}. There, it was shown that the homology defined through KLRW algebras coincides with colored Khovanov homology as defined in \cite{CK}. As such, we feel no need to dwell on this section, only providing a few examples for illustrative purposes. \\
\begin{ex}{\textbf{Unknot.}}
For the unknot in $S^3$, 
$$\tikzdiag[yscale=1]{
\draw[rstdhl] (0,0) circle (0.5);
	} $$
We start with the projective resolution of a cup, $\cupp $, 
\begin{equation}
\begin{tikzcd} q \hspace{2mm}
\tikzdiagh[xscale=0.75]{0}{
\draw[stdhl](0,0)--(0,0.5);
\draw[stdhl](1,0)--(1,0.5);
\draw (0.5,0)--(0.5,0.5);
\filldraw [fill=white, draw=black] (-0.25,.5) rectangle (1.5,1) node[midway] { $\Tb$};
}
\ar{rr}
\ar[dash]{rr}{
\tikzdiag[xscale=0.75]{
	\draw[stdhl](0,0)--(0,0.5);
\draw[stdhl](1,0)--(1,0.5);
\draw (1.4,0)..controls (1.4,0.25) and (0.5,0.25)..(0.5,0.5);
}
}
& &
\tikzdiagh[xscale=0.75]{0}{
\draw[stdhl](0,0)--(0,0.5);
\draw[stdhl](1,0)--(1,0.5);
\draw (1.4,0)--(1.4,0.5);
\filldraw [fill=white, draw=black] (-0.25,.5) rectangle (1.5,1) node[midway] { $\Tb$};
}
\ar{rr}
\ar[dash]{rr}{
\tikzdiag[xscale=0.75]{
	\draw[stdhl](0,0)--(0,0.5);
\draw[stdhl](1,0)--(1,0.5);
\draw (0.5,0)..controls (0.5,0.25) and (1.4,0.25)..(1.4,0.5);
}
}
& & q^{-1} \hspace{2mm}
\tikzdiagh[xscale=0.75]{0}{
\draw[stdhl](0,0)--(0,0.5);
\draw[stdhl](1,0)--(1,0.5);
\draw (0.5,0)--(0.5,0.5);

\filldraw [fill=white, draw=black] (-0.25,.5) rectangle (1.5,1) node[midway] { $\Tb$};
}
\ar{rr}
\ar[dash]{rr}{
\tikzdiag[xscale=0.75]{
\draw (0.5,0)-- (0.5,0.5);
\draw[stdhl] (1,0.5) arc (0:-180:0.5);
}
}
& & \cupp
\end{tikzcd}
\label{eq:cap_res}
\end{equation}
Tensoring the with cap, we find, 
$$q \hspace{1mm} \kk \longrightarrow 0 \longrightarrow q^{-1} \hspace{1mm} \kk $$
So, the homology of the unknot is simply $$\mathcal{H}^{S^3}(U, \kk) \cong q^{-1} \hspace{1mm} \kk [-1] \oplus q \hspace{1mm} \kk [1] $$
Notice that in the usual formulation of Khovanov homology \cite{Kh} and in its extension to $I$-bundles over surfaces \cite{APS}, we have, 
$Kh(U,\kk) \cong q^{-1} \hspace{1mm} \kk [0] \oplus q \hspace{1mm} \kk [0]$. In other words, the unknot is normalized differently in the two theories. 
    
\end{ex}

\begin{ex}{\textbf{Hopf Link.}} 
Let us compute the homology for the Hopf link in the KLRW approach,
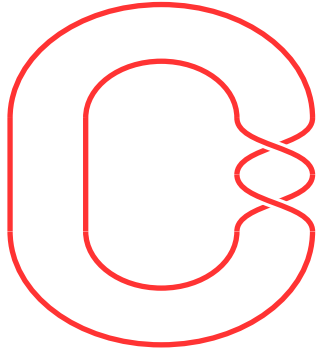
\[
\tikzdiag[yscale=0.75]{
\draw[rstdhl] (1,0) arc (0:-180:1);
\draw[rstdhl] (2,0) arc (0:-180:2);
\draw[rstdhl] (1,0) ..controls (1,0.5) and (2,0.5) .. (2,1);
\draw[rstdhl] (2,0) ..controls (2,0.5) and (1,0.5) .. (1,1);
\draw[rstdhl] (1,1) ..controls (1,1.5) and (2,1.5) .. (2,2);
\draw[rstdhl] (2,1) ..controls (2,1.5) and (1,1.5) .. (1,2);
\draw[rstdhl] (-1,0) -- (-1,2);
\draw[rstdhl] (-2,0) -- (-2,2);
\draw[rstdhl] (1,2) arc (0:180:1);
\draw[rstdhl] (2,2) arc (0:180:2);
}
\]
Applying the cap to the resolution \eqref{eq:cap_res} and resolving once again, we find, 
\begin{equation}
\label{eq:doublecap_res}
\begin{tikzcd}[row sep=tiny,column sep=small]
&
q^2 \hspace{2mm} \tikzdiagh[xscale=0.5]{0}{
\draw[stdhl](0,0)--(0,0.5);
\draw[stdhl](1,0)--(1,0.5);
\draw[stdhl](2,0)--(2,0.5);
\draw[stdhl](3,0)--(3,0.5);
\draw (0.5,0)--(0.5,0.5);
\draw (2.5,0)--(2.5,0.5);
\filldraw [fill=white, draw=black] (-0.5,.5) rectangle (3.5,1) node[midway] { $\Tb$};
}
\ar{rd}
\ar[dash]{rd}{
\tikzdiagh[xscale=0.4,yscale=0.75]{0}{
\draw[stdhl](0,0)--(0,0.5);
\draw[stdhl](1,0)--(1,0.5);
\draw[stdhl](2,0)--(2,0.5);
\draw[stdhl](3,0)--(3,0.5);
\draw (0.5,0)--(0.5,0.5);
\draw (1.5,0)..controls (1.5,.25) and (2.5,.25) .. (2.5,0.5);
}
}
&\\ 
q^3  \hspace{2mm}\tikzdiagh[xscale=0.5]{0}{
\draw[stdhl](0,0)--(0,0.5);
\draw[stdhl](1,0)--(1,0.5);
\draw[stdhl](2,0)--(2,0.5);
\draw[stdhl](3,0)--(3,0.5);
\draw (0.5,0)--(0.5,0.5);
\draw (1.5,0)--(1.5,0.5);
\filldraw [fill=white, draw=black] (-0.5,.5) rectangle (3.5,1) node[midway] { $\Tb$};
}
\ar{rddd}
\ar{ru}
\ar[dash]{ru}{ - \ \tikzdiag[xscale=0.4,yscale=0.75]{
\draw[stdhl](0,0)--(0,0.5);
\draw[stdhl](1,0)--(1,0.5);
\draw[stdhl](2,0)--(2,0.5);
\draw[stdhl](3,0)--(3,0.5);
\draw (0.5,0)--(0.5,0.5);
\draw (2.5,0)..controls (2.5,.25) and (1.5,.25) .. (1.5,0.5);
}
}
\ar{rd}
\ar[swap,dash]{rd}{
\tikzdiagh[xscale=0.4,yscale=0.75]{0}{
\draw[stdhl](0,0)--(0,0.5);
\draw[stdhl](1,0)--(1,0.5);
\draw[stdhl](2,0)--(2,0.5);
\draw[stdhl](3,0)--(3,0.5);
\draw (0.33,0)--(0.33,0.5);
\draw (.66,0)..controls (.66,.25) and (1.5,.25) .. (1.5,0.5);
}
}
& \bigoplus \ar{rddd} &
q \hspace{2mm }\tikzdiagh[xscale=0.5]{0}{
\draw[stdhl](0,0)--(0,0.5);
\draw[stdhl](1,0)--(1,0.5);
\draw[stdhl](2,0)--(2,0.5);
\draw[stdhl](3,0)--(3,0.5);
\draw (0.5,0)--(0.5,0.5);
\draw (1.5,0)--(1.5,0.5);
\filldraw [fill=white, draw=black] (-0.5,.5) rectangle (3.5,1) node[midway] { $\Tb$};
} \ar{rddd}\\
 & 
q^2 \hspace{2mm} \tikzdiagh[xscale=0.5]{0}{
\draw[stdhl](0,0)--(0,0.5);
\draw[stdhl](1,0)--(1,0.5);
\draw[stdhl](2,0)--(2,0.5);
\draw[stdhl](3,0)--(3,0.5);
\draw (0.33,0)--(0.33,0.5);
\draw (.66,0)--(.66,0.5);
\filldraw [fill=white, draw=black] (-0.5,.5) rectangle (3.5,1) node[midway] { $\Tb$};
\node[orange] at(0.5,.75){$\bullet$}
}
\ar{ru}
\ar[swap,dash]{ru}{
\tikzdiagh[xscale=0.4,yscale=0.75]{0}{
\draw[stdhl](0,0)--(0,0.5);
\draw[stdhl](1,0)--(1,0.5);
\draw[stdhl](2,0)--(2,0.5);
\draw[stdhl](3,0)--(3,0.5);
\draw (0.33,0)--(0.33,0.5);
\draw (1.5,0)..controls (1.5,.25) and (.66,.25) .. (.66,0.5);
}
}
&\bigoplus \\
& &
q \hspace{2mm} \tikzdiagh[xscale=0.5]{0}{
\draw[stdhl](0,0)--(0,0.5);
\draw[stdhl](1,0)--(1,0.5);
\draw[stdhl](2,0)--(2,0.5);
\draw[stdhl](3,0)--(3,0.5);
\draw (3.5,0)--(3.5,0.5);
\draw (2.5,0)--(2.5,0.5);
\filldraw [fill=white, draw=black] (-0.5,.5) rectangle (3.5,1) node[midway] { $\Tb$};
\node[orange] at(0.5,.75){$\bullet$}
}
\ar{rd}
\ar[dash]{rd}{
\tikzdiagh[xscale=0.4,yscale=0.75]{0}{
\draw[stdhl](0,0)--(0,0.5);
\draw[stdhl](1,0)--(1,0.5);
\draw[stdhl](2,0)--(2,0.5);
\draw[stdhl](3,0)--(3,0.5);
\draw (3.5,0)--(3.5,0.5);
\draw (1.5,0)..controls (1.5,.25) and (2.5,.25) .. (2.5,0.5);
}
}
&\\ & 
q^2  \hspace{2mm}\tikzdiagh[xscale=0.5]{0}{
\draw[stdhl](0,0)--(0,0.5);
\draw[stdhl](1,0)--(1,0.5);
\draw[stdhl](2,0)--(2,0.5);
\draw[stdhl](3,0)--(3,0.5);
\draw (3.5,0)--(3.5,0.5);
\draw (1.5,0)--(1.5,0.5);
\filldraw [fill=white, draw=black] (-0.5,.5) rectangle (3.5,1) node[midway] { $\Tb$};
\node[orange] at(0.5,.75){$\bullet$}
}
\ar{rddd}
\ar{ru}
\ar[dash]{ru}{ - \ \tikzdiag[xscale=0.4,yscale=0.75]{
\draw[stdhl](0,0)--(0,0.5);
\draw[stdhl](1,0)--(1,0.5);
\draw[stdhl](2,0)--(2,0.5);
\draw[stdhl](3,0)--(3,0.5);
\draw (3.5,0)--(3.5,0.5);
\draw (2.5,0)..controls (2.5,.25) and (1.5,.25) .. (1.5,0.5);
}
}
\ar{rd}
\ar[swap,dash]{rd}{
\tikzdiagh[xscale=0.4,yscale=0.75]{0}{
\draw[stdhl](0,0)--(0,0.5);
\draw[stdhl](1,0)--(1,0.5);
\draw[stdhl](2,0)--(2,0.5);
\draw[stdhl](3,0)--(3,0.5);
\draw (3.5,0)--(3.5,0.5);
\draw (.66,0)..controls (.66,.25) and (1.5,.25) .. (1.5,0.5);
}
}
& \bigoplus \ar{rddd} &
\tikzdiagh[xscale=0.5]{0}{
\draw[stdhl](0,0)--(0,0.5);
\draw[stdhl](1,0)--(1,0.5);
\draw[stdhl](2,0)--(2,0.5);
\draw[stdhl](3,0)--(3,0.5);
\draw (3.5,0)--(3.5,0.5);
\draw (1.5,0)--(1.5,0.5);
\filldraw [fill=white, draw=black] (-0.5,.5) rectangle (3.5,1) node[midway] { $\Tb$};
\node[orange] at(0.5,.75){$\bullet$};
} \ar{rddd}\\ &
 & 
q \hspace{2mm} \tikzdiagh[xscale=0.5]{0}{
\draw[stdhl](0,0)--(0,0.5);
\draw[stdhl](1,0)--(1,0.5);
\draw[stdhl](2,0)--(2,0.5);
\draw[stdhl](3,0)--(3,0.5);
\draw (3.5,0)--(3.5,0.5);
\draw (.66,0)--(.66,0.5);
\filldraw [fill=white, draw=black] (-0.5,.5) rectangle (3.5,1) node[midway] { $\Tb$};
\node[orange] at(0.5,.75){$\bullet$};
}
\ar{ru}
\ar[swap,dash]{ru}{
\tikzdiagh[xscale=0.4,yscale=0.75]{0}{
\draw[stdhl](0,0)--(0,0.5);
\draw[stdhl](1,0)--(1,0.5);
\draw[stdhl](2,0)--(2,0.5);
\draw[stdhl](3,0)--(3,0.5);
\draw (3.5,0)--(3.5,0.5);
\draw (1.5,0)..controls (1.5,.25) and (.66,.25) .. (.66,0.5);
}
}
&\bigoplus \\ &
& &
\tikzdiagh[xscale=0.5]{0}{
\draw[stdhl](0,0)--(0,0.5);
\draw[stdhl](1,0)--(1,0.5);
\draw[stdhl](2,0)--(2,0.5);
\draw[stdhl](3,0)--(3,0.5);
\draw (0.5,0)--(0.5,0.5);
\draw (2.5,0)--(2.5,0.5);
\filldraw [fill=white, draw=black] (-0.5,.5) rectangle (3.5,1) node[midway] { $\Tb$};
}
\ar{rd}
\ar[dash]{rd}{
\tikzdiagh[xscale=0.4,yscale=0.75]{0}{
\draw[stdhl](0,0)--(0,0.5);
\draw[stdhl](1,0)--(1,0.5);
\draw[stdhl](2,0)--(2,0.5);
\draw[stdhl](3,0)--(3,0.5);
\draw (0.5,0)--(0.5,0.5);
\draw (1.5,0)..controls (1.5,.25) and (2.5,.25) .. (2.5,0.5);
}
}
&\\ & &
q \hspace{2mm}\tikzdiagh[xscale=0.5]{0}{
\draw[stdhl](0,0)--(0,0.5);
\draw[stdhl](1,0)--(1,0.5);
\draw[stdhl](2,0)--(2,0.5);
\draw[stdhl](3,0)--(3,0.5);
\draw (0.5,0)--(0.5,0.5);
\draw (1.5,0)--(1.5,0.5);
\filldraw [fill=white, draw=black] (-0.5,.5) rectangle (3.5,1) node[midway] { $\Tb$};
}
\ar{ru}
\ar[dash]{ru}{ - \ \tikzdiag[xscale=0.4,yscale=0.75]{
\draw[stdhl](0,0)--(0,0.5);
\draw[stdhl](1,0)--(1,0.5);
\draw[stdhl](2,0)--(2,0.5);
\draw[stdhl](3,0)--(3,0.5);
\draw (0.5,0)--(0.5,0.5);
\draw (2.5,0)..controls (2.5,.25) and (1.5,.25) .. (1.5,0.5);
}
}
\ar{rd}
\ar[swap,dash]{rd}{
\tikzdiagh[xscale=0.4,yscale=0.75]{0}{
\draw[stdhl](0,0)--(0,0.5);
\draw[stdhl](1,0)--(1,0.5);
\draw[stdhl](2,0)--(2,0.5);
\draw[stdhl](3,0)--(3,0.5);
\draw (0.33,0)--(0.33,0.5);
\draw (.66,0)..controls (.66,.25) and (1.5,.25) .. (1.5,0.5);
}
}
& \bigoplus &
    q^{-1} \hspace{2mm}\tikzdiagh[xscale=0.5]{0}{
\draw[stdhl](0,0)--(0,0.5);
\draw[stdhl](1,0)--(1,0.5);
\draw[stdhl](2,0)--(2,0.5);
\draw[stdhl](3,0)--(3,0.5);
\draw (0.5,0)--(0.5,0.5);
\draw (1.5,0)--(1.5,0.5);
\filldraw [fill=white, draw=black] (-0.5,.5) rectangle (3.5,1) node[midway] { $\Tb$};
}\\ &
& & 
\tikzdiagh[xscale=0.5]{0}{
\draw[stdhl](0,0)--(0,0.5);
\draw[stdhl](1,0)--(1,0.5);
\draw[stdhl](2,0)--(2,0.5);
\draw[stdhl](3,0)--(3,0.5);
\draw (0.33,0)--(0.33,0.5);
\draw (.66,0)--(.66,0.5);
\filldraw [fill=white, draw=black] (-0.5,.5) rectangle (3.5,1) node[midway] { $\Tb$};
\node[orange] at(0.5,.75){$\bullet$};
}
\ar{ru}
\ar[swap,dash]{ru}{
\tikzdiagh[xscale=0.4,yscale=0.75]{0}{
\draw[stdhl](0,0)--(0,0.5);
\draw[stdhl](1,0)--(1,0.5);
\draw[stdhl](2,0)--(2,0.5);
\draw[stdhl](3,0)--(3,0.5);
\draw (0.33,0)--(0.33,0.5);
\draw (1.5,0)..controls (1.5,.25) and (.66,.25) .. (.66,0.5);
}
}
&
\end{tikzcd}
\end{equation}
The unspecified differentials above moving down diagonally are determined by extending the maps in \eqref{eq:cap_res} through the projective resolutions of each term. Now, we apply the braiding functors to the chain complex above. Notice that it is easy to derive the relevant projective resolutions, 

\[
\begin{tikzcd}[row sep=tiny]
& q^4  \hspace{2mm}\tikzdiagh[xscale=0.5]{0}{
\draw[stdhl](1,0)--(1,0.5);
\draw[stdhl](2,0)--(2,0.5);
\draw (2.4,0)--(2.4,0.5);
\filldraw [fill=white, draw=black] (0,.5) rectangle (3,1) node[midway] { $\Tb$};
\node at(0.1,.3){$...$};
\node at(2.9,.3){$...$};
}
\ar{rdd}
&q^2  \hspace{2mm}\tikzdiagh[xscale=0.5]{0}{
\draw[stdhl](1,0)--(1,0.5);
\draw[stdhl](2,0)--(2,0.5);
\draw (2.4,0)--(2.4,0.5);
\filldraw [fill=white, draw=black] (0,.5) rectangle (3,1) node[midway] { $\Tb$};
\node at(0.1,.3){$...$};
\node at(2.9,.3){$...$};
} 
\ar{rd}
\\
q^5  \hspace{2mm}\tikzdiagh[xscale=0.5]{0}{
\draw[stdhl](1,0)--(1,0.5);
\draw[stdhl](2,0)--(2,0.5);
\draw (1.5,0)--(1.5,0.5);
\filldraw [fill=white, draw=black] (0,.5) rectangle (3,1) node[midway] { $\Tb$};
\node at(0.1,.3){$...$};
\node at(2.9,.3){$...$};
}
\ar{ru}
\ar{rd}
&
\bigoplus & \bigoplus & \Bb_i \circ \Bb_i \bigg( \hspace{2mm}\tikzdiagh[xscale=0.5]{0}{
\draw[stdhl](1,0) node[below]{$i$} --(1,0.5);
\draw[stdhl](2,0) node[below]{$i+1$} --(2,0.5);
\draw (1.5,0)--(1.5,0.5);
\filldraw [fill=white, draw=black] (0,.5) rectangle (3,1) node[midway] { $\Tb$};
\node at(0.1,.3){$...$};
\node at(2.9,.3){$...$};
}\bigg) \\
& 
q^4  \hspace{2mm}\tikzdiagh[xscale=0.5]{0}{
\draw[stdhl](1,0)--(1,0.5);
\draw[stdhl](2,0)--(2,0.5);
\draw (.6,0)--(.6,0.5);
\filldraw [fill=white, draw=black] (0,.5) rectangle (3,1) node[midway] { $\Tb$};
\node at(0.1,.3){$...$};
\node at(2.9,.3){$...$};
}
\ar{ruu}
&
q^2  \hspace{2mm}\tikzdiagh[xscale=0.5]{0}{
\draw[stdhl](1,0)--(1,0.5);
\draw[stdhl](2,0)--(2,0.5);
\draw (.6,0)--(.6,0.5);
\filldraw [fill=white, draw=black] (0,.5) rectangle (3,1) node[midway] { $\Tb$};
\node at(0.1,.3){$...$};
\node at(2.9,.3){$...$};
}
\ar{ru}
\end{tikzcd}
\]

\[\begin{tikzcd}
q\hspace{2mm}\tikzdiagh[xscale=0.5]{0}{
\draw[stdhl](1,0) node[below]{$i$} --(1,0.5);
\draw[stdhl](2,0) node[below]{$i+1$} --(2,0.5);
\filldraw [fill=white, draw=black] (0,.5) rectangle (3,1) node[midway] { $\Tb$};
\node at(0.1,.3){$...$};
\node at(2.9,.3){$...$};
}\ar{r}&
    \Bb_i \circ \Bb_i \bigg( \hspace{2mm}\tikzdiagh[xscale=0.5]{0}{
\draw[stdhl](1,0) node[below]{$i$} --(1,0.5);
\draw[stdhl](2,0) node[below]{$i+1$} --(2,0.5);
\filldraw [fill=white, draw=black] (0,.5) rectangle (3,1) node[midway] { $\Tb$};
\node at(0.1,.3){$...$};
\node at(2.9,.3){$...$};
}\bigg)
\end{tikzcd}\]

Where all arrows above are the obvious permutations of the bottom indempotent. We may notice at this point that after capping twice (as in the diagram we drew for the Hopf link), all terms in \ref{eq:doublecap_res} with an orange bullet \color{orange} $ \bullet \hspace{2mm}$\color{black}  will give a trivial contribution to the final chain complex. So, in the end, the relevant terms that will survive the capping are two isolated shifted copies of, 
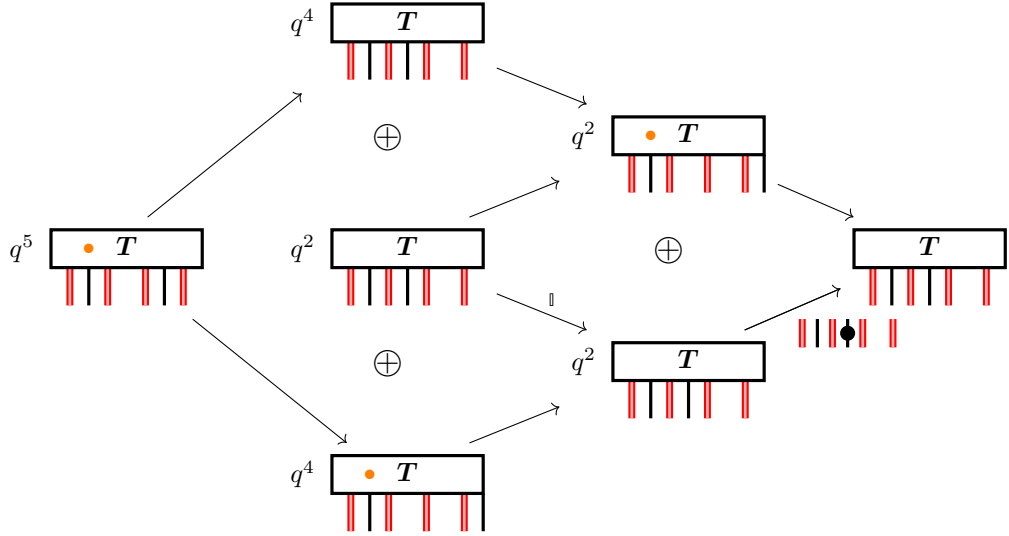
\[
\begin{tikzcd}[row sep =tiny]
& & q^4 \hspace{2mm}
\tikzdiagh[xscale=0.5]{0}{
\draw[stdhl](0,0)--(0,0.5);
\draw[stdhl](1,0)--(1,0.5);
\draw[stdhl](2,0)--(2,0.5);
\draw[stdhl](3,0)--(3,0.5);
\draw (0.5,0)--(0.5,0.5);
\draw (1.5,0)--(1.5,0.5);
\filldraw [fill=white, draw=black] (-0.5,.5) rectangle (3.5,1) node[midway] { $\Tb$};
} \ar{rd} \\
&
& \bigoplus &
q^2\hspace{2mm}\tikzdiagh[xscale=0.5]{0}{
\draw[stdhl](0,0)--(0,0.5);
\draw[stdhl](1,0)--(1,0.5);
\draw[stdhl](2,0)--(2,0.5);
\draw[stdhl](3,0)--(3,0.5);
\draw (0.5,0)--(0.5,0.5);
\draw (3.5,0)--(3.5,0.5);
\filldraw [fill=white, draw=black] (-0.5,.5) rectangle (3.5,1) node[midway] { $\Tb$};
\node[orange] at(0.5,.75){$\bullet$};
}
\ar{rd}
&\\ & q^5 \hspace{2mm}\tikzdiagh[xscale=0.5]{0}{
\draw[stdhl](0,0)--(0,0.5);
\draw[stdhl](1,0)--(1,0.5);
\draw[stdhl](2,0)--(2,0.5);
\draw[stdhl](3,0)--(3,0.5);
\draw (0.5,0)--(0.5,0.5);
\draw (2.5,0)--(2.5,0.5);
\filldraw [fill=white, draw=black] (-0.5,.5) rectangle (3.5,1) node[midway] { $\Tb$};
\node[orange] at(0.5,.75){$\bullet$};
}
\ar{ruu} 
\ar{rdd} 
&
q^2 \hspace{2mm}\tikzdiagh[xscale=0.5]{0}{
\draw[stdhl](0,0)--(0,0.5);
\draw[stdhl](1,0)--(1,0.5);
\draw[stdhl](2,0)--(2,0.5);
\draw[stdhl](3,0)--(3,0.5);
\draw (0.5,0)--(0.5,0.5);
\draw (1.5,0)--(1.5,0.5);
\filldraw [fill=white, draw=black] (-0.5,.5) rectangle (3.5,1) node[midway] { $\Tb$};
}
\ar{ru}
\ar{rd}{\mathbb{I}}
& \bigoplus &
\tikzdiagh[xscale=0.5]{0}{
\draw[stdhl](0,0)--(0,0.5);
\draw[stdhl](1,0)--(1,0.5);
\draw[stdhl](2,0)--(2,0.5);
\draw[stdhl](3,0)--(3,0.5);
\draw (0.5,0)--(0.5,0.5);
\draw (1.5,0)--(1.5,0.5);
\filldraw [fill=white, draw=black] (-0.5,.5) rectangle (3.5,1) node[midway] { $\Tb$};
}\\ &
& \bigoplus & 
q^2 \hspace{2mm}
\tikzdiagh[xscale=0.5]{0}{
\draw[stdhl](0,0)--(0,0.5);
\draw[stdhl](1,0)--(1,0.5);
\draw[stdhl](2,0)--(2,0.5);
\draw[stdhl](3,0)--(3,0.5);
\draw (0.5,0)--(0.5,0.5);
\draw (1.5,0)--(1.5,0.5);
\filldraw [fill=white, draw=black] (-0.5,.5) rectangle (3.5,1) node[midway] { $\Tb$};
}
\ar{ru}
\ar[swap,dash]{ru}{
\tikzdiagh[xscale=0.4,yscale=0.75]{0}{
\draw[stdhl](0,0)--(0,0.5);
\draw[stdhl](1,0)--(1,0.5);
\draw[stdhl](2,0)--(2,0.5);
\draw[stdhl](3,0)--(3,0.5);
\draw (0.5,0)--(0.5,0.5);
\draw (1.5,0)-- (1.5,0.5) node[pos=.5,tikzdot]{} ;
}
}
& \\ 
& & q^4\hspace{2mm}\tikzdiagh[xscale=0.5]{0}{
\draw[stdhl](0,0)--(0,0.5);
\draw[stdhl](1,0)--(1,0.5);
\draw[stdhl](2,0)--(2,0.5);
\draw[stdhl](3,0)--(3,0.5);
\draw (0.5,0)--(0.5,0.5);
\draw (3.5,0)--(3.5,0.5);
\filldraw [fill=white, draw=black] (-0.5,.5) rectangle (3.5,1) node[midway] { $\Tb$};
\node[orange] at(0.5,.75){$\bullet$};
}
\ar{ru}
\end{tikzcd}
\]
Again, we label the projectives that are killed upon capping by an orange bullet. Note that the map with a black dot is trivial after capping. Therefore, the two terms connected by the isomorphism in the middle don't contribute to the homology. As a result, the homology of the Hopf Link is simply, 
$$\mathcal{H}^{S^3}(H,\kk) \cong q^{2} \hspace{1mm} \kk [-1]\oplus q^{4} \hspace{1mm} \kk[1] \oplus q^{2}\hspace{1mm} \kk[1] \oplus q^{6} \hspace{1mm} \kk [3]$$

Comparing with $Kh(\cdot)$ as described in \cite{APS}, \cite{Mel}, we have, 
$$\mathcal{H}^{S^3}(H,\kk) \otimes_\kk (q^{-1}\hspace{1mm} \kk \oplus q\hspace{1mm} \kk ) \cong Kh(H,\kk) \otimes_\kk (q^{-1}\hspace{1mm} \kk[-1] \oplus q\hspace{1mm} \kk [1])$$
And more generally, we expect this relation to hold for any link, $L$, 
$$\mathcal{H}^{S^3}(L,\kk) \otimes_\kk Kh(U,\kk) \cong Kh(L,\kk) \otimes_\kk \mathcal{H}^{S^3}(U,\kk)$$

\end{ex}

\subsection{Unknot and Annular Khovanov Homology}\label{sec:AKh}
In \cite{APS}, a generalization of Khovanov Homology for $I$-bundles over surfaces was introduced and studied. In the case of $S^1 \times D^2$, we can directly compare our homology theory (Definition \ref{def:tangle_functor}) to their construction. In this section, we will do this by explicit computations of some simple knots in $S^1\times D^2$. Given the evidence for a possible equivalence between the two homology theories provided here, it would be interesting to compare them in a more systematic manner in the future. \\ \\
One might worry that the blue strands must be traced over their braid in order to speak of the homology associated to a knot complement. However, we remark that the difference between tracing over every strand and tracing over every strand but one is simply a constant overall tensor factor. In the decategorified setting of computing colored Jones polynomials, this is analogous to the fact that we can consider the diagram of a knot as the diagram of a $(1,1)$-tangle but `opening up' one of the strands. The linear map associated to the $(1,1)$-tangle is simply scalar multiplication, with this scalar being the \textit{normalized} colored Jones polynomial. In our categorified setting, the $(1,1)$-tangle associated to the unknot, is a single vertical blue strand. The functor associated to this tangle by Definition \ref{def:tangle_functor} sends projectives to themselves tensor some homology groups with $\kk$ coefficients. Throughout this section, we will write these factors without mention to the actual functor. In other words, let $D=V\cup L$ be the diagram in question, with $V$ in \eqref{eq:diag_decomp} being a single vertical blue strand and $L$ specifying some link in $S^1 \times D^2$. Then, we use the notation, 
$$\Phi(D)(\Pb^\lambda_{\rho}) = \mathcal{H}^{S^1\times D^2}(D,\kk) \otimes_\kk\Pb^\lambda_{\rho}$$

All computations in Annular Khovanov Homology were done closely following \cite{Mel}, so we point the unfamiliar reader there for an exposition. 
\begin{rem}
    Below, we will often write the homologies for both Annular Khovanov Homology ($AKh(\hspace{1mm}\cdot\hspace{1mm}, \kk)$) and $\mathcal{H}^{S^1\times D^2}(\hspace{1mm}\cdot\hspace{1mm},\kk)$ up to an overall grading shift. This can easily be fixed by ensuring the Euler characteristics of the homologies match, but we will not do this here.
\end{rem}
\begin{ex}{\textbf{Unknot on Non-Trivial Cycle.}}\label{ex:U1} Here we focus on the homology of an unknot wrapped around the $1\in \pi_1(S^1\times D^2)$ class. We will call this knot $U_1$. On the Annular Khovanov Homology side, we draw this as, 
$$\tikzdiag[yscale=1]{
\draw[rstdhl] (0,0) circle (0.5);
\node[blue] at (0,0){$\times$}; 
} $$
It has a trivial cube of resolutions and its homology is simply, 
\be \label{eq:AkhU1}AKh(U_1,\kk) \cong x q\hspace{1mm} \kk \oplus x^{-1} q^{-1}\hspace{1mm} \kk
\ee 

In the KLRW approach, we draw this knot as, 
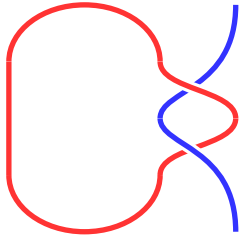
\[
\tikzdiag[yscale=0.75]{
\draw[rstdhl] (1,0) arc (0:-180:1);
\draw[rstdhl] (1,0) ..controls (1,0.5) and (2,0.5) .. (2,1);
\draw[bstdhl] (2,-1) ..controls (2,0.5) and (1,0.5) .. (1,1);
\draw[bstdhl] (1,1) ..controls (1,1.5) and (2,1.5) .. (2,3);
\draw[rstdhl] (2,1) ..controls (2,1.5) and (1,1.5) .. (1,2);
\draw[rstdhl] (-1,0) -- (-1,2);
\draw[rstdhl] (1,2) arc (0:180:1);
}
\]
To compute its homology, we begin by applying the cup to the projective module $\Pb^{\lambda, \lambda}_{0,0}$ and computing its projective resolution, 
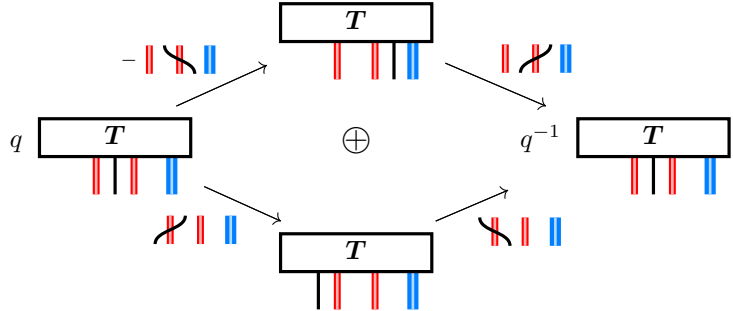
\[
\begin{tikzcd}[row sep=tiny]
    &
 \tikzdiagh[xscale=0.5]{0}{
\draw[stdhl](1,0)--(1,0.5);
\draw[stdhl](2,0)--(2,0.5);
\draw[vstdhl](3,0)--(3,0.5);
\draw (2.5,0)--(2.5,0.5);
\filldraw [fill=white, draw=black] (-0.5,.5) rectangle (3.5,1) node[midway] { $\Tb$};
}
\ar{rd}
\ar[dash]{rd}{
\tikzdiagh[xscale=0.4,yscale=0.75]{0}{
\draw[stdhl](1,0)--(1,0.5);
\draw[stdhl](2,0)--(2,0.5);
\draw[vstdhl](3,0)--(3,0.5);
\draw (1.5,0)..controls (1.5,.25) and (2.5,.25) .. (2.5,0.5);
}
}
&\\ 
q  \hspace{2mm}\tikzdiagh[xscale=0.5]{0}{
\draw[stdhl](1,0)--(1,0.5);
\draw[stdhl](2,0)--(2,0.5);
\draw[vstdhl](3,0)--(3,0.5);
\draw (1.5,0)--(1.5,0.5);
\filldraw [fill=white, draw=black] (-0.5,.5) rectangle (3.5,1) node[midway] { $\Tb$};
}
\ar{ru}
\ar[dash]{ru}{ - \ \tikzdiag[xscale=0.4,yscale=0.75]{
\draw[stdhl](1,0)--(1,0.5);
\draw[stdhl](2,0)--(2,0.5);
\draw[vstdhl](3,0)--(3,0.5);
\draw (2.5,0)..controls (2.5,.25) and (1.5,.25) .. (1.5,0.5);
}
}
\ar{rd}
\ar[swap,dash]{rd}{
\tikzdiagh[xscale=0.4,yscale=0.75]{0}{
\draw[stdhl](1,0)--(1,0.5);
\draw[stdhl](2,0)--(2,0.5);
\draw[vstdhl](3,0)--(3,0.5);
\draw (.5,0)..controls (.5,.25) and (1.5,.25) .. (1.5,0.5);
}
}
& \bigoplus &
q^{-1} \hspace{2mm }\tikzdiagh[xscale=0.5]{0}{
\draw[stdhl](1,0)--(1,0.5);
\draw[stdhl](2,0)--(2,0.5);
\draw[vstdhl](3,0)--(3,0.5);
\draw (1.5,0)--(1.5,0.5);
\filldraw [fill=white, draw=black] (-0.5,.5) rectangle (3.5,1) node[midway] { $\Tb$};
}\\
 & 
\tikzdiagh[xscale=0.5]{0}{
\draw[stdhl](1,0)--(1,0.5);
\draw[stdhl](2,0)--(2,0.5);
\draw[vstdhl](3,0)--(3,0.5);
\draw (.5,0)--(.5,0.5);
\filldraw [fill=white, draw=black] (-0.5,.5) rectangle (3.5,1) node[midway] { $\Tb$};
}
\ar{ru}
\ar[swap,dash]{ru}{
\tikzdiagh[xscale=0.4,yscale=0.75]{0}{
\draw[stdhl](1,0)--(1,0.5);
\draw[stdhl](2,0)--(2,0.5);
\draw[vstdhl](3,0)--(3,0.5);
\draw (1.5,0)..controls (1.5,.25) and (.5,.25) .. (.5,0.5);
}
}
& \\
\end{tikzcd}
\]
In applying the braiding functor, only the topmost projective contributes a resolution of length $>1$. Applying the braiding functor once and resolving gives, 
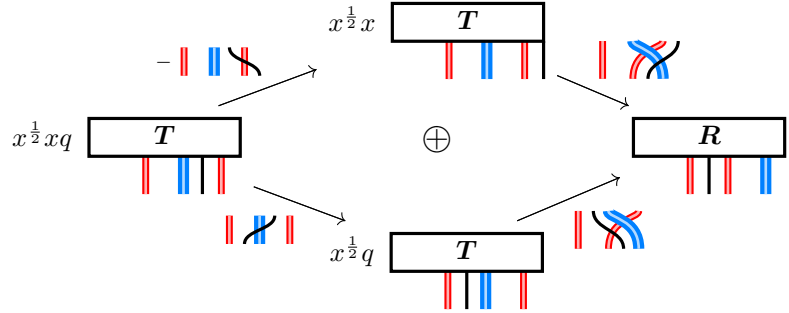
\[
\begin{tikzcd}[row sep=tiny]
    &
x^{\frac{1}{2}} x \hspace{2mm}\tikzdiagh[xscale=0.5]{0}{
\draw[stdhl](1,0)--(1,0.5);
\draw[vstdhl](2,0)--(2,0.5);
\draw[stdhl](3,0)--(3,0.5);
\draw (3.5,0)--(3.5,0.5);
\filldraw [fill=white, draw=black] (-0.5,.5) rectangle (3.5,1) node[midway] { $\Tb$};
}
\ar{rd}
\ar[dash]{rd}{
\tikzdiagh[xscale=0.4]{0}{
\draw[stdhl](1,0)--(1,0.5);
\draw[stdhl](2,0)..controls (2,.4) and (3,.4) ..(3,0.5);
\draw[vstdhl](3,0)..controls (3,0.4) and (2,0.4) .. (2,0.5);
\draw (2.5,0)..controls (2.5,.25) and (3.5,.25) .. (3.5,0.5);
}
}
&\\ 
x^{\frac{1}{2}} x q  \hspace{2mm}\tikzdiagh[xscale=0.5]{0}{
\draw[stdhl](1,0)--(1,0.5);
\draw[vstdhl](2,0)--(2,0.5);
\draw[stdhl](3,0)--(3,0.5);
\draw (2.5,0)--(2.5,0.5);
\filldraw [fill=white, draw=black] (-0.5,.5) rectangle (3.5,1) node[midway] { $\Tb$};
}
\ar{ru}
\ar[dash]{ru}{ - \ \tikzdiag[xscale=0.4,yscale=0.75]{
\draw[stdhl](1,0)--(1,0.5);
\draw[vstdhl](2,0)--(2,0.5);
\draw[stdhl](3,0)--(3,0.5);
\draw (3.5,0)..controls (3.5,.25) and (2.5,.25) .. (2.5,0.5);
}
}
\ar{rd}
\ar[swap,dash]{rd}{
\tikzdiagh[xscale=0.4,yscale=0.75]{0}{
\draw[stdhl](1,0)--(1,0.5);
\draw[vstdhl](2,0)--(2,0.5);
\draw[stdhl](3,0)--(3,0.5);
\draw (1.5,0)..controls (1.5,.25) and (2.5,.25) .. (2.5,0.5);
}
}
& \bigoplus &
\tikzdiagh[xscale=0.5]{0}{
\draw[stdhl](1,0)--(1,0.5);
\draw[stdhl](2,0)--(2,0.5);
\draw[vstdhl](3,0)--(3,0.5);
\draw (1.5,0)--(1.5,0.5);
\filldraw [fill=white, draw=black] (-0.5,.5) rectangle (3.5,1) node[midway] { $\Rb$};
}\\
 & 
x^{\frac{1}{2}}q \hspace{2mm}\tikzdiagh[xscale=0.5]{0}{
\draw[stdhl](1,0)--(1,0.5);
\draw[vstdhl](2,0)--(2,0.5);
\draw[stdhl](3,0)--(3,0.5);
\draw (1.5,0)--(1.5,0.5);
\filldraw [fill=white, draw=black] (-0.5,.5) rectangle (3.5,1) node[midway] { $\Tb$};
}
\ar{ru}
\ar[swap,dash]{ru}{
\tikzdiagh[xscale=0.4]{0}{
\draw[stdhl](1,0)--(1,0.5);
\draw[stdhl](2,0)..controls (2,.4) and (3,.4) ..(3,0.5);
\draw[vstdhl](3,0)..controls (3,0.4) and (2,0.4) .. (2,0.5);
\draw (2.5,0)..controls (2.5,.25) and (1.5,.25) .. (1.5,0.5);
}
}
& \\
\end{tikzcd}
\]
Applying it once more and resolving, we find the chain complex, 
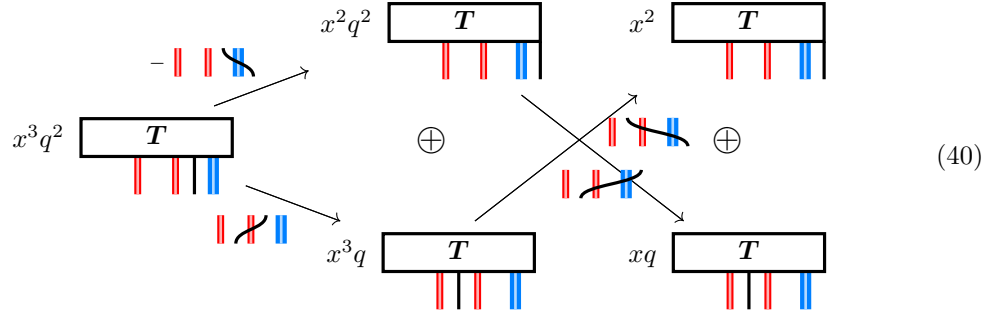
\be
\begin{tikzcd}[row sep=tiny]
    &
x^2 q^2 \hspace{2mm}\tikzdiagh[xscale=0.5]{0}{
\draw[stdhl](1,0)--(1,0.5);
\draw[stdhl](2,0)--(2,0.5);
\draw[vstdhl](3,0)--(3,0.5);
\draw (3.5,0)--(3.5,0.5);
\filldraw [fill=white, draw=black] (-0.5,.5) rectangle (3.5,1) node[midway] { $\Tb$};
}
\ar{rdd}
\ar[dash]{rdd}{
\tikzdiagh[xscale=0.4,yscale=0.75]{0}{
\draw[stdhl](1,0)--(1,0.5);
\draw[stdhl](2,0)--(2,0.5);
\draw[vstdhl](3,0)-- (3,0.5);
\draw (3.5,0)..controls (3.5,.25) and (1.5,.25) .. (1.5,0.5);
}
}
&x^2\hspace{2mm}\tikzdiagh[xscale=0.5]{0}{
\draw[stdhl](1,0)--(1,0.5);
\draw[stdhl](2,0)--(2,0.5);
\draw[vstdhl](3,0)--(3,0.5);
\draw (3.5,0)--(3.5,0.5);
\filldraw [fill=white, draw=black] (-0.5,.5) rectangle (3.5,1) node[midway] { $\Tb$};
}\\ 
x^3 q^2  \hspace{2mm}\tikzdiagh[xscale=0.5]{0}{
\draw[stdhl](1,0)--(1,0.5);
\draw[stdhl](2,0)--(2,0.5);
\draw[vstdhl](3,0)--(3,0.5);
\draw (2.5,0)--(2.5,0.5);
\filldraw [fill=white, draw=black] (-0.5,.5) rectangle (3.5,1) node[midway] { $\Tb$};
}
\ar{ru}
\ar[dash]{ru}{ - \ \tikzdiag[xscale=0.4,yscale=0.75]{
\draw[stdhl](1,0)--(1,0.5);
\draw[stdhl](2,0)--(2,0.5);
\draw[vstdhl](3,0)--(3,0.5);
\draw (3.5,0)..controls (3.5,.25) and (2.5,.25) .. (2.5,0.5);
}
}
\ar{rd}
\ar[swap,dash]{rd}{
\tikzdiagh[xscale=0.4,yscale=0.75]{0}{
\draw[stdhl](1,0)--(1,0.5);
\draw[stdhl](2,0)--(2,0.5);
\draw[vstdhl](3,0)--(3,0.5);
\draw (1.5,0)..controls (1.5,.25) and (2.5,.25) .. (2.5,0.5);
}
}
& \bigoplus &
\bigoplus \\
 & 
x^{3}q \hspace{2mm}\tikzdiagh[xscale=0.5]{0}{
\draw[stdhl](1,0)--(1,0.5);
\draw[stdhl](2,0)--(2,0.5);
\draw[vstdhl](3,0)--(3,0.5);
\draw (1.5,0)--(1.5,0.5);
\filldraw [fill=white, draw=black] (-0.5,.5) rectangle (3.5,1) node[midway] { $\Tb$};
}
\ar{ruu}
\ar[swap, dash]{ruu}{
\tikzdiagh[xscale=0.4, yscale=0.75]{0}{
\draw[stdhl](1,0)--(1,0.5);
\draw[stdhl](2,0)--(2,0.5);
\draw[vstdhl](3,0)-- (3,0.5);
\draw (1.5,0)..controls (1.5,.25) and (3.5,.25) .. (3.5,0.5);
}
}
& x q\hspace{2mm}\tikzdiagh[xscale=0.5]{0}{
\draw[stdhl](1,0)--(1,0.5);
\draw[stdhl](2,0)--(2,0.5);
\draw[vstdhl](3,0)--(3,0.5);
\draw (1.5,0)--(1.5,0.5);
\filldraw [fill=white, draw=black] (-0.5,.5) rectangle (3.5,1) node[midway] { $\Tb$};
}\\
\end{tikzcd}
\label{eq:R^2_res}
\ee
Now, we apply the cap and only the projectives with a black strand between the two red strands survive. Taking the homology of the total complex gives, 
$$\mathcal{H}^{S^1\times D^2}(U_1,\kk) \cong x^{-1}q^{-1}\hspace{1mm}\kk[-1] \oplus xq \hspace{1mm}\kk [1] $$
This is not the same as \eqref{eq:AkhU1}, but they differ in a very similar way to the unknots in the $S^3$ case. We will see this is likely a more general feature of the two theories. 

\end{ex}

\begin{ex}{\textbf{Trefoil on Non-Trivial Cycle.}} 
Let us consider an embedded ``trefoil"  $K\subset S^1\times D^2$, which may be pictured as, 
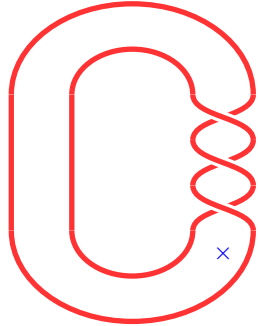
\[
\tikzdiag[xscale = 0.8,yscale=0.6]{
\draw[rstdhl] (1,0) arc (0:-180:1);
\draw[rstdhl] (2,0) arc (0:-180:2);
\draw[rstdhl] (1,0) ..controls (1,0.5) and (2,0.5) .. (2,1);
\draw[rstdhl] (2,0) ..controls (2,0.5) and (1,0.5) .. (1,1);
\draw[rstdhl] (1,1) ..controls (1,1.5) and (2,1.5) .. (2,2);
\draw[rstdhl] (2,1) ..controls (2,1.5) and (1,1.5) .. (1,2);
\draw[rstdhl] (1,2) ..controls (1,2.5) and (2,2.5) .. (2,3);
\draw[rstdhl] (2,2) ..controls (2,2.5) and (1,2.5) .. (1,3);
\draw[rstdhl] (-1,0) -- (-1,3);
\draw[rstdhl] (-2,0) -- (-2,3);
\draw[rstdhl] (1,3) arc (0:180:1);
\draw[rstdhl] (2,3) arc (0:180:2);
\node[blue] at (1.5,-0.5){$\times$};
}
\]
The cube of resolutions for the chain complex in AKh is given by, 
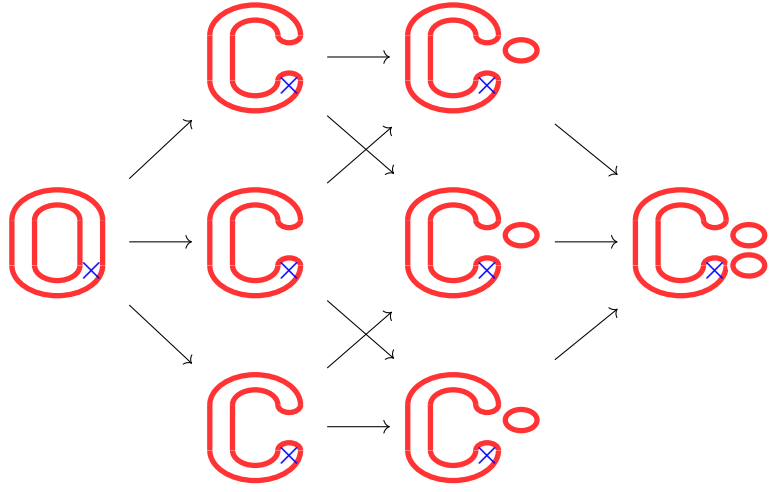
\[
\begin{tikzcd}
& 
\tikzdiag[xscale = 0.3,yscale=0.2]{
\draw[rstdhl] (1,0) arc (0:-180:1);
\draw[rstdhl] (2,0) arc (0:-180:2);
\draw[rstdhl] (2,3) arc (0:-180:0.5);
\draw[rstdhl] (2,0) arc (0:180:0.5);
\draw[rstdhl] (-1,0) -- (-1,3);
\draw[rstdhl] (-2,0) -- (-2,3);
\draw[rstdhl] (1,3) arc (0:180:1);
\draw[rstdhl] (2,3) arc (0:180:2);
\node[blue] at (1.5,-0.5){ \Large $\times$};
}
\ar{r}
\ar{rd}
&
\tikzdiag[xscale = 0.3,yscale=0.2]{
\draw[rstdhl] (1,0) arc (0:-180:1);
\draw[rstdhl] (2,0) arc (0:-180:2);
\draw[rstdhl] (2,3) arc (0:-180:0.5);
\draw[rstdhl] (2,0) arc (0:180:0.5);
\draw[rstdhl] (-1,0) -- (-1,3);
\draw[rstdhl] (-2,0) -- (-2,3);
\draw[rstdhl] (1,3) arc (0:180:1);
\draw[rstdhl] (2,3) arc (0:180:2);
\node[blue] at (1.5,-0.5){ \Large $\times$};
\draw[rstdhl] (3,2) circle (0.7);
}
\ar{rd}
\\
\tikzdiag[xscale = 0.3,yscale=0.2]{
\draw[rstdhl] (1,0) arc (0:-180:1);
\draw[rstdhl] (2,0) arc (0:-180:2);
\draw[rstdhl] (1,0) -- (1,3);
\draw[rstdhl] (2,0) --(2,3);
\draw[rstdhl] (-1,0) -- (-1,3);
\draw[rstdhl] (-2,0) -- (-2,3);
\draw[rstdhl] (1,3) arc (0:180:1);
\draw[rstdhl] (2,3) arc (0:180:2);
\node[blue] at (1.5,-0.5){ \Large $\times$};
} 
\ar{ru}
\ar{r}
\ar{rd}
& 
\tikzdiag[xscale = 0.3,yscale=0.2]{
\draw[rstdhl] (1,0) arc (0:-180:1);
\draw[rstdhl] (2,0) arc (0:-180:2);
\draw[rstdhl] (2,3) arc (0:-180:0.5);
\draw[rstdhl] (2,0) arc (0:180:0.5);
\draw[rstdhl] (-1,0) -- (-1,3);
\draw[rstdhl] (-2,0) -- (-2,3);
\draw[rstdhl] (1,3) arc (0:180:1);
\draw[rstdhl] (2,3) arc (0:180:2);
\node[blue] at (1.5,-0.5){ \Large $\times$};
} 
\ar{ru}
\ar{rd}
& 
\tikzdiag[xscale = 0.3,yscale=0.2]{
\draw[rstdhl] (1,0) arc (0:-180:1);
\draw[rstdhl] (2,0) arc (0:-180:2);
\draw[rstdhl] (2,3) arc (0:-180:0.5);
\draw[rstdhl] (2,0) arc (0:180:0.5);
\draw[rstdhl] (-1,0) -- (-1,3);
\draw[rstdhl] (-2,0) -- (-2,3);
\draw[rstdhl] (1,3) arc (0:180:1);
\draw[rstdhl] (2,3) arc (0:180:2);
\node[blue] at (1.5,-0.5){ \Large $\times$};
\draw[rstdhl] (3,2) circle (0.7);
}
\ar{r}
& 
\tikzdiag[xscale = 0.3,yscale=0.2]{
\draw[rstdhl] (1,0) arc (0:-180:1);
\draw[rstdhl] (2,0) arc (0:-180:2);
\draw[rstdhl] (2,3) arc (0:-180:0.5);
\draw[rstdhl] (2,0) arc (0:180:0.5);
\draw[rstdhl] (-1,0) -- (-1,3);
\draw[rstdhl] (-2,0) -- (-2,3);
\draw[rstdhl] (1,3) arc (0:180:1);
\draw[rstdhl] (2,3) arc (0:180:2);
\node[blue] at (1.5,-0.5){ \Large $\times$};
\draw[rstdhl] (3,2) circle (0.7);
\draw[rstdhl] (3,0) circle (0.7);
}\\
& 
\tikzdiag[xscale = 0.3,yscale=0.2]{
\draw[rstdhl] (1,0) arc (0:-180:1);
\draw[rstdhl] (2,0) arc (0:-180:2);
\draw[rstdhl] (2,3) arc (0:-180:0.5);
\draw[rstdhl] (2,0) arc (0:180:0.5);
\draw[rstdhl] (-1,0) -- (-1,3);
\draw[rstdhl] (-2,0) -- (-2,3);
\draw[rstdhl] (1,3) arc (0:180:1);
\draw[rstdhl] (2,3) arc (0:180:2);
\node[blue] at (1.5,-0.5){ \Large $\times$};
}
\ar{ru}
\ar{r}
&
\tikzdiag[xscale = 0.3,yscale=0.2]{
\draw[rstdhl] (1,0) arc (0:-180:1);
\draw[rstdhl] (2,0) arc (0:-180:2);
\draw[rstdhl] (2,3) arc (0:-180:0.5);
\draw[rstdhl] (2,0) arc (0:180:0.5);
\draw[rstdhl] (-1,0) -- (-1,3);
\draw[rstdhl] (-2,0) -- (-2,3);
\draw[rstdhl] (1,3) arc (0:180:1);
\draw[rstdhl] (2,3) arc (0:180:2);
\node[blue] at (1.5,-0.5){ \Large $\times$};
\draw[rstdhl] (3,2) circle (0.7);
}
\ar{ru}
\end{tikzcd}
\]

We find that the homology of this complex can be computed to be, 
$$AKh(K,\kk) \cong \begin{matrix}
    x\hspace{1mm}\kk \oplus x^{-1}q^{-2}\hspace{1mm}\kk\oplus \\
    xq^4\hspace{1mm}\kk[2] \oplus x^{-1}q^2\hspace{1mm}\kk[2] \oplus \\
    xq^6\hspace{1mm}\kk[3] \oplus x^{-1}q^4\hspace{1mm}\kk[3]
\end{matrix}$$

In the KLRW approach, we consider the representative diagram for $K$,
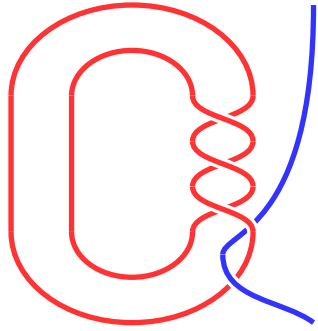
\[
\tikzdiag[xscale = 0.8,yscale=0.6]{
\draw[bstdhl] (1.5,-0.5)..controls (1.5,0) and (3, 0) .. (3,5);
\draw[rstdhl] (1,0) arc (0:-180:1);
\draw[rstdhl] (2,0) arc (0:-180:2);
\draw[rstdhl] (1,0) ..controls (1,0.5) and (2,0.5) .. (2,1);
\draw[rstdhl] (2,0) ..controls (2,0.5) and (1,0.5) .. (1,1);
\draw[rstdhl] (1,1) ..controls (1,1.5) and (2,1.5) .. (2,2);
\draw[rstdhl] (2,1) ..controls (2,1.5) and (1,1.5) .. (1,2);
\draw[rstdhl] (1,2) ..controls (1,2.5) and (2,2.5) .. (2,3);
\draw[rstdhl] (2,2) ..controls (2,2.5) and (1,2.5) .. (1,3);
\draw[rstdhl] (-1,0) -- (-1,3);
\draw[rstdhl] (-2,0) -- (-2,3);
\draw[rstdhl] (1,3) arc (0:180:1);
\draw[rstdhl] (2,3) arc (0:180:2);
\draw[bstdhl] (3,-2) ..controls (2.5,-1.5) and (1.5,-1.5) .. (1.5,-0.5);
}
\]

We omit the calculation for the homology of this diagram and leave it as an exercise to the reader. In the next example, we give another way to derive this result. We find, 
$$\mathcal{H}^{S^1\times D^2}(K,\kk) \cong x\hspace{1mm} \kk \hspace{1mm} \oplus x^{-1} q^{-2} \hspace{1mm} \kk[2]\hspace{1mm} \oplus x q^{4} \hspace{1mm} \kk[2]\hspace{1mm} \oplus x q^{6} \hspace{1mm} \kk[3]\hspace{1mm} \oplus  x^{-1} q^{2} \hspace{1mm} \kk[4]\hspace{1mm} \oplus x^{-1} q^{4} \hspace{1mm} \kk[5]$$

It is easy to see, 
$$\mathcal{H}^{S^1\times D^2}(K,\kk) \otimes_\kk AKh(U_1,\kk) \cong AKh(K,\kk) \otimes_\kk \mathcal{H}^{S^1\times D^2}(U_1,\kk)$$
\end{ex}

\begin{ex}
Let us generalize the previous example and consider any knot $K \subset S^1 \times D^2$ that can be described by a diagram of the kind, 
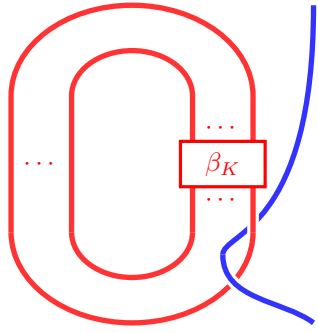
\[
\tikzdiag[xscale = 0.8,yscale=0.6]{
\draw[bstdhl] (1.5,-0.5)..controls (1.5,0) and (3, 0) .. (3,5);
\draw[rstdhl] (1,0) arc (0:-180:1);
\draw[rstdhl] (2,0) arc (0:-180:2);
\draw[rstdhl] (2,0) -- (2,3);
\draw[rstdhl] (1,0) -- (1,3);
\draw[rstdhl] (-1,0) -- (-1,3);
\draw[rstdhl] (-2,0) -- (-2,3);
\draw[rstdhl] (1,3) arc (0:180:1);
\draw[rstdhl] (2,3) arc (0:180:2);
\draw[bstdhl] (3,-2) ..controls (2.5,-1.5) and (1.5,-1.5) .. (1.5,-0.5);
\filldraw [fill=white, draw=red] (0.8,1) rectangle (2.2,2) node[red,midway] {$\beta_K$};
\node[red] at (1.5,0.7) {$\cdots$};
\node[red] at (1.5,2.3) {$\cdots$};
\node[red] at (-1.5,1.5) {$\cdots$};
}
\]
Where $\beta_K$ is any braid word. Let $K' \subset S^3$ be the knot acquired from the diagram above by erasing the blue strand. We can write the homology of $K'$ and $K$ as, 
$$\mathcal{H}^{S^3}(K',\kk) \cong \mathcal{H} (\mathcal{F}_\beta (\cupp))$$
$$\mathcal{H}^{S^1\times D^2}(K,\kk) \cong \mathcal{H} \big((\mathcal{F}_\beta \circ \Bb_2 \circ \Bb_2)(\cupp)\big)$$
where $\mathcal{F}_\beta$ is some composition of braiding, cap, and cup functors. Furthermore, if we recall the resolution of the cap \eqref{eq:cap_res}, we can notice that the projective in the middle never contributes to $\mathcal{H} (\mathcal{F}_\beta (\cupp))$ since it has a black strand to the right of all red strands and the gradual application of $\mathcal{F}_\beta$ never changes this at any step. Therefore, when we eventually apply the final cap, the contributions from this projective will all be trivial. Thus, we can substitute the resolution of $\cupp$ above with the simpler complex, 
\begin{equation}
\begin{tikzcd} q \hspace{2mm}
\tikzdiagh[xscale=0.75]{0}{
\draw[stdhl](0,0)--(0,0.5);
\draw[stdhl](1,0)--(1,0.5);
\draw (0.5,0)--(0.5,0.5);
\filldraw [fill=white, draw=black] (-0.25,.5) rectangle (1.5,1) node[midway] { $\Tb$};
}
\ar{rr}
& &
0
\ar{rr}
& & q^{-1} \hspace{2mm}
\tikzdiagh[xscale=0.75]{0}{
\draw[stdhl](0,0)--(0,0.5);
\draw[stdhl](1,0)--(1,0.5);
\draw (0.5,0)--(0.5,0.5);

\filldraw [fill=white, draw=black] (-0.25,.5) rectangle (1.5,1) node[midway] { $\Tb$};
}
\end{tikzcd}
\label{eq:K'res}
\end{equation}
With this in mind, let us see how $\cupp$ behaves under braiding with the blue strand with the goal of connecting the homologies of $K$ and $K'$. We can write a projective resolution resolution of $(\Bb_2 \circ \Bb_2) (\mathbf{p}\cupp )$,

\[
\begin{tikzcd}[row sep=tiny]
& xq\hspace{1mm}\tikzdiagh[xscale=0.5]{0}{
\draw[stdhl](0,0)--(0,0.5);
\draw[stdhl](1,0)--(1,0.5);
\draw[vstdhl](2,0)--(2,0.5);
\draw (0.5,0)--(0.5,0.5);
\filldraw [fill=white, draw=black] (-0.5,0.5) rectangle (2.5,1) node[midway] { $\Tb$};
}
\ar{r}
\ar{rdddd}
&
x\hspace{1mm}\tikzdiagh[xscale=0.5]{0}{
\draw[stdhl](0,0)--(0,0.5);
\draw[stdhl](1,0)--(1,0.5);
\draw[vstdhl](2,0)--(2,0.5);
\draw (-0.3,0)--(-0.3,0.5);
\filldraw [fill=white, draw=black] (-0.5,0.5) rectangle (2.5,1) node[midway] { $\Tb$};
\node[orange] at (0.4,0.75){$\bullet$};
}
\ar{rdd}
& \\ & \bigoplus &\bigoplus & \\
x^3q^{2}\hspace{1mm}\tikzdiagh[xscale=0.5]{0}{
\draw[stdhl](0,0)--(0,0.5);
\draw[stdhl](1,0)--(1,0.5);
\draw[vstdhl](2,0)--(2,0.5);
\draw (0.5,0)--(0.5,0.5);
\filldraw [fill=white, draw=black] (-0.5,0.5) rectangle (2.5,1) node[midway] { $\Tb$};
\node[orange] at (0.4,0.75){$\bullet$};
}
\ar{r}
\ar{rdd}
&
x^2q^{2}\hspace{1mm}\tikzdiagh[xscale=0.5]{0}{
\draw[stdhl](0,0)--(0,0.5);
\draw[stdhl](1,0)--(1,0.5);
\draw[vstdhl](2,0)--(2,0.5);
\draw (1.5,0)--(1.5,0.5);
\filldraw [fill=white, draw=black] (-0.5,0.5) rectangle (2.5,1) node[midway] { $\Tb$};
\node[orange] at (0.4,0.75){$\bullet$};
}
\ar{rdd}
&
x^2\hspace{1mm}\tikzdiagh[xscale=0.5]{0}{
\draw[stdhl](0,0)--(0,0.5);
\draw[stdhl](1,0)--(1,0.5);
\draw[vstdhl](2,0)--(2,0.5);
\draw (2.3,0)--(2.3,0.5);
\filldraw [fill=white, draw=black] (-0.5,0.5) rectangle (2.5,1) node[midway] { $\Tb$};
\node[orange] at (0.4,0.75){$\bullet$};
}
\ar{r}
&
xq^{-1}\hspace{1mm}\tikzdiagh[xscale=0.5]{0}{
\draw[stdhl](0,0)--(0,0.5);
\draw[stdhl](1,0)--(1,0.5);
\draw[vstdhl](2,0)--(2,0.5);
\draw (0.5,0)--(0.5,0.5);
\filldraw [fill=white, draw=black] (-0.5,0.5) rectangle (2.5,1) node[midway] { $\Tb$};
}\\
& \bigoplus & \bigoplus \\
& x^3q\hspace{1mm}\tikzdiagh[xscale=0.5]{0}{
\draw[stdhl](0,0)--(0,0.5);
\draw[stdhl](1,0)--(1,0.5);
\draw[vstdhl](2,0)--(2,0.5);
\draw (0.5,0)--(0.5,0.5);
\filldraw [fill=white, draw=black] (-0.5,0.5) rectangle (2.5,1) node[midway] { $\Tb$};
}
\ar{ruu}
&
xq\hspace{1mm}\tikzdiagh[xscale=0.5]{0}{
\draw[stdhl](0,0)--(0,0.5);
\draw[stdhl](1,0)--(1,0.5);
\draw[vstdhl](2,0)--(2,0.5);
\draw (0.5,0)--(0.5,0.5);
\filldraw [fill=white, draw=black] (-0.5,0.5) rectangle (2.5,1) node[midway] { $\Tb$};
}
\ar{ruu}
\end{tikzcd}
\]

Notice that all summands with the orange bullet will not contribute to the homology $\mathcal{H}^{S^1\times D^2}(K,\kk)$ since we know we will be capping off the red strands at the top of the diagram, and any projective with a black strand to the right/left of all red strands must eventually be killed by the caps. This allows us to replace the complex above by, 
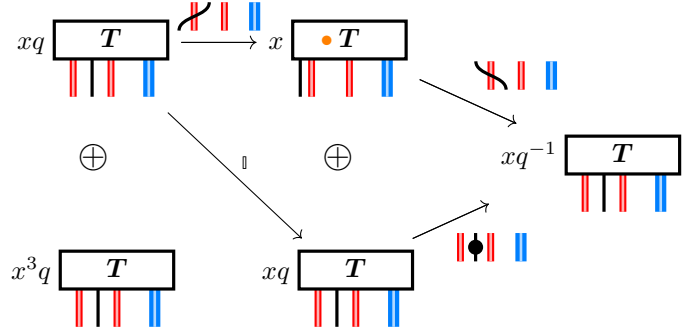
\[
\begin{tikzcd}[row sep=tiny]
& xq\hspace{1mm}\tikzdiagh[xscale=0.5]{0}{
\draw[stdhl](0,0)--(0,0.5);
\draw[stdhl](1,0)--(1,0.5);
\draw[vstdhl](2,0)--(2,0.5);
\draw (0.5,0)--(0.5,0.5);
\filldraw [fill=white, draw=black] (-0.5,0.5) rectangle (2.5,1) node[midway] { $\Tb$};
}
\ar{r}
\ar[dash]{r}{
\tikzdiagh[xscale=0.4,yscale=0.75]{0}{
\draw[stdhl](1,0)--(1,0.5);
\draw[stdhl](2,0)--(2,0.5);
\draw[vstdhl](3,0)--(3,0.5);
\draw (.5,0)..controls (.5,.25) and (1.5,.25) .. (1.5,0.5);
}
}
\ar{rdd}{\mathbb{I}}
&
x\hspace{1mm}\tikzdiagh[xscale=0.5]{0}{
\draw[stdhl](0,0)--(0,0.5);
\draw[stdhl](1,0)--(1,0.5);
\draw[vstdhl](2,0)--(2,0.5);
\draw (-0.3,0)--(-0.3,0.5);
\filldraw [fill=white, draw=black] (-0.5,0.5) rectangle (2.5,1) node[midway] { $\Tb$};
\node[orange] at (0.4,0.75){$\bullet$};
}
\ar{rd}
\ar[dash]{rd}{
\tikzdiagh[xscale=0.4,yscale=0.75]{0}{
\draw[stdhl](1,0)--(1,0.5);
\draw[stdhl](2,0)--(2,0.5);
\draw[vstdhl](3,0)--(3,0.5);
\draw (1.5,0)..controls (1.5,.25) and (.5,.25) .. (.5,0.5);
}
}
& \\ 
& \bigoplus & \bigoplus &
xq^{-1}\hspace{1mm}\tikzdiagh[xscale=0.5]{0}{
\draw[stdhl](0,0)--(0,0.5);
\draw[stdhl](1,0)--(1,0.5);
\draw[vstdhl](2,0)--(2,0.5);
\draw (0.5,0)--(0.5,0.5);
\filldraw [fill=white, draw=black] (-0.5,0.5) rectangle (2.5,1) node[midway] { $\Tb$};
}
\\
& x^3q\hspace{1mm}\tikzdiagh[xscale=0.5]{0}{
\draw[stdhl](0,0)--(0,0.5);
\draw[stdhl](1,0)--(1,0.5);
\draw[vstdhl](2,0)--(2,0.5);
\draw (0.5,0)--(0.5,0.5);
\filldraw [fill=white, draw=black] (-0.5,0.5) rectangle (2.5,1) node[midway] { $\Tb$};
}
&
xq\hspace{1mm}\tikzdiagh[xscale=0.5]{0}{
\draw[stdhl](0,0)--(0,0.5);
\draw[stdhl](1,0)--(1,0.5);
\draw[vstdhl](2,0)--(2,0.5);
\draw (0.5,0)--(0.5,0.5);
\filldraw [fill=white, draw=black] (-0.5,0.5) rectangle (2.5,1) node[midway] { $\Tb$};
}
\ar{ru}
\ar[swap,dash]{ru}{
\tikzdiagh[xscale=0.4,yscale=0.75]{0}{
\draw[stdhl](1,0)--(1,0.5);
\draw[stdhl](2,0)--(2,0.5);
\draw[vstdhl](3,0)--(3,0.5);
\draw (1.5,0) -- (1.5,.5) node[pos=.5,tikzdot]{};
}
}
\end{tikzcd}
\]

Furthermore, we can replace this complex by an even simpler complex by noticing that the top center summand has trivial contributions to the final chain complex and, $$Im\bigg (\tikzdiag[xscale=0.4,yscale=0.75]{
\draw[stdhl](1,0)--(1,0.5);
\draw[stdhl](2,0)--(2,0.5);
\draw[vstdhl](3,0)--(3,0.5);
\draw (1.5,0) -- (1.5,.5) node[pos=.5,tikzdot]{};
}\bigg) 
\subset
Im\bigg (\tikzdiag[xscale=0.4,yscale=0.75]{
\draw[stdhl](1,0)--(1,0.5);
\draw[stdhl](2,0)--(2,0.5);
\draw[vstdhl](3,0)--(3,0.5);
\draw (1.5,0)..controls (1.5,.25) and (.5,.25) .. (.5,0.5);
}\bigg)$$
So, in the end we have the replacement by the complex, 
\[
\begin{tikzcd}
x^3q\hspace{1mm}\tikzdiagh[xscale=0.5]{0}{
\draw[stdhl](0,0)--(0,0.5);
\draw[stdhl](1,0)--(1,0.5);
\draw[vstdhl](2,0)--(2,0.5);
\draw (0.5,0)--(0.5,0.5);
\filldraw [fill=white, draw=black] (-0.5,0.5) rectangle (2.5,1) node[midway] { $\Tb$};
} 
\ar{rr}&& 0 \ar{rr}&& xq^{-1}\hspace{1mm}\tikzdiagh[xscale=0.5]{0}{
\draw[stdhl](0,0)--(0,0.5);
\draw[stdhl](1,0)--(1,0.5);
\draw[vstdhl](2,0)--(2,0.5);
\draw (0.5,0)--(0.5,0.5);
\filldraw [fill=white, draw=black] (-0.5,0.5) rectangle (2.5,1) node[midway] { $\Tb$};
}
\end{tikzcd}
\]
We can compare this to \eqref{eq:K'res} to reach the conclusion that the two homologies are related by an overall tensor factor, 
$$\mathcal{H}^{S^1\times D^2}(K,\kk) \otimes_\kk (q^{-1}\hspace{1mm}\kk[-1]\oplus q\hspace{1mm}\kk[1])\cong \mathcal{H}^{S^3}(K',\kk)\otimes_\kk (x^{-1}q^{-1}\hspace{1mm}\kk[-1]\oplus x q\hspace{1mm}\kk[1])$$

The connection to AKh then follows. However, this calculation hints at a more general feature of the two theories. That is, let $K\subset S^1\times D^2$ be any annular knot. Suppose $[K] = n \in \pi_1(S^1\times D^2)$. Let $U_n$ denote an unknot wrapping the non-trivial cycle of the solid torus $n$ times (see next example for diagram). Then, it is tempting to conjecture, 
\begin{conj}
$\mathcal{H}^{S^1\times D^2}(K,\kk) \otimes_\kk AKh(U_n,\kk) \cong AKh(K,\kk) \otimes_\kk \mathcal{H}^{S^1\times D^2}(U_n,\kk)$
\end{conj}
\end{ex}

\begin{ex}{$\mathbf{U_n.}$} Let $U_n$ be the knot in $S^1\times D^2$ described by the following diagram,
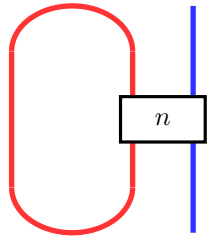
\[
\tikzdiag[xscale = 0.8,yscale=0.6]{
\draw[rstdhl] (1,0) arc (0:-180:1);
\draw[bstdhl] (2,-1) -- (2,4);
\draw[rstdhl] (1,0) -- (1,3);
\draw[rstdhl] (-1,0) -- (-1,3);
\draw[rstdhl] (1,3) arc (0:180:1);
\filldraw [fill=white, draw=black] (0.8,1) rectangle (2.2,2) node[black,midway] {$n$};
}
\]
Where the box indicates $n$ full twists. The calculation of its homology is very similar to the case of $U_1$ in Example \ref{ex:U1}. Except here, we must repeatedly braid and resolve the complex \eqref{eq:R^2_res}. It is easy to do in general and we find its homology to be, 
$$\mathcal{H}^{S^1\times D^2}(U_n, \kk) \cong x^n \bigg( q^{-1}\hspace{1mm} \kk \oplus \bigoplus^n_{i=1} x^{2i}q^{2i-1}\hspace{1mm} \kk[2i] \oplus \bigoplus^{n-1}_{i=1} x^{2i}q^{2i+1}\hspace{1mm} \kk[2i+1]  \bigg)$$

For the AKh of this annular knot, we use the cube of resolutions, 
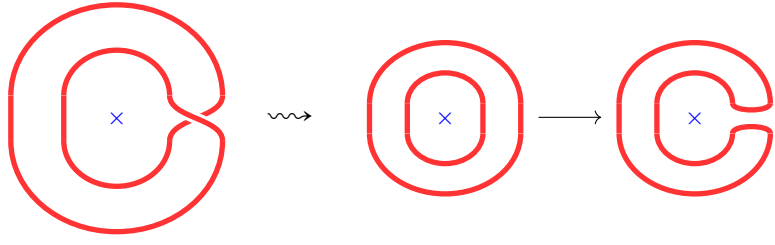
\[
\tikzdiag[xscale = 0.7,yscale=0.6]{
\draw[rstdhl] (1,0) arc (0:-180:1);
\draw[rstdhl] (2,0) arc (0:-180:2);
\draw[rstdhl] (1,0) ..controls (1,0.5) and (2,0.5) .. (2,1);
\draw[rstdhl] (2,0) ..controls (2,0.5) and (1,0.5) .. (1,1);
\draw[rstdhl] (-1,0) -- (-1,1);
\draw[rstdhl] (-2,0) -- (-2,1);
\draw[rstdhl] (1,1) arc (0:180:1);
\draw[rstdhl] (2,1) arc (0:180:2);
\node[blue] at (0,0.5){$\times$};
} \hspace{5mm}
\mbox{\large $\rightsquigarrow$} \hspace{5mm}
\begin{tikzcd}
    \tikzdiag[xscale = 0.5,yscale=0.4]{
\draw[rstdhl] (1,0) arc (0:-180:1);
\draw[rstdhl] (2,0) arc (0:-180:2);
\draw[rstdhl] (1,0) -- (1,1);
\draw[rstdhl] (2,0) -- (2,1);
\draw[rstdhl] (-1,0) -- (-1,1);
\draw[rstdhl] (-2,0) -- (-2,1);
\draw[rstdhl] (1,1) arc (0:180:1);
\draw[rstdhl] (2,1) arc (0:180:2);
\node[blue] at (0,0.5){$\times$};
}
\ar{r}
&
\tikzdiag[xscale = 0.5,yscale=0.4]{
\draw[rstdhl] (1,0) arc (0:-180:1);
\draw[rstdhl] (2,0) arc (0:-180:2);
\draw[rstdhl] (1,0) .. controls (1,0.3) and (2,0.3) .. (2,0);
\draw[rstdhl] (1,1) .. controls (1,0.7) and (2,0.7) .. (2,1);
\draw[rstdhl] (-1,0) -- (-1,1);
\draw[rstdhl] (-2,0) -- (-2,1);
\draw[rstdhl] (1,1) arc (0:180:1);
\draw[rstdhl] (2,1) arc (0:180:2);
\node[blue] at (0,0.5){$\times$};
} 
\end{tikzcd}
\]
to find that, 
$$AKh(U_2, \kk) \cong \kk \oplus x^2 q^2 \hspace{1mm} \kk \oplus x^{-2} q^{-2} \hspace{1mm} \kk \oplus q^2 \hspace{1mm} \kk[1]$$
In particular, the two Poincare polynomials to compare are, 
$$\mathcal{P}(AKh(U_2, \kk)) \cong x^2 q^2 +x^{-2}q^{-2} +1 +q^2 t$$
$$\mathcal{P}(\mathcal{H}^{S^1\times D^2}(U_2, \kk)) \cong x^2 q^2 t^4 +x^{-2}q^{-2} +t^2 +q^2 t^3$$
They could be related in some way (for instance, multiplying the top by $x^2$ and shifting $x\rightarrow x t$), but it seems strange to require them to be the same when even in $S^3$ we see different normalizations... Nevertheless, we can expect that for general $n$, we have,  
$$dim_\kk AKh(U_n, \kk)) = dim_\kk \mathcal{H}^{S^1\times D^2}(U_n, \kk)$$

And more generally, we conjecture the following,

\begin{conj}
For any annular knot $K \subset S^1\times D^2$, we have, 
$$dim_\kk AKh(K, \kk)) = dim_\kk \mathcal{H}^{S^1\times D^2}(K, \kk)$$
\end{conj}

\end{ex}

\section{Discussion} \label{sec:discussion}
\subsection{Towards Cups and Caps for Verma Modules}
In \cite{Web}, a beautiful description of a diagrammatic categorification of the cups and caps for finite dimensional representations was given. We have done our best to translate Webster's description in the case of the fundamental representation in Section \ref{sec:fund_strands}. However, the key point there was choosing a certain identification of the representation and its dual. Unfortunately, this is unavailable for the case of Verma modules and we must instead look to a different identification, which we may coherently generalize to the infinite dimensional case. In this section, we describe such an identification. 
\subsubsection{Decategorified Cups and Caps}\label{sec:decat_cups_caps}
Here, we will explicitly discuss the case for Verma modules, but the reader should note that it includes irreducible, finite-dimensional representations as a subcase (simply set $2\lambda \rightarrow N$ in the expressions). \\ \\
In the decategorified setting of quantum invariants from Verma module representations, a cap ought to be an $\mathfrak{sl_2}$-invariant map, 
$$\hat{C}: M^*_\lambda\otimes M_\lambda \rightarrow \C((x,q))$$
We wish to write this map down explicitly so that we may ponder its categorification. Let us begin by choosing the dual basis in the following way, 

$$v_{\lambda,0}^*(v_{\lambda,i})=c_0 \delta_{i,0}$$

$$v_{\lambda,i}^* := E^iv_{\lambda,0}^*$$

Where $c_0 \in \C((x,q))$ is some constant and $v_{\lambda, i}$ is the basis described in Section \ref{sec: prelim}. It follows by a direct calculation that, 
\begin{prop}
$$v_{\lambda,i}^*(v_{\lambda,j}) = \delta_{i,j} \cdot e_i \cdot c_0$$

$$e_i := (-1)^i x^i q^{i-i^2}\prod^i_{k=1}\bigg( [k]_q [2\lambda-k+1]_q\bigg)$$
\end{prop}
\begin{proof}
    Using the fact that for an element $f\in U_q(\mathfrak{sl_2})$, we have, 
    $$(f \cdot v^*)(v) = v^* (S(f)\cdot v)$$
    for the antipode $S$ described in Section \ref{sec: prelim}, this is a straightforward calculation. 
\end{proof}
Similarly, it is easy to work out the $\mathfrak{sl_2}$ action on this basis, and we find, 

$$K \cdot v_{\lambda,i}^* = x^{-1}q^{2i}v_{\lambda,i}^*$$
$$F\cdot v_{\lambda,i}^* = [i]_q [2\lambda-i+1]_qv_{\lambda,i-1}^*$$
$$E\cdot v_{\lambda,i}^* = v_{\lambda,i+1}^*$$

Using this action, it is easy to write down a suitable cap, 
\begin{prop} \label{prop:decat_cap_formula}
Up to an overall constant, the cap is given by, 
$$\hat{C} (v_{\lambda,i}^* \otimes v_{\lambda,j}) = \delta_{i,j} \cdot e_i $$
\end{prop}
\begin{proof}
    It suffices to check invariance under $\mathfrak{sl}_2$ (i.e. $\hat{C}\bigg(f\cdot(v_{\lambda,i}^* \otimes v_{\lambda,j}) \bigg)  =0$). This is easily done using the coproduct in Section \ref{sec: prelim}. 
\end{proof}

The cup is also easy to work out using the same basis and it is given by, 
\begin{prop} \label{prop:decat_cup_formula}
    Up to an overall constant, the cup is given by, 
    $$\check{C}(1) = \sum^\infty_{i=0} \frac{q^{2i}}{e_i} \cdot v^*_{\lambda,i}\otimes v_{\lambda,i}$$
\end{prop}

The resulting invariant of the unknot is, up to an overall constant, given by, 
$$\hat{C}\big(\check{C}(1)\big) = \sum^\infty_{i=0}q^{2i} = \frac{1}{1-q^2}$$
Of course, the overall constant of the cup and the cap can be fixed by some topological constraints. That is, the unnormalized unknot in the state-sum formalism is given by \cite{Park}, 
$$\tilde{f}_{O} = x^{-1}q\frac{1}{1-q^2}$$

\subsubsection{Finite Dimensional Representations and Oriented Black Strands}
In the section above, we have diverged from \cite{Web} by picking a different basis and identification for the duals. In this section, we will argue that our choices also have a very natural categorification in the KLRW framework by passing to the diagrammatic algebra which includes upward moving strands as well as downward moving strands. This gives us a way of categorifying $E$ and $F$ ``together" and allows us to explicitly write down the bimodules for the caps (and hence the cups as well). \\ \\
We will introduce a new diagrammatic algebra with oriented strands. This was done in \cite{Web}, where they were called ``double tensor product algebras.'' These algebras are Morita-equivalent to KLRW algebras \cite{Web}. Following Webster, we must now keep track of the weight spaces in each region of a diagram. This is done by requiring the labeling of regions to be in line with the following rules, 
\[
\begin{tikzcd}
\tikzdiag[scale=1]{
\draw[mid->] (0,0)--(0,1) node[left,midway]{$\beta \hspace{2mm}$} node[right,midway]{$ \hspace{2mm}\beta - 2$}
} && \tikzdiag[scale=1]{
\draw[mid<-] (0,0)--(0,1) node[left,midway]{$\beta \hspace{2mm}$} node[right,midway]{$ \hspace{2mm}\beta + 2$}
}\\
\tikzdiag[scale=1]{
\draw[pstdhl,mid->] (0,0) node[below]{\small $\mu_i$}--(0,1) node[color=black,left,midway]{$\beta \hspace{2mm}$} node[color=black,right,midway]{$ \hspace{2mm}\beta + \mu_i$}
} && \tikzdiag[scale=1]{
\draw[pstdhl,mid<-] (0,0) node[below]{\small $\mu_i^*$}--(0,1) node[color=black,left,midway]{$\beta \hspace{2mm}$} node[color=black,right,midway]{$ \hspace{2mm}\beta - \mu_i$}
}
\end{tikzcd}
\]
For any given diagram, the labeling of regions is uniquely determined by the rules above, so we will often not write the weights of each region. We will allow black strands to close into bubbles, self-intersect, and form caps. So, there are black caps and cups as new diagrammatic generators, 
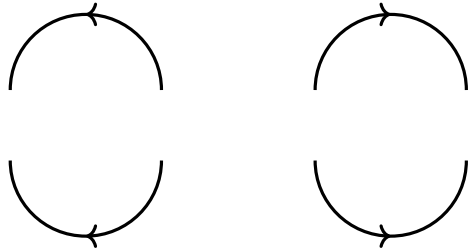
\[
\begin{tikzcd}
\tikzdiag[xscale=1,yscale=1]{
\draw[mid->] (0,0) arc (0:180:1) ;
} && \tikzdiag[xscale=-1,yscale=1]{
\draw[mid->] (0,0) arc (0:180:1);
}\\
\tikzdiag[xscale=1,yscale=-1]{
\draw[mid->] (0,0) arc (0:180:1);
} && \tikzdiag[xscale=-1,yscale=-1]{
\draw[mid->] (0,0) arc (0:180:1);
}
\end{tikzcd}
\]
By a \textbf{double string diagram}, we will mean string diagram with oriented strands as described above where the far left region is labeled by the weight $0$. Note that the string diagrams from the dg-KLRW algebras considered thus far in this paper may be considered as double string diagrams with strictly downward oriented black strands and upward oriented colored strands. 
\begin{defn} \label{defn:double_dg_KLRW}
The \textbf{double dg-KLRW algebra} $D\Tb^{\underline{\mu}}$ is the $\kk$-algebra spanned by double string diagrams such that, 
\begin{itemize}
    \item It is $\Z \times \Z^2$ graded as in Definition \ref{def:dgKLRWalgebra}, with the grading of a digram dictated by, 
     \[
        deg\hspace{2mm} \tikzdiag[scale=1]{
        \draw[upper->] (1,0)..controls (1,0.5) and (0,0.5) .. (0,1);
        \draw[upper->] (0,0) ..controls (0,0.5) and (1,0.5) .. (1,1);
        }\hspace{2mm} = \hspace{2mm} deg\hspace{2mm} \tikzdiag[yscale=-1]{
        \draw[upper->] (1,0)..controls (1,0.5) and (0,0.5) .. (0,1);
        \draw[upper->] (0,0) ..controls (0,0.5) and (1,0.5) .. (1,1);
        }\hspace{2mm} =
        \hspace{2mm} q^{-2}
        \]
        \[
        deg \hspace{2mm} \tikzdiag[yscale=1]{
        \draw[upper->] (0,0)--(0,1) node[tikzdot,pos=0.25]{};
        }\hspace{2mm} = \hspace{2mm} deg \hspace{2mm} \tikzdiag[yscale=1]{
        \draw[upper<-] (0,0)--(0,1) node[tikzdot,pos=0.25]{};
        }\hspace{2mm} = \hspace{2mm} q^2
        \]
        \[\begin{tikzcd}
        deg \hspace{2mm} \tikzdiag[scale=1]{
        \draw[mid<-] (0,0) arc (0:180:0.5) node[midway,below=3pt]{\small $\beta$};
        } \hspace{2mm} = \hspace{2mm} q^{\beta -1} && deg \hspace{2mm} \tikzdiag[scale=1]{
        \draw[mid->] (0,0) arc (0:180:0.5) node[midway,below=3pt]{\small $\beta$};
        } \hspace{2mm} = \hspace{2mm} q^{-\beta -1} \\
        deg \hspace{2mm} \tikzdiag[yscale=-1]{
        \draw[mid->] (0,0) arc (0:180:0.5) node[midway,above=3pt]{\small $\beta$};
        } \hspace{2mm} = \hspace{2mm} q^{\beta -1} && deg \hspace{2mm} \tikzdiag[yscale=-1]{
        \draw[mid<-] (0,0) arc (0:180:0.5) node[midway,above=3pt]{\small $\beta$};
        } \hspace{2mm} = \hspace{2mm} q^{-\beta -1}
        \end{tikzcd}
        \]
        \[
        deg\hspace{2mm} \tikzdiag[scale=1]{
        \draw (1,0)..controls (1,0.5) and (0,0.5) .. (0,1);
        \draw[pstdhl] (0,0) node[below]{\small $\mu_i$} ..controls (0,0.5) and (1,0.5) .. (1,1);
        }\hspace{2mm} = \begin{cases}
            q^{\mu_i} &\text{if the strands are oppositely oriented} \\
            1&\text{if the strands are oriented the same way}
        \end{cases}
        \]
        \[
        deg\hspace{2mm} \tikzdiag[scale=1]{
        \draw[lower->] (1,0)..controls (1,0.5) and (0,0.5) .. (0,0.5);
        \draw[upper->] (0,0.5) ..controls (0.02,0.5) and (1,0.5) .. (1,1);
        \draw[pstdhl,<-] (0,0) node[below=3pt]{\small $\mu_i$} --(0,1) node[midway,nail]{};
        } \hspace{2mm} = \hspace{2mm} deg\hspace{2mm} \tikzdiag[scale=1]{
        \draw[lower<-] (1,0)..controls (1,0.5) and (0,0.5) .. (0,0.5);
        \draw[upper<-] (0,0.5) ..controls (0.02,0.5) and (1,0.5) .. (1,1);
        \draw[pstdhl,->] (0,0) node[below=3pt]{\small $\mu_i$} --(0,1) node[midway,nail]{};
        } \hspace{2mm}= \hspace{2mm} h q^{2 \mu_i}
        \]
        \item Multiplication is defined in the same way as in Definition \ref{def:dgKLRWalgebra}
        \item Any violated diagram is 0 (any diagram with a black strand at the far left)
        \item The algebra is subjected to all the relations in Definition \ref{def:dgKLRWalgebra} (viewed as double string diagrams) as well as the new relations, 
        \begin{enumerate}
            \item Black strand relations,
            \[
           \tikzdiag[scale=1]{
            \draw[lower->] (1,0) arc (0:180:0.5);
            \draw[upper<-] (0.5,0) ..controls (0.5,0.5) and (1,0.5) .. (1,1); 
            } \hspace{2mm} = \hspace{2mm} \tikzdiag[scale=1]{
            \draw[lower->] (1,0) arc (0:180:0.5);
            \draw[upper<-] (0.5,0) ..controls (0.5,0.5) and (0,0.5) .. (0,1); 
            } \hspace{30mm}
            \tikzdiag[scale=1]{
            \draw[lower->] (1,0) arc (0:180:0.5);
            \draw[upper->] (0.5,0) ..controls (0.5,0.5) and (1,0.5) .. (1,1); 
            } \hspace{2mm} = \hspace{2mm} \tikzdiag[scale=1]{
            \draw[lower->] (1,0) arc (0:180:0.5);
            \draw[upper->] (0.5,0) ..controls (0.5,0.5) and (0,0.5) .. (0,1); 
            } \hspace{5mm}\text{and their horizontal and vertical mirrors}
            \]
            \[
            \tikzdiag[scale=1.5]{
            \draw (1,0)..controls (1,0.5) and (0,0.5) .. (0,1);
            \draw (0,0) ..controls (0,0.5) and (1,.5) .. (1,1);
            \draw[mid->] (1,1) arc (0:180:0.5);
            } \hspace{2mm}=\hspace{2mm} - \sum_{a+b = -1} \hspace{2mm} \tikzdiag[scale=1.5]{
            \draw[lower->] (1,1) arc (0:360:0.4) node[tikzdot, pos=0.5]{} node[left, pos=0.5]{$a$};
            \draw[lower<-] (1,0) arc (0:180:0.5) node[tikzdot, pos=0.75]{} node[left, pos=0.75]{$b$};
            }
            \] \\ \\
            \[
            \tikzdiag[scale=1.5]{
            \draw (1,0)..controls (1,0.5) and (0,0.5) .. (0,1);
            \draw (0,0) ..controls (0,0.5) and (1,.5) .. (1,1);
            \draw[mid<-] (1,1) arc (0:180:0.5);
            } \hspace{2mm}=\hspace{2mm}  \sum_{a+b = -1} \hspace{2mm} \tikzdiag[scale=1.5]{
            \draw[lower<-] (1,1) arc (0:360:0.4) node[tikzdot, pos=0.5]{} node[left, pos=0.5]{$a$};
            \draw[lower->] (1,0) arc (0:180:0.5) node[tikzdot, pos=0.75]{} node[left, pos=0.75]{$b$};
            }
            \] \\ \\
            \[
            \tikzdiag[scale=1.5]{
            \draw[->] (1,0)..controls (1,0.5) and (0,0.5) .. (0,1);
            \draw (0,0) ..controls (0,0.5) and (1,.5) .. (1,1);
            \draw[<-] (1,1) ..controls (1,1.5) and (0,1.5) ..(0,2);
            \draw (0,1) ..controls (0,1.5) and (1,1.5) .. (1,2);
            } \hspace{2mm}=\hspace{2mm} - \hspace{2mm}\tikzdiag[scale=1.5]{
            \draw[mid<-](0,0)--(0,2);
            \draw[mid->](1,0)--(1,2);
            }  \hspace{2mm}+\hspace{2mm}\sum_{a+b+c = -2} \hspace{2mm} \tikzdiag[scale=1.5]{
            \draw[upper<-] (1,2) arc (0:-180:0.5) node[pos=0.25, tikzdot]{} node[pos=0.25,right ]{$c$};
            \draw[lower<-] (2,1) arc (0:360:0.4) node[tikzdot, pos=0.5]{} node[left, pos=0.5]{$a$};
            \draw[lower->] (1,0) arc (0:180:0.5) node[tikzdot, pos=0.75]{} node[left, pos=0.75]{$b$};
            }
            \] \\ \\
            \[
            \tikzdiag[scale=1.5,yscale=-1]{
            \draw[->] (1,0)..controls (1,0.5) and (0,0.5) .. (0,1);
            \draw (0,0) ..controls (0,0.5) and (1,.5) .. (1,1);
            \draw[<-] (1,1) ..controls (1,1.5) and (0,1.5) ..(0,2);
            \draw (0,1) ..controls (0,1.5) and (1,1.5) .. (1,2);
            } \hspace{2mm}=\hspace{2mm} - \hspace{2mm}\tikzdiag[scale=1.5,yscale=-1]{
            \draw[mid<-](0,0)--(0,2);
            \draw[mid->](1,0)--(1,2);
            }  \hspace{2mm}+\hspace{2mm}\sum_{a+b+c = -2} \hspace{2mm} \tikzdiag[scale=1.5,yscale=-1]{
            \draw[upper<-] (1,2) arc (0:-180:0.5) node[pos=0.25, tikzdot]{} node[pos=0.25,right ]{$c$};
            \draw[lower<-] (2,1) arc (0:360:0.4) node[tikzdot, pos=0.5]{} node[left, pos=0.5]{$a$};
            \draw[lower->] (1,0) arc (0:180:0.5) node[tikzdot, pos=0.75]{} node[left, pos=0.75]{$b$};
            }
            \] 
            We also allow for bubbles (closed circles of black strands) to have a negative number of black dots and declare that any bubble of degree $1$ is equal to $1$ and if the degree is of the form $x^0 q^{-n}$ for $n$ a positive integer, then it is 0. 

            \item Colored strand relations (and their horizontal and vertical mirrors), 
            \be
           \tikzdiag[scale=2]{
            \draw[lower->] (1,0) arc (0:180:0.5);
            \draw[pstdhl,upper<-] (0.5,0) ..controls (0.5,0.5) and (1,0.5) .. (1,1); 
            } \hspace{2mm} = \hspace{2mm} \tikzdiag[scale=2]{
            \draw[lower->] (1,0) arc (0:180:0.5);
            \draw[pstdhl,upper<-] (0.5,0) ..controls (0.5,0.5) and (0,0.5) .. (0,1); } \ee \\
            
            \be 
            \tikzdiag[scale=2]{
            \draw[->] (0,0)--(1,1);
            \draw[->] (1,0)--(0,1);
            \draw[pstdhl,mid->] (0.5,0) ..controls (0,0.25) and (0,.75) .. (0.5,1); 
            } \hspace{2mm}= \hspace{2mm} \tikzdiag[scale=2]{
            \draw[->] (0,0)--(1,1);
            \draw[->] (1,0)--(0,1);
            \draw[pstdhl,mid->] (0.5,0) ..controls (1,0.25) and (1,.75) .. (0.5,1); 
            }
            \ee \\
            
            \be
            \tikzdiag[scale=2]{
            \draw[->] (0,0)--(1,1);
            \draw[pstdhl,->] (1,0)--(0,1);
            \draw[mid->] (0.5,0) ..controls (0,0.25) and (0,.75) .. (0.5,1); 
            } \hspace{2mm} = \hspace{2mm} \tikzdiag[scale=2]{
            \draw[->] (0,0)--(1,1);
            \draw[pstdhl,->] (1,0)--(0,1);
            \draw[mid->] (0.5,0) ..controls (1,0.25) and (1,.75) .. (0.5,1); 
            }
            \ee \\

            \be
            \tikzdiag[scale=2]{
            \draw[->] (0,0)--(1,1) node[tikzdot, pos=0.2]{};
            \draw[pstdhl,->] (1,0)--(0,1); 
            } \hspace{2mm} = \hspace{2mm} \tikzdiag[scale=2]{
            \draw[->] (0,0)--(1,1) node[tikzdot, pos=0.8]{};
            \draw[pstdhl,->] (1,0)--(0,1); 
            }
            \ee \\

            \be 
            \tikzdiag[scale=2,scale =1]{
      \draw[pstdhl,mid->] (0,0) ..controls (1,0.25) and (1,0.75) .. (0,1);
      \draw[mid->] (1,0) ..controls (0,0.25) and (0,0.75) .. (1,1);
      } \hspace{2mm}= \hspace{2mm} \tikzdiag[scale=2,scale =1]{
      \draw[pstdhl,mid->] (0,0) -- (0,1);
      \draw[mid->] (1,0) -- (1,1) ;
      } \ee \\
      \be 
      \tikzdiag[scale=2,xscale =0.5]{
      \draw[pstdhl,mid->] (0,0) node[below=1mm]{\small $\mu_1$} -- (0,1);
      \draw[mid->] (1,0) -- (1,1) node[pos=0.5,right=5mm]{$\cdots$};
      } \hspace{2mm} = \hspace{2mm} 0
      \ee \\ 
      
      \item Red strand relations (and mirrors across horizontal axis), 
      
      \be 
      \tikzdiag[scale=2,xscale = 1]{
      \draw[stdhl,mid->] (0,0) node[below=1mm]{\small $\mu_1$} -- (0,1);
      \draw[mid<-] (1,0.5) arc (0:360:0.3) node[pos=0.25,tikzdot]{} node[pos=0.25,above=1mm]{\small $n$};
      } \hspace{2mm} = \hspace{2mm} 0 \hspace{10mm} n\geq -\mu_i
      \ee \\
      \item Blue strand relations (and mirrors across horizontal axis), 
      \be 
      \tikzdiag[scale=2,xscale = 1]{
      \draw[vstdhl,mid->] (0,0) node[below=1mm]{\small $\mu_1$} -- (0,1);
      \draw[mid<-] (1,0.5) arc (0:360:0.3) node[pos=0.25,tikzdot]{} node[pos=0.25,above=1mm]{\small $n$};
      } \hspace{2mm} = \hspace{2mm} 0 \hspace{10mm} n\in \mathbb{Z}
      \ee
      \item Finally, we can give $D\Tb^{\underline{\mu}}$ a dg-algebra structure through the differential \eqref{eq:differential}, just as in Definition \ref{def:dgKLRWalgebra}.

        \end{enumerate}
\end{itemize}
\end{defn}

Just as in the dg-KLRW algebra, the projective modules over this algebra may be labeled by the idempotents. In an attempt to maintain a shred of notational consistency, we will label these by $1_\rho$, where now $\rho=(b_1,b_2,...,b_\ell)$ is a tuple of words $b_i \in M_{\pm}$ where $M_{\pm}$ is the set of all words in 2 generators $+$ and $-$ (``$0$'' will denote the empty word). For instance, if $\underline{\mu}=(\lambda,\lambda)$ and $\rho = (--+,-+)$, then, 
\[1_\rho \hspace{2mm} = \hspace{2mm}
\tikzdiag[xscale=1.5,yscale=0.75]{
\draw[vstdhl] (0,0)--(0,1);
\draw[mid<-] (0.25,0)-- (.25,1);
\draw[mid<-] (.5,0) --(.5,1);
\draw[mid->] (.75,0)--(.75,1);
\draw[vstdhl] (1,0)--(1,1);
\draw[mid<-] (1.25,0)--(1.25,1);
\draw[mid->] (1.5,0) --(1.5,1);

}
\]

Now, let us define a cap bimodule on $D\Tb^{\underline{\mu}}$. We will focus on the case $\underline{\mu} = (\mu_i^*,\mu_i)$ in this section. 
\begin{defn}
    The cap bimodule $D\Kb^\kk_{(\mu_i^*,\mu_i)}$ for the double dg-KLRW algebra $D\Tb^{(\mu_i^*,\mu_i)}$ is a $\kk-D\Tb^{(\mu_i^*,\mu_i)}$ bimodule generated by the diagram, 
    \[
    \tikzdiag[scale=1.5]{
    \draw[pstdhl,mid->] (0,1) node[below=2pt]{\small $\mu_i$} arc (0:180:0.5) node[below]{\small $\mu_i^*$};
    }
    \]
    subject to all the relations of Definition \ref{defn:double_dg_KLRW} as well as the local relations, 
    \[
    \tikzdiag[scale=1.5]{
    \draw[pstdhl,mid->] (1,0) arc (0:180:0.5);
    \draw[upper<-] (1.2,0) arc (0:180:0.3) node[pos=0.8,left=2pt]{$...$};
    } \hspace{2mm} = \hspace{2mm} 0 
    \]
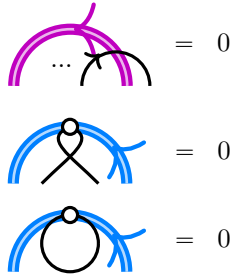
    \[
    \tikzdiag[scale=1.5]{
    \draw[vstdhl,lower->] (1,0) arc (0:180:0.5);
    \draw (0.75,0).. controls (0.5,.25) and (.25,.4) .. (0.5,.5);
    \draw (0.25,0).. controls (0.5,.25) and (.75,.4) .. (0.5,.5) node[pos=1,nail]{};
    } \hspace{2mm} = \hspace{2mm} 0 
    \]
    \[
    \tikzdiag[scale=1.5]{
    \draw[vstdhl,lower->] (1,0) arc (0:180:0.5);
    \draw (0.75,.25) arc (0:360:.25) node[pos=.25,nail]{};
    } \hspace{2mm} = \hspace{2mm} 0 
    \]
    The cap functor is defined as before as $D\Kbb^\kk_{(\mu_i^*,\mu_i)}(-) := D\Kb^\kk_{(\mu_i^*,\mu_i)} \bigotimes^L_{D\Tb} -$. When the colored strand labelings are understood, we will often write just $D\Kb = D\Kb^\kk_{(\mu_i^*,\mu_i)}$ for legibility.
\end{defn}

Let us study this cap in some simple colored strand configurations, where there are no Verma modules. 
\begin{ex}{$\mathbf{\mu_i=1}$}
    Set $\underline{\mu} = (1^*,1)$. As usual, in order to understand how these caps and cups behave, let us try applying them to projectives. We find vector spaces $D\Kb \otimes_{D\Tb} \Pb_\rho $, which are only non-zero for 2 projectives, $\Pb_{0,0}$ and $\Pb_{+-,0}$. Using the definitions and relations above, it is easy to see that each vectors space is 1-dimensional with grading $1$,
    \[
    D\Kb \cdot 1_{0,0} \hspace{2mm} \cong \hspace{2mm} \kk \hspace{2mm}\bigg\langle  \tikzdiag[scale=1]{
    \draw[stdhl] (0,1) arc (0:180:0.5);
    }\bigg\rangle 
    \]
    \[
    D\Kb \cdot 1_{+-,0} \hspace{2mm} \cong \hspace{2mm} \kk \hspace{2mm}\bigg\langle  \tikzdiag[scale=1]{
    \draw[stdhl] (0,1) arc (0:180:0.5);
    \draw[mid<-] (-0.2,1) arc (0:180:0.3);
    }\bigg\rangle 
    \]
    This means that $D\Kb \otimes^L_{D\Tb} \Sb_{0,0} \cong \kk $ and $D\Kb \otimes^L_{D\Tb} \Sb_{+,-} \cong q\cdot \kk[1] $. Recall that Proposition \ref{prop:decat_cap_formula} gives, $\hat{C}(v^*_{1,0}\otimes v_{1,0} ) = 1$ and $\hat{C}(v^*_{1,1}\otimes v_{1,1} ) = -q$, so the cap bimodule $D\Kb$ exactly categorifies the expression for the cap in the previous section. \\
    We can define the cup bimodule by reflecting across the horizontal axis as before, $D\Cb := \dot{D\Kb}$ and study its resolution by projectives. It is straightforward to derive this chain complex and the result is given by, 
    \[
    \begin{tikzcd}
        q^2\hspace{2mm}\tikzdiagh[xscale=0.5]{0}{
\draw[stdhl](0,0)--(0,0.5);
\draw[stdhl](2,0)--(2,0.5);
\draw[mid->] (0.66,0)--(0.66,0.5);
\draw[mid<-] (1.33,0) --(1.33,0.5);
\filldraw [fill=white, draw=black] (-0.5,0.5) rectangle (2.5,1) node[midway] { $\Tb$};
} 
\ar{rr}
\ar[dash]{rr}{
\tikzdiag[xscale=0.75,yscale=0.5]{
\draw[stdhl] (0,0) --(0,1);
\draw[stdhl] (1,0) --(1,1);
\draw[mid->] (0.33,0)--(0.33,1);
\draw[mid<-] (1.5,0)..controls (1.5,.5) and (0.66,.5) .. (0.66,1);
} 
}
&&
q\hspace{2mm}\tikzdiagh[xscale=0.5]{0}{
\draw[stdhl](0,0)--(0,0.5);
\draw[stdhl](1.33,0)--(1.33,0.5);
\draw[mid->] (0.66,0)--(0.66,0.5);
\draw[mid<-] (2,0) --(2,0.5);
\filldraw [fill=white, draw=black] (-0.5,0.5) rectangle (2.5,1) node[midway] { $\Tb$};
}
\ar{rr}
\ar[dash]{rr}{
\tikzdiag[xscale=0.75,yscale=0.5]{
\draw[stdhl] (0,0) --(0,1);
\draw[stdhl] (1,0) --(1,1);
\draw[upper->] (1.5,1) arc (0:-180:0.5);
} 
}
&&
\tikzdiagh[xscale=0.5]{0}{
\draw[stdhl](0,0)--(0,0.5);
\draw[stdhl](1.33,0)--(1.33,0.5);
\filldraw [fill=white, draw=black] (-0.5,0.5) rectangle (2.5,1) node[midway] { $\Tb$};
}
\ar{rr}
&&
D\Cb
    \end{tikzcd}
    \]
Of course, using this we can compute the homology of the unknot and we find precisely the same result as Webster,
$$D\Kb \otimes^L_{D\Tb }D\Cb \cong q\hspace{1mm}\kk[1] \oplus q^{-1}\hspace{1mm}\kk[-1]  $$
\end{ex}

\begin{ex}{$\mu_i=2$} Now, we set $\underline{\mu}=(2^*,2)$. In this case, the relevant projectives are $\Pb_{0,0}$, $\Pb_{+-,0}$, and $\Pb_{++--,0}$. After a bit of work, one can write a basis for the cap evaluated at each of these modules and we find, 
\[
    D\Kb \cdot 1_{0,0} \hspace{2mm} \cong \hspace{2mm} \kk \hspace{2mm}\bigg\langle  \tikzdiag[scale=1]{
    \draw[stdhl] (0,1) arc (0:180:0.5);
    }\bigg\rangle 
    \]
    \[
    D\Kb \cdot 1_{+-,0} \hspace{2mm} \cong \hspace{2mm} \kk \hspace{2mm}\bigg\langle  \tikzdiag[scale=1]{
    \draw[stdhl] (0,1) arc (0:180:0.5);
    \draw[upper<-] (-0.2,1) arc (0:180:0.3);
    }\hspace{2mm},\hspace{2mm} \tikzdiag[scale=1]{
    \draw[stdhl] (0,1) arc (0:180:0.5);
    \draw[upper<-] (-0.2,1) arc (0:180:0.3) node[pos=0.25,tikzdot]{};
    } \bigg\rangle 
    \]
    \[
    D\Kb \cdot 1_{++--,0} \hspace{2mm} \cong \hspace{2mm} \kk \hspace{2mm}\bigg\langle  \tikzdiag[scale=1.5]{
    \draw[stdhl] (0,1) arc (0:180:0.5);
    \draw[upper<-] (-0.1,1) arc (0:180:0.4);
    \draw[mid<-] (-0.3,1) arc (0:180:0.2);
    }\hspace{2mm},\hspace{2mm} \tikzdiag[scale=1.5]{
    \draw[stdhl] (0,1) arc (0:180:0.5);
    \draw[lower<-] (-0.1,1) arc (0:180:0.3);
    \draw[upper<-] (-0.3,1) arc (0:180:0.3);
    }\hspace{2mm},\hspace{2mm} \tikzdiag[scale=1.5]{
    \draw[stdhl] (0,1) arc (0:180:0.5);
    \draw[lower<-] (-0.1,1) arc (0:180:0.3) node[pos=0.75,tikzdot]{};
    \draw[upper<-] (-0.3,1) arc (0:180:0.3);
    }\hspace{2mm},\hspace{2mm} \tikzdiag[scale=1.5]{
    \draw[stdhl] (0,1) arc (0:180:0.5);
    \draw[lower<-] (-0.1,1) arc (0:180:0.3) node[pos=0.75,tikzdot]{} node[pos=0.35,tikzdot]{};
    \draw[upper<-] (-0.3,1) arc (0:180:0.3);
    }
     \bigg\rangle 
    \]

Using these we find,
$$\begin{matrix}
    D\Kb \otimes^L_{D\Tb} \Sb_{0,0} \cong \kk &&&&& D\Kb \otimes^L_{D\Tb} \Sb_{+,-} \cong (q \hspace{1mm} \kk  \oplus q^3 \hspace{1mm} \kk) [1] \end{matrix}$$
    
    $$ D\Kb \otimes^L_{D\Tb} \Sb_{++,--} \cong \hspace{1mm}( \kk  \oplus q^2 \hspace{1mm} \kk \oplus  q^2 \hspace{1mm}  \kk \oplus  q^4 \hspace{1mm} \kk)[2]$$

Once again, these categorify the expressions in Proposition \ref{prop:decat_cap_formula} and the cap functor categorifies the evaluation map in the decategorified setting. We suspect that after computing a projective resolution for the cup, as in the previous example, and taking the homology of the unknot $D\Kb \otimes^L_{D\Tb }D\Cb$, we will find the same result as Webster \cite{Web} for the irreducible representation of dimension $3$. As a sanity check, the resulting homologies of the unknot in our case are inifinite-dimensional since $D\Cb$ categorifies the rational functions $e_i^{-1}$ in Proposition \ref{prop:decat_cup_formula}. 
\end{ex}

In general, for finite-dimensional representations, the quasi-isomorphism of functors in \cite{Lau}, $\mathcal{E}\mathcal{F}\boldsymbol{1}_n \cong \mathcal{F} \mathcal{E}\boldsymbol{1}_n\oplus_{[n]_q}\boldsymbol{1}_n$ for positive weight spaces and $\mathcal{F}\mathcal{E}\boldsymbol{1}_n \cong \mathcal{E} \mathcal{F}\boldsymbol{1}_n\oplus_{[-n]_q}\boldsymbol{1}_n$ for negative weight spaces ensure that the cap we've defined always categorifies the expression in Proposition \ref{prop:decat_cap_formula}. In the second part of this paper, we will study the projective resolutions of the resulting cups and their relation to the invariants defined in \cite{Web}. 

\subsubsection{Caps for Verma Modules}
We would like to now briefly tackle the primary ambition of this section, which is the categorification of the caps in the case that the strands are colored by Verma modules. So, for the rest of this section, we consider $\underline{\mu} = (\lambda^*,\lambda)$. \\ \\ 
Just as in the finite-dimensional case, projectives of the form $\Pb^{(\lambda^*,\lambda)}_{(+)^i(-)^i,0}$ for $i\in \mathbb{N}$ will be of primary interest to us. Note that these are not the only projectives upon which the cap will be non-trivial. For instance, the cap will have a non-trivial basis on $\Pb^{(\lambda^*,\lambda)}_{+-+-+-,0}$. However, $\Pb^{(\lambda^*,\lambda)}_{(+)^i(-)^i,0}$ are the only non-trivial projective modules in the chain complex for a standard of the form $\Sb^{(\lambda^*,\lambda)}_{(+)^i,(-)^i}$. We calculate the first few bases for these indempotents and they are given by, 

\[
    D\Kb \cdot 1_{0,0} \hspace{2mm} \cong \hspace{2mm} \kk \hspace{2mm}\bigg\langle  \tikzdiag[scale=2]{
    \draw[vstdhl] (0,1) node[below]{\small $\lambda$} arc (0:180:0.5) node[below]{\small $\lambda^*$};
    }\bigg\rangle 
    \]
    
\[
    D\Kb \cdot 1_{+-,0} \hspace{2mm} \cong \hspace{2mm} \kk \hspace{2mm}\bigg\langle  \tikzdiag[scale=1]{
    \draw[vstdhl] (2,0) arc (0:180:1);
    \draw[mid<-] (1.5,0) arc (0:180:.5) node[pos=.75,tikzdot]{} node[pos=.75,above]{\small $p$};
    }\hspace{2mm}, \hspace{2mm} \tikzdiag[scale=1]{
    \draw[vstdhl] (2,0) arc (0:180:1);
    \draw[upper->] (.05,.25) .. controls (.5,.25) and (1.5,1) .. (1.5,0) node[pos=.5,tikzdot]{} node[pos=.5,above]{\small $p$};
    \draw (.5,0) .. controls (.5,.25) and (.1,.25) ..(.05,.25) node[pos=1,nail]{};
    }\bigg\rangle_{p\geq 0} 
    \]

\[
D\Kb \cdot 1_{++--,0} \hspace{2mm} \cong \hspace{2mm} \kk \hspace{2mm} \left\langle \begin{Bmatrix}
    \tikzdiag[scale=1]{
    \draw[vstdhl] (2,0) arc (0:180:1);
    \draw[mid<-] (1.75,0) arc (0:180:.75) node[pos=.75,tikzdot]{};
    \draw[mid<-] (1.5,0) arc (0:180:.5) node[pos=.75,tikzdot]{};
    }\hspace{2mm}, \hspace{2mm} \tikzdiag[scale=1]{
    \draw[vstdhl] (2,0) arc (0:180:1);
    \draw[mid<-] (1.75,0) arc (0:90:.75) node[pos=.75,tikzdot]{};
    \draw[mid<-] (1.5,0) arc (0:90:.5) node[pos=.75,tikzdot]{};
    \draw (0.5,0) ..controls (.5,.25) and (.5,.75) .. (1,.75);
    \draw (.25,0) ..controls (.25,.25) and (.25,.5) .. (1,.5);
    }\\
    \tikzdiag[scale=1]{
    \draw[vstdhl] (2,0) arc (0:180:1);
    \draw[upper->] (.05,.25) .. controls (.5,.25) and (1.5,1) .. (1.5,0) node[pos=.5,tikzdot]{};
    \draw (.5,0) .. controls (.5,.25) and (.1,.25) ..(.05,.25) node[pos=1,nail]{};
    \draw[lower<-] (1.75,0) arc (0:180:.75) node[pos=.5,tikzdot]{};
    }\hspace{2mm},\hspace{2mm} \tikzdiag[scale=1]{
    \draw[vstdhl] (2,0) arc (0:180:1);
    \draw[upper->] (.05,.25) .. controls (.4,.25) and (1.5,1.3) .. (1.75,0) node[pos=.9,tikzdot]{};
    \draw (.25,0) .. controls (.25,.25) and (.1,.25) ..(.05,.25) node[pos=1,nail]{};
    \draw[lower<-] (1.5,0) arc (0:180:.5) node[pos=.1,tikzdot]{};
    }\hspace{2mm},\hspace{2mm}\tikzdiag[scale=1]{
    \draw[vstdhl] (2,0) arc (0:180:1);
    \draw[mid->] (.05,.25) .. controls (.3,.25) and (1.5,1.3) .. (1.5,0) node[pos=.2,tikzdot]{};
    \draw (.5,0) .. controls (.5,.25) and (.1,.25) ..(.05,.25) node[pos=1,nail]{};
    \draw[lower<-] (1.75,0) arc (0:180:.5) node[pos=.8,tikzdot]{};
    }\hspace{2mm},\hspace{2mm} \tikzdiag[scale=1]{
    \draw[vstdhl] (2,0) arc (0:180:1);
    \draw[mid->] (.15,.4) .. controls (.3,.4) and (1.5,1.3) .. (1.5,0) node[pos=.9,tikzdot]{};
    \draw (.5,0) .. controls (.5,.25) and (.1,.25) ..(.15,.4) node[pos=1,nail]{};
    \draw[lower<-] (1.25,0) arc (0:180:.5) node[pos=.1,tikzdot]{};
    } \\
    \tikzdiag[scale=1]{
    \draw[vstdhl] (2,0) arc (0:180:1);
    \draw[upper->] (.05,.25) .. controls (.4,.25) and (1.5,1.3) .. (1.75,0) node[pos=.9,tikzdot]{};
    \draw (.25,0) .. controls (.25,.25) and (.1,.25) ..(.05,.25) node[pos=1,nail]{};
    \draw[lower->] (.25,.6) .. controls (.4,.6) and (1,1.3) .. (1.5,0) node[pos=.9,tikzdot]{};
    \draw (.5,0) .. controls (.5,.5) and (.25,.5) ..(.25,.6) node[pos=1,nail]{};
    }\hspace{2mm}, \hspace{2mm} \tikzdiag[scale=1]{
    \draw[vstdhl] (2,0) arc (0:180:1);
    \draw[upper->] (.05,.25) .. controls (.4,.25) and (1.,1.) .. (1.5,0) node[pos=.9,tikzdot]{};
    \draw (.25,0) .. controls (.25,.25) and (.1,.25) ..(.05,.25) node[pos=1,nail]{};
    \draw[lower->] (.25,.6) .. controls (.4,.6) and (1.5,1.3) .. (1.75,0) node[pos=.9,tikzdot]{};
    \draw (.5,0) .. controls (.5,.5) and (.25,.5) ..(.25,.6) node[pos=1,nail]{};
    }
\end{Bmatrix} \right \rangle_{p_1,p_2\geq 0}
\]

In the last expression, by each unlabeled black dot , we mean that it is labeled by $p_1$ or $p_2$, before considering all configurations of black dots on the 2 strands (running over $p_1,p_2\geq 0$). This allows us to compute the graded dimensions of the cap applied to standards, and we find, 

$$gdim_{x,q} \big( D\Kb \otimes ^L_{D\Tb}\Sb^{\lambda^*,\lambda}_{0,0}\big) =1$$
$$gdim_{x,q} \big( D\Kb \otimes ^L_{D\Tb}\Sb^{\lambda^*,\lambda}_{+,-}\big) = \frac{1}{1-q^2}(qh +qh^2 x^2) $$
$$gdim_{x,q} \big( D\Kb \otimes ^L_{D\Tb}\Sb^{\lambda^*,\lambda}_{++,--}\big) = \frac{1}{(1-q^2)^2}\big (h^2 +q^2h^2 +2h^3x^2+\frac{h^3x^2}{q^2}+q^2h^3x^2+h^4x^4+\frac{h^4x^4}{q^2} \big ) $$

In general, one can workout the bases of the caps for the indempotent $1_{(+)^i (-)^i,0}$ to show, 
$$gdim_{x,q} \big( D\Kb \otimes ^L_{D\Tb}\Sb^{\lambda^*,\lambda}_{(+)^i,(-)^i}\big) = \frac{h^i x^i q^{2i-i^2}}{(1-q^2)^i} [i]_q!\prod^i_{k=1}(hxq^{-k+1}+x^{-1}q^{k-1})$$
Comparing with Proposition \ref{prop:decat_cap_formula}, we easily conclude, 
$$gdim_{x,q} \big( D\Kb \otimes ^L_{D\Tb}\Sb^{\lambda^*,\lambda}_{(+)^i,(-)^i}\big)|_{h=-1} = \hat{C}(v^*_{\lambda,i} \otimes v_{\lambda,i})$$

Thus, we find that the cap bimodule for Verma modules categorifies the cap described in the previous section. Moreover, they are compatible with the differentials \eqref{eq:differential} in the sense that for finite-dimensional representations, we should recover the results from Webster's caps (for the case of the fundamental representation, at least).  \\

This is a very surprising and interesting result which essentially establishes the road for defining homological quantum knot invariants from Verma modules (the categorification of the Gukov-Manolescu series). That is, we've seemingly surpassed the hurdle of describing the duals of Verma modules by simply including the `doubled' part of the tensor product algebra of \cite{Web}. This allowed us to describe the categorified caps and cups as functors which commute with 1-morphisms of categorified $U_q(\mathfrak{sl}_2)$ (\cite{Lau}) and which reduce to the appropiate objects from Section \ref{sec:decat_cups_caps} in the Grothendieck group. This provides a very useful and explicit workaround to the issue of `taking the large color limit' of Webster's constructions for finite dimensional representations, which are not conducive to the generalization to Verma modules.  \\

We note here that Definition \ref{defn:double_dg_KLRW} should include some additional nail relations, which were not written in this paper. For instance, bubbles with nails must be appropriately dealt with using local relations. This will be addressed in the upcoming sequel to this paper, which will build on this section to a much greater extent. \\ \\
One might worry that in this doubled category, braiding functors might take a more complicated and less explicit form. Although this is true in general, our more noble goal will be to define homological knot invariants for braid negative (or positive) knots. If one then uses braid diagrams to define knot homologies as in Section \ref{sec:khov_hom_braids}, it becomes apparent that all braiding will take place between upward-moving, colored strands. And though we might have to deal with downward moving black strands in this setting, this is a non-issue in light of the categorified $[E,F]$ relation, which allows us to rewrite any projective with downward-moving black strands in the upward-moving colored section into sums of projectives encompassed by Proposition \ref{prop:Rresolution}. In other words, the doubled category is just as computable as the rest of this paper. \\ \\
As such, immediately upon the construction of the categorified cups/caps, we will be able to extend the construction in Section \ref{sec:khov_hom_braids} to include Khovanov Homology in knot complements. As mentioned, we do this in the follow-up paper, where we will also explicitly compute the homology in the case of $3_1$ and $5_1$. 

\section*{Acknowledgments}
We would like to thank Raphael Rouquier, Mrunmay Jagadale, Sergei Gukov, Pedro Vaz, and Ben Webster for helpful discussions. We also especially thank Gregoire Naisse for insightful conversations, comments on this draft, and for sharing an early draft of his work on the braiding bimodules of dg-KLRW algebras. Finally, we thank Ciprian Manolescu and Mikhail Khovanov for helpful comments on this draft. 

\bibliographystyle{unsrt}
\bibliography{bibliography}

\end{document}